\documentclass[manuscript,screen,12pt]{acmart}

\usepackage{graphicx} 
\usepackage{amsfonts}
\usepackage[english]{babel}
\usepackage{hyperref}
\usepackage{url}
\usepackage{bussproofs}
\usepackage{proof}
\usepackage{cancel}
\usepackage{eqnarray,amsmath}
\usepackage{paralist}

\newcommand{\bit}{\begin{itemize}}
	\newcommand{\eit}{\end{itemize}}
\newcommand{\ben}{\begin{enumerate}}
	\newcommand{\een}{\end{enumerate}}
\newcommand{\bds}{\begin{description}}
	\newcommand{\eds}{\end{description}}

\newcounter{romc}

\newcounter{alphc}

\newcommand{\blr}{\begin{list}{~(\roman{romc})~} {\usecounter{romc}
			\setlength{\topsep}{0pt} \setlength{\itemsep}{0pt}}}
	\newcommand{\elr}{\end{list}}
\newcommand{\bla}{\begin{list}{~(\alph{alphc})~} {\usecounter{alphc}
			\setlength{\topsep}{0pt} \setlength{\itemsep}{0pt}}}
	\newcommand{\ela}{\end{list}}
\newcommand{\boxg}{{\Box_{\mathfrak o}}}
\newcommand{\oboxg}{{\overline{\Box}_{\mathfrak o}}}
\newcommand{\odiamondg}{{\overline{\Diamond}_{\mathfrak o}}}
\newcommand{\odiamondm}{{\overline{\Diamond}_{\mathfrak p}}}
\newcommand{\gboxg}[1]{{\Box_{\mathfrak o}^{#1}}}
\newcommand{\ogboxg}[1]{{\overline{\Box}_{\mathfrak o}^{#1}}}
\newcommand{\boxm}{{\Box_{\mathfrak p}}}
\newcommand{\oboxm}{{\overline{\Box}_{\mathfrak p}}}
\newcommand{\gboxm}[1]{{\Box_{\mathfrak p}^{#1}}}
\newcommand{\ogboxm}[1]{{\overline{\Box}_{\mathfrak p}^{#1}}}
\newcommand{\diamondg}{{\Diamond_{\mathfrak o}}}
\newcommand{\gdiamondg}[1]{{\Diamond_{\mathfrak o}^{#1}}}
\newcommand{\ogdiamondg}[1]{{\overline{\Diamond}_{\mathfrak o}^{#1}}}
\newcommand{\diamondm}{{\Diamond_{\mathfrak p}}}
\newcommand{\gdiamondm}[1]{{\Diamond_{\mathfrak p}^{#1}}}
\newcommand{\ogdiamondm}[1]{{\overline{\Diamond}_{\mathfrak p}^{#1}}}
\newcommand{\boxng}{{\boxminus_{\mathfrak o}}}
\newcommand{\oboxng}{{\overline{\boxminus}_{\mathfrak o}}}
\newcommand{\gboxng}[1]{{\boxminus_{\mathfrak o}^{#1}}}
\newcommand{\ogboxng}[1]{{\overline{\boxminus}_{\mathfrak o}^{#1}}}
\newcommand{\boxnm}{{\boxminus_{\mathfrak p}}}
\newcommand{\gboxnm}[1]{{\boxminus_{\mathfrak p}^{#1}}}
\newcommand{\oboxnm}{{\overline{\boxminus}_{\mathfrak p}}}
\newcommand{\ogboxnm}[1]{{\overline{\boxminus}_{\mathfrak p}^{#1}}}
\newcommand{\badstart}[0]{\ \\[-.2in]}
\newcommand\sbullet[1][.3]{\mathbin{\vcenter{\hbox{\scalebox{#1}{$\bullet$}}}}}

\newtheorem{re}{Remark}

\begin{document}

\title[Reasoning with Formal Contexts]{On the Logical and Algebraic Aspects of Reasoning with Formal Contexts}

\author{Prosenjit Howlader}
\authornote{This is an extended version of \cite{HowladerL23}. Both authors contributed equally to this research.}
\email{prosen@mail.iis.sinica.edu.tw}
\orcid{1234-5678-9012}
\author{Churn-Jung Liau}
\authornotemark[1]
\email{liaucj@iis.sinica.edu.tw}
\orcid{0000-0001-6842-9637}
\affiliation{%
  \institution{Institute of Information Science,Academia Sinica}
  \city{Taipei}
   \postcode{115}
   \country{Taiwan}
}







\renewcommand{\shortauthors}{Howlader and  Liau.}

\begin{abstract}
A formal context consists of objects, properties, and the incidence relation between them. Various notions of concepts defined with respect to formal contexts and their associated algebraic structures have been studied extensively, including formal concepts in formal concept analysis (FCA), rough concepts arising from rough set theory (RST), and semiconcepts and protoconcepts for dealing with negation. While all these kinds of concepts are associated with lattices, semiconcepts and protoconcepts additionally yield an ordered algebraic structure, called double Boolean algebras. As the name suggests, a double Boolean algebra contains two underlying Boolean algebras.

In this paper, we investigate logical and algebraic aspects of the representation and reasoning about different concepts  with respect to formal contexts. We first review our previous work on two-sorted modal logic systems \textbf{KB} and \textbf{KF} for the representation and reasoning of rough concepts and  formal concepts, respectively. 
Then, in order to represent and reason about both formal and rough concepts in a single framework, these two logics are unified into a two-sorted Boolean modal logic \textbf{BM}, in which semiconcepts and protoconcepts are also expressible. Based on the logical representation of semiconcepts and protoconcepts, we prove the characterization of double Boolean algebras in terms of their underlying Boolean algebras. Finally, we also discuss the possibilities of extending our logical systems for the representation and reasoning of more fine-grained quantitative information in formal contexts.
\end{abstract}


\keywords{Modal logic, Boolean algebra, Double Boolean algbra, Formal concept analysis, Rough set theory.}


\maketitle

\section{Introduction}
\label{intro}
In lattice theory, a {\it polarity} \cite{birkhoff1940lattice} is a triplet $\mathbb{K}:=(G, M, I)$ where $G$ and $M$ are  sets and $I\subseteq G\times M$ is a binary relation between them. For $g\in G $ and $m\in M$, we usually use $Igm$ or $gIm$ to denote $(g,m)\in I$. A given polarity induces two operators $+: (\mathcal{P}(G), \subseteq)\rightarrow (\mathcal{P}(M), \supseteq) $ and $-: (\mathcal{P}(M), \supseteq)\rightarrow (\mathcal{P}(G), \subseteq)$, where for all $A\in \mathcal{P}(G)$ and $B\in \mathcal{P}(M)$:
\[A ^{+}= \{ m \in  M \mid \mbox{ for all }  g \in  A ~~  g I m \},\]
\[B ^{-}= \{ g \in  G \mid \mbox{ for all }  m \in  B ~~  g I m \}.\]
In  formal concept analysis (FCA)~\cite{FCA}, a polarity is called a {\it formal context} (or simply context). A {\it (formal) concept} is a pair of sets $(A, B)$ such that $A^{+}=B$ and $A=B^{-}$, where $A$ and $B$ are called its {\it extent } and {\it intent}, respectively. The set of all concepts is denoted by $\mathbf{B}(\mathbb{K})$ and forms  a complete lattice $\underline{\mathbf{B}}(\mathbb{K})$. Although the conjunction and disjunction of concepts are defined in such lattices, there is no natural way to define the negation of concepts as simply complementing the extent or intent of a concept may not yield another concept. For example, $(\overline{A},\overline{A}^+)$ may be not a formal concept even though $A$ is the extent of a formal concept. To deal with negated concepts, the notions of {\it semiconcepts} and {\it protoconcepts} are introduced in \cite{wille} by relaxing the Galois stable conditions imposed on formal concepts. Then, in accordance with the definition,  both $(\overline{A},\overline{A}^+)$  and $(\overline{B}^-,\overline{B})$  are semiconcepts for $A\subseteq G$ and $B\subseteq M$. Algebraic studies of these notions led to the definition of double Boolean algebras (dBa) and pure double Boolean algebras \cite{wille}. These structures have been studied extensively \cite{wille,vormbrock2005semiconcept,BALBIANI2012260,MR4566932,howlader2018algebras, howlader2020}. A Stone-type representation theorem for  fully contextual and pure dBas is given in \cite{howlader2020}. Representation theorems for arbitrary  dBas, however, remain open. In this article, we address this problem using Boolean algebras.

D\"{u}ntsch et al. \cite{duntsch2002modal} defined sufficiency, dual sufficiency, possibility and necessity operators based on a context. For a context  $\mathbb{K}:=(G,M,I)$, $g\in G$, and $m\in M$, $I_{g\sbullet}:= \{m\in M\mid Igm\}$ and $I_{\sbullet m}:=\{g\in G\mid Igm\}$ are the {\it right-neighborhood} and {\it left-neighborhood} of $g$ and $m$, respectively. For $A\subseteq G$, and $B\subseteq M$, the pairs of dual approximation operators are defined as:

$(\mbox{possibility})~B_{I}^{\diamondm}:=\{g\in G:I_{g\sbullet}\cap B\neq\emptyset\},~~~~ ~~~~~(\mbox{necessity})~B_{I}^{\boxm}:=\{g\in G:I_{g\sbullet}\subseteq B\}$.

$(\mbox{possibility})~A_{I}^{\diamondg}:=\{m\in M:I_{\sbullet m}\cap A\neq\emptyset\},~~~ (\mbox{necessity})~A_{I}^{\boxg}:=\{m\in M:I_{\sbullet m}\subseteq A\}$.

If there is no confusion about the relation involved, we shall omit the subscript and denote $B_{I}^{\diamondm}$ by  $B^{\diamondm}$, $B_{I}^{\boxm}$ by $B^{\boxm}$ and  similarly for the case of $A$.

The necessity and possibility operators correspond to approximation operators in rough set theory (RST) \cite{pawlak1982rough}. Based on these operators, D\"{u}ntsch et al. \cite{duntsch2002modal} and Yao \cite{yao2004concept}  introduced property oriented concepts and object oriented concepts, respectively.
A pair $(A,B)$ is a  {\it property oriented concept}  of  $\mathbb{K}$ iff $ A^{\diamondg}=B$ and $B^{\boxm}=A$, and it  is an {\it object oriented concept}  of $\mathbb{K}$ iff $ A^{\boxg}=B$ and $B^{\diamondm}=A$. These two kinds of concepts are also called rough concepts. 

The operators used in formal and rough concepts correspond to modalities used in modal logic  \cite{Gargov1987,blackburn2002moda}. In particular, the operator used in FCA is the window (sufficiency)  modality~\cite{Gargov1987} and those used in RST are box (necessity) and diamond (possibility) modalities \cite{blackburn2002moda}. Furthermore, a context is a two-sorted structure consisting of object  and  property universes. 

Considering these facts, a significant amount of work has been done on relational semantics of logics interpreted in context-based models. In other words, for these kinds of investigations, contexts are semantic frames of logical formulas. Gehrke \cite{Gehrke} presented a two-sorted  relational semantics for the implication-fusion fragment of several sub-structural logics, where a reduced and separating frame is defined as a reduced and separating context \cite{Gehrke} with
additional relations. The two-sorted  relational semantics for non-distributive modal logic was studied by Conradie et al.  \cite{CWFSPAPMTAWN, Conradie2017167,CONRADIE2019923,CONRADIE2021371,nlsm}, Greco et al. \cite{Greco} and  Hartonas \cite{HCgame, CH} independently.   In \cite{CWFSPAPMTAWN}, an intuitive, epistemic interpretation of RS-frames for modal logic, in terms of categorization
systems and agents are proposed. A sound and complete epistemic logic of  categories and  categorical perception of agents is studied in \cite{Conradie2017167}. An LE-logic is proposed in  \cite{Greco} as a non-distributive modal logic for lattice expansion.  In \cite{CONRADIE2021371}, a two-sorted relational semantics based on context for LE-logic is presented from the perspective of FCA and RST. For FCA, modal operators are treated as window modality and formulas are interpreted as formal concepts of a context. While RST is taken into consideration, the modal operators are simply box and diamond modalities. 

All the logics studied in \cite{CWFSPAPMTAWN,Conradie2017167,Greco,nlsm} are Gentzen-style sequent calculus. In this paper, we study Hilbert-style axiomatic systems for formal contexts. Motivated from Boolean modal logic \cite{Gargov1987}, we developed its two-sorted version that allows us to work with box and window modalities within a single model.

To achieve the goal, we first introduce two-sorted modal logics \textbf{KB} and \textbf{KF}  to deal with box and window modalities, respectively. The two logics are then integrated into a two-sorted Boolean modal logic \textbf{BM} where both kind of modalities are interpreted in a single context. We prove the soundness and completeness of the proposed logics and show that these  logics can  represent most important notions of concepts in FCA. 

More specifically, the language of \textbf{BM} consist of 2-indexed family $\{P_{s_{1}}, P_{s_{2}}\}$ of propositional variables, logical connectives and  modal  connectives  $\boxg, \boxm, \boxng, \boxnm$. The compound $s$-sort formulas are built over atomic formulas in $P_s$  for $s=s_1,s_2$. The indexed family of formulas is denoted by $Fm(\textbf{BM}):=\{Fm(\textbf{BM})_{s_{1}}, Fm(\textbf{BM})_{s_{2}}\}$. We characterize the pairs of formulas that represent concepts, semiconcepts, and protoconcepts in a context. For instance, a pair of formulas $(\varphi, \psi)$ is a right semiconcept in a context-based model $\mathfrak{M}$ if $\mathfrak{M}\models \varphi\leftrightarrow\boxnm\psi$ and is a left semiconcept if $\mathfrak{M}\models \boxng\varphi\leftrightarrow\psi$. The set of formula pairs representing left (resp.\ right) semiconcepts is denoted  by  ${\tt SC}^{left}$ (resp.\ ${\tt SC}^{right}$). Two left (resp.\ right) semiconcepts $(\varphi_{1}, \psi_{1})$ and $(\varphi_{2}, \psi_{2})$  are called equivalent if $\varphi_{1}$ and $\varphi_{2}$ (resp.\ $\psi_{1}$ and $\psi_{2}$) are equivalent. This defines an equivalence relation $\equiv^{left}_{\tt SC}$ (resp. $\equiv^{right}_{\tt SC}$) over ${\tt SC}^{left}$ (resp. ${\tt SC}^{right}$). It is shown that the corresponding equivalence classes  ${\tt SC}^{left}/\equiv^{left}_{\tt SC}$ (resp. ${\tt SC}^{right}/\equiv^{right}_{\tt SC}$)  form Boolean algebras. For a subset $\mathfrak{D}$ of $Fm(\textbf{BM})_{s_{1}}\times Fm(\textbf{BM})_{s_{2}}$, we consider the maps   $r:\mathfrak{D}\rightarrow {\tt SC}^{left}/\equiv^{left}_{\tt SC}$ and $e:{\tt SC}^{left}/\equiv^{left}_{\tt SC}\rightarrow \mathfrak{D}$; $r^{\prime}:\mathfrak{D}\rightarrow {\tt SC}^{right}/\equiv^{right}_{\tt SC}$ and $e^{\prime}:{\tt SC}^{right}/\equiv^{right}_{\tt SC}\rightarrow \mathfrak{D}$ such that $r\circ e=id_{ {\tt SC}^{left}/\equiv^{left}_{\tt SC}}$ and $r^{\prime}\circ e^{\prime}=id_{ {\tt SC}^{right}/\equiv^{right}_{\tt SC}}$. These maps induce an algebraic structure $(\mathfrak{D}, \sqcap, \sqcup, \bar{\neg}, \lrcorner, \top_{\mathfrak{D}}, \bot_{\mathfrak{D}})$ where  for all $x, y\in \mathfrak{D}$

 \begin{center}
     $x\sqcap y:= e(r(x)\wedge r(y))$ and  $x\sqcup y:= e^{\prime}(r^{\prime}(x)\vee r^{\prime}(y))$\\
     $\bar{\neg} x:=e(\neg r(x))$ and $\lrcorner x:=e^{\prime}(\neg r^{\prime}(x))$\\
    $\bot_{\mathfrak{D}}:=e([(\bot_{s_{1}}, \boxng\bot_{s_{1}})])$ and $\top_{\mathfrak{D}}:= e^{\prime}([(\boxnm\bot_{s_{2}}, \bot_{s_{2}})])$
 \end{center}

Then, it is shown that $(\mathfrak{D}, \sqcap, \sqcup, \bar{\neg}, \lrcorner, \top_{\mathfrak{D}}, \bot_{\mathfrak{D}})$ is a pure dBa iff $r, r^{\prime}, e, e^{\prime}$ satisfies certain conditions. Moreover, we prove a more general result as follows:

 \begin{description}
      \item  Let $(B, \wedge, \neg, 0,1)$ and $(B^{\prime}, \vee^{\prime}, \neg^{\prime}, 0',1')$ be two Boolean algebras and let $r:A \rightleftharpoons B: e$ and $r^{\prime}:A \rightleftharpoons B^{\prime}: e^{\prime}$ be  pairs of maps such that $r\circ e=id_{B}$ and $r^{\prime}\circ e^{\prime}=id_{B^{\prime}}$. If $\textbf{A}:=(A, \sqcap, \sqcup, \bar{\neg},\lrcorner, e^{\prime}(1'), e(0))$ is an algebra where $\sqcap$, $\sqcup$, $\bar{\neg}$ and $\lrcorner$ are defined as above, then $\textbf{A}$ is a dBa iff $r, r^{\prime}, e, e^{\prime}$ satisfy certain conditions.
 \end{description}
 
In a typical application of FCA to association rule mining, a formal context can represent transaction data,  where $G$ is the set of customers, $M$ is the set of product items, and $(g,m)\in I$ means that $g$ has purchased $m$. In such a scenario, the conditional probability that a group of customers will purchase a product or not is a crucial piece of information. To facilitate more expressive language for such kind of information, we also suggest possible extensions of the above-mentioned logic to two-sorted  graded  and  weighted modal logics. To demonstrate the expressiveness of two-sorted graded modal logic, we characterize a class of contexts with the properties such as partial functions, functions, one-one and onto functions in terms of logical formulas. We also discuss which two-sorted modal axiom can be extended to the weighted case.

The remainder of the paper is organized as follows. We first review basic definitions and main results of general many-sorted  modal logic and double Boolean algebras in next section. Then, in Section \ref{sec2}, we present logics \textbf{KB} and  \textbf{KF} for the representation and reasoning of rough and formal concepts, respectively.  In addition,  we investigate the concept lattices and their relationships based on the logical representation.  In Section \ref{TWBML}, we propose the two-sorted Boolean modal logic $\textbf{BM}$  as a uniform framework for representing and reasoning about different notions of concepts.  We define logical representations of concepts, semiconcepts and protoconcepts in $\textbf{BM}$ and study their relationships.  In Section~\ref{representation of dBa}, we prove a characterization of dBas in terms of Boolean algebras.   In Section \ref{sec6}, we discuss graded and weighted modal logics for expressing quantitative information in formal contexts. In Section~\ref{relatedwork}, we compare our logics with those proposed in previous work. 
Finally, we summarize the paper and indicate directions of future work in Section \ref{conclusion}.

In this paper, we use the symbol $\implies$ to denote implication in the metalanguage. The complement and power set of a set $X$ are denoted by $\overline{X}$  and $\mathcal{P}(X)$, respectively.

\section{Preliminaries}
This section cover preliminary knowledge about many-sorted modal logic, FCA, and dBa. Our primary references are \cite{mspml,wille, yao2004comparative, MR4566932}.
\subsection{Many-sorted modal logic} \label{mspml}
The many-sorted polyadic modal logic is introduced in \cite{mspml}. Its alphabet consists of a many-sorted signature $(S, \Sigma)$, where $S$ is the collection of sorts and $\Sigma$ is the set of modalities, and an {\it $S$-indexed} family $P:=\{P_{s}\}_{s\in S}$ of propositional variables, where $P_{s}\neq \emptyset$ and $P_{s}\cap P_{s'}= \emptyset$ for distinct $s, s'\in S$. Each modality $\sigma\in\Sigma$ is associated with an arity $s_{1}s_{2}\ldots s_{n}\rightarrow s$. For any $n\in \mathbb{N}$, we denote  $\Sigma_{s_{1}s_{2}\ldots s_{n}s}=\{\sigma\in \Sigma\mid \sigma:s_{1}s_{2}\ldots s_{n}\rightarrow s\}$.  For the purpose of this paper, we only need to consider unary modalities. Hence, we will simply present many-sorted monadic modal logic below.

Let $(S, \Sigma)$ be a signature such that $\Sigma$ is  a set of unary modalities. Then, the set of formulas  is an $S$-index family $Fm_{S}:=\{Fm_{s}\mid s\in S\}$, defined inductively for each $s\in S$ by
\[\varphi_s::= p_s\;\mid\;\neg\varphi_s\;\mid\;\varphi_s\wedge\varphi_s\;\mid\;\langle\sigma\rangle\varphi_{s'}\;\mid\;[\sigma]\varphi_{s'}\]
where $p_s\in P_s$ and $\sigma\in \Sigma_{s's}$.

A {\it many-sorted frame} is a pair  $\mathfrak{F}:=(\{W_{s}\}_{s\in S}, \{R_{\sigma}\}_{\sigma\in \Sigma})$ where $W_{s}\neq \emptyset$, $W_{s}\cap W_{s'}=\emptyset$ for  $s\not=s'\in S$ and $R_{\sigma}\subseteq W_{s}\times W_{s'}$  if  $\sigma\in \Sigma_{s's}$ for  $s,s'\in S$. The class of all  many-sorted  frames is denoted by $\mathbb{MSF}$. A {\it valuation} $v$ is an $S$-indexed family of  maps $\{v_{s}\}_{s\in S}$, where $v_{s}: P_{s}\rightarrow \mathcal{P}(W_{s})$. A  many-sorted model $\mathfrak{M}:=(\mathfrak{F}, v)$ consists of a many-sorted frame $\mathfrak{F}$ and a valuation $v$. The satisfaction of a  formula in a model $\mathfrak{M}$ is defined inductively as follows.

\begin{definition}
\label{satisfiction}
{\rm Let $\mathfrak{M}:=(\{W_{s}\}_{s\in S}, \{R_{\sigma}\}_{\sigma\in \Sigma}, v)$ be a  many-sorted model,  $w\in W_{s}$ and $\varphi \in Fm_{s}$  for $s\in S$. We define $\mathfrak{M}, w\models_{s} \varphi$  by induction over $\varphi$ as follows:
\begin{enumerate}
    \item $\mathfrak{M},w\models_{s} p$ iff $w\in v_{s}(p)$
    \item $\mathfrak{M},w\models_{s}\neg\varphi$ iff $\mathfrak{M},w\not\models_{s}\varphi$
    \item $\mathfrak{M},w\models_{s} \varphi_{1}\wedge \varphi_{2}$ iff $\mathfrak{M},w\models_{s}\varphi_{1}$ and $\mathfrak{M},w\models_{s}\varphi_{2}$
    \item If $\sigma\in \Sigma_{s's}$, then $\mathfrak{M},w\models_{s} \langle\sigma\rangle\psi$ iff there is $w'\in W_{s'}$ such that $(w,w')\in R_{\sigma}$ and $\mathfrak{M},w'\models_{s'}\psi$
    \item If $\sigma\in \Sigma_{s's}$, then  $\mathfrak{M},w\models_{s} [\sigma] \psi$ iff  for all $w'\in W_{s'}$, $(w,w')\in R_{\sigma}$ implies that $\mathfrak{M},w'\models_{s'}\psi$
\end{enumerate} }
\end{definition}
Let $\mathfrak{M}$ be an $(S,\Sigma)$-model. Then, for a set $\Phi\subseteq Fm_s$, $\mathfrak{M},w\models_{s}\Phi$ if $\mathfrak{M},w\models_{s}\varphi$ for all $\varphi\in\Phi$. Let $\mathcal{C}$ be a class of models. Then, for a set $\Phi\cup\{\varphi\}\subseteq Fm_{s}$, $\varphi$ is a {\em local semantic consequence\/} of $\Phi$ over $\mathcal{C}$, denoted by $\Phi\models^{\mathcal{C}}_{s}\varphi$,  if $\mathfrak{M}, w\models_{s}\Phi$ implies $\mathfrak{M}, w\models_{s}\varphi$ for all models $\mathfrak{M}\in\mathcal{C}$ and $w$ in $\mathfrak{M}$. If $\mathcal{C}$ is the class of all models, we omit the superscript and denote it as $\Phi\models_{s}\varphi$.
If $\Phi$ is empty, we say $\varphi$ is valid in  $\mathcal{C}$ and denoted it as $\models_{s}^\mathcal{C}\varphi$. When $\mathcal{C}$ is the class of all models based on a given frame $\mathfrak{F}$, we also denote it by $\models_{s}^\mathfrak{F}\varphi$.

To characterize the local semantic consequence, the modal system $\mathbf{K}_{(S, \Sigma)}:=\{\mathbf{K}_s\}_{s\in S}$ is proposed in \cite{mspml}. The axiomatic system   $\mathbf{K}_s$ is shown in Figure~\ref{fig1} where $\sigma\in\Sigma_{s's}$. When the signature is clear from the context, the subscripts may
be omitted and we simply write the system as $\mathbf{K}$.
\begin{figure}[htp]\centering
	\framebox[110mm] {\parbox{100mm}{\begin{enumerate}
\item Axioms
\begin{enumerate}
    \item (PL): All propositional tautologies of sort $s$.
    \item (K$_{\sigma}$): $[\sigma](\varphi_{s'}\rightarrow\psi_{s'})\rightarrow([\sigma]\varphi_{s'}\rightarrow[\sigma]\psi_{s'}) $
    \item (Dual$_{\sigma}$): $\langle\sigma\rangle\varphi_{s'}\leftrightarrow\neg[\sigma]\neg\varphi_{s'}$
\end{enumerate}
\item  Inference rules:
\begin{enumerate}
    \item (MP)$_{s}$: \[\infer{\psi_s}{\varphi_s, \varphi_s\rightarrow \psi_s}\]
    \item (UG$_\sigma)$: \[\infer{[\sigma]\varphi_{s'}}{\varphi_{s'}}\]
\end{enumerate}
\end{enumerate}
}} \caption{The axiomatic system $\mathbf{K}_s$}\label{fig1}
\end{figure}

\begin{definition}
    {\rm  Let $\Lambda=\{\Lambda_s\subseteq Fm_s\}_{s\in S}$ be an $S$-indexed  set of axioms. The axiomatic system of the normal modal logic defined by  $\Lambda$ is  $\mathbf{K}\Lambda=\{\mathbf{K}_s\cup\Lambda_s\}_{s\in S}$}
\end{definition}

\begin{definition}{\rm 
A sequence of formulas $\varphi_{1}, \varphi_{2}, \ldots\varphi_{n}$ is called a $\mathbf{K}\Lambda$-proof for the formula $\varphi$  if  $\varphi_{n}=\varphi$ and  $\varphi_{i}$ is an instance of formulas in $\mathbf{K}\Lambda$ or inferred from $\varphi_{1}, \ldots, \varphi_{i-1}$ using modus ponens and universal generalization.  If $\varphi\in Fm_s$ has a proof in $\mathbf{K}\Lambda$,  we say that $\varphi$ is a theorem and write $\vdash_{s}  \varphi$.
Let $\Phi\cup\{\varphi\}\subseteq Fm_s$ be a set of formulas. Then, we say that $\varphi$ is provable form $\Phi$, denoted by $\Phi\vdash_{s}  \varphi$, if there exist $\varphi_1,\ldots,\varphi_n\in\Phi$ such that $\vdash_{s}  (\varphi_1\wedge\ldots\wedge\varphi_n)\rightarrow\varphi$. In addition, the set $\Phi$ is $\mathbf{K}\Lambda$-inconsistent  if $\bot$ is provable from it, otherwise it is $\mathbf{K}\Lambda$-consistent.}
\end{definition}
The completeness of $\mathbf{K}\Lambda$  is proved by using the standard canonical model construction\cite{mspml}. First, we need the following proposition.
\begin{proposition}
\label{suffconstrongcomp}
   {\rm $\mathbf{K}\Lambda$  is strongly complete with respect to a class of models $\mathcal{C}$ iff any  $\mathbf{K}\Lambda$-consistent set of formulas is satisfied in some model from  $\mathcal{C}$.}
\end{proposition}
Then, the canonical model is defined as follows.
\begin{definition}
    \label{canonicalmodel1}
    {\rm The canonical model of $\mathbf{K}\Lambda$ is \[\mathfrak{M}^{\mathbf{K}\Lambda}:=(\{W_{s}^{\mathbf{K}\Lambda}\}_{s\in S}, \{R_{\sigma}^{\mathbf{K}\Lambda}\}_{\sigma\in\Sigma}, V^{\mathbf{K}\Lambda})\] where
    \begin{itemize}
        \item[(a)] for any $s\in S$,  $W_{s}^{\mathbf{K}\Lambda}=\{\Phi\subseteq Fm_{s}~\mid~ \Phi ~\mbox{is maximally} ~\mathbf{K}\Lambda\mbox{-consistent}\}$,
        \item[(b)] for any $\sigma\in \Sigma_{s's}$, $w\in W_{s}^{\mathbf{K}\Lambda}$, and $w'\in W_{s'}^{\mathbf{K}\Lambda}$, $R_{\sigma}^{\mathbf{K}\Lambda}ww'$ iff for any $\psi$, $\psi\in w'$ implies that $\langle\sigma\rangle\psi\in w$.
        \item[(c)] $V^{\mathbf{K}\Lambda}=\{V^{\mathbf{K}\Lambda}_{s}\}$ is the valuation defined by $V_{s}^{\mathbf{K}\Lambda}(p)=\{w\in W_{s}^{\mathbf{K}\Lambda}~\mid~ p\in w\}$ for any $s\in S$ and $p\in P_{s}$.
    \end{itemize}}
\end{definition}
For simplicity, we will omit the superscript $\mathbf{K}\Lambda$ from the canonical model from now on.
\begin{lemma}
\label{truthlemma}
    {\rm (Truth Lemma) If $s\in S$, $\varphi\in Fm_{s}$, $\sigma \in \Sigma_{s's}$ and $w\in W_{s}$ then the following hold:
    \begin{itemize}
        \item[(a)] for any $w'\in W_{s'}$,
 $R_{\sigma}ww'$ iff for any formulas $\psi$, $[\sigma]\psi\in w$ implies $\psi
        \in w'$.
        \item[(b)] If $\langle\sigma\rangle\psi\in w$ then there is $w'\in W_{s'}$ such that $\psi\in w'$ and $R_{\sigma}ww'$.
        \item[(c)] $\mathfrak{M}  , w\models_{s}\varphi$ iff $\varphi\in w$.
    \end{itemize}}
\end{lemma}

 \begin{proposition}
 \label{compcanonic}
     {\rm  If $\Phi_s$ is a $\mathbf{K}\Lambda$-consistent set of formulas then it is satisfied in  the canonical model. }
 \end{proposition}

These results lead to the soundness and completeness of $\mathbf{K}$ directly.
\begin{theorem}\label{mainsoundcomplete}
 {\rm  $\mathbf{K}$ is sound and strongly complete with respect to the class of all $(S,\Sigma)$-models, i.e. for any $s\in S$ and $\Phi\cup\{\varphi\}\subseteq Fm_{s}$, $\Phi\vdash_{s}\varphi$ iff $\Phi\models_{s}\varphi$. }
\end{theorem}

\subsection{Formal and rough concept analysis }
In addition to formal concept lattices mentioned in Section \ref{intro}, the set $\mathbf{O}(\mathbb{K})$ of all object oriented concepts and the set $\mathbf{P}(\mathbb{K})$ of all property oriented concepts also form complete lattices,  which are called {\it object oriented concept lattices} and {\it property oriented concept lattices}, respectively.

 For concept lattices ${\mathcal X}=\mathbf{B}(\mathbb{K})$, $\mathbf{O}(\mathbb{K})$,
 $\mathbf{P}(\mathbb{K})$, the set of all extents and intents of $\mathcal X$  are denoted  by  $\mathcal{X}_{ext}$ and $\mathcal{X}_{int}$, respectively.

\begin{proposition}
    {\rm For a context $\mathbb{K}:=(G, M, I)$, the following holds.
    \begin{itemize}
        \item[(a)] $\mathbf{B}(\mathbb{K})_{ext}=\{ A\subseteq G\mid A^{+-}=A\}$ and $\mathbf{B}(\mathbb{K})_{int}=\{B\subseteq M\mid B^{-+}=B\}$.
        \item[(b)] $\mathbf{P}(\mathbb{K})_{ext}=\{ A\subseteq G\mid A^{\diamondg\boxm}=A\}$ and $\mathbf{P}(\mathbb{K})_{int}=\{B\subseteq M\mid B^{\boxm\diamondg}=B\}$.
        \item[(c)] $\mathbf{O}(\mathbb{K})_{ext}=\{A\subseteq G\mid A^{\boxg\diamondm}=A\}$ and $\mathbf{O}(\mathbb{K})_{int}=\{ B\subseteq M\mid B^{\diamondm\boxg}=B\}$.
    \end{itemize}}
\end{proposition}

It can be shown that the sets $\mathbf{B}(\mathbb{K})_{ext}, \mathbf{O}(\mathbb{K})_{ext}$ and $\mathbf{P}(\mathbb{K})_{ext}$ form complete lattices and are isomorphic to the corresponding concept  lattices. Analogously, the sets $\mathbf{B}(\mathbb{K})_{int}, \mathbf{O}(\mathbb{K})_{int}$ and $\mathbf{P}(\mathbb{K})_{int}$ form complete lattices and are dually isomorphic to the corresponding concept lattices. Therefore, a concept can be identified with its extent or intent. The relationship among the three kinds of concept lattices are investigated in \cite{yao2004comparative}. 
In particular, the following theorem is proved.
\begin{theorem}
    {\rm \cite{yao2004comparative} For a context $\mathbb{K}=(G, M, I)$ and the complemented context $\overline{\mathbb{K}}=(G, M, \overline{I})$, the following holds.
    \begin{itemize}
        \item[(a)]  $\mathbf{B}(\mathbb{K})$ is isomorphic to  $\mathbf{P}(\overline{\mathbb{K}})$.
        \item[(b)] $\mathbf{P}(\mathbb{K})$. is dually  isomorphic to $\mathbf{O}(\mathbb{K})$.
        \item[(c)] $\mathbf{B}(\mathbb{K})$ is dually isomorphic to $\mathbf{O}(\overline{\mathbb{K}})$.
    \end{itemize}}
\end{theorem}

To deal with the negation of formal concepts, Wille~\cite{wille} introduced semiconcepts and protoconcepts. In line with that, object oriented semiconcepts  and  protoconcept  were  also introduced in \cite{MR4566932}.
\begin{definition}
{\rm Let $\mathbb{K}$ be a context. Then, a pair of sets  $(A,B)$ is called 
\begin{enumerate}
    \item a {\it   protoconcept} (resp. {\it semiconcept})   of $\mathbb{K}$  if  $A^{+-}=B^{-}$ (resp. $A^{+}=B$ or $B^{-}=A$). 
    \item an {\it  object oriented protoconcept} (resp. {\it semiconcept}) of $\mathbb{K}$  if  $A^{\boxg\diamondm}=B^{\diamondm}$ (resp. $A^{\boxg}=B$ or $B^{\diamondm}=A$).
\end{enumerate}}
\end{definition}

As mentioned in Section \ref{intro}, the notions of semiconcepts and protoconcepts lead to the study of dBa and pure dBa.
\begin{definition}
	\label{DBA}
	{\rm \cite{wille} An  algebra $  \textbf{D}:= (D,\sqcup, \sqcap, \bar{\neg},\lrcorner,\top,\bot)$ satisfying the following properties is called a  {\it double Boolean algebra} (dBa). For any $x,y,z \in D$,
		
		$\begin{array}{ll}
			(1a)  (x \sqcap x ) \sqcap  y = x \sqcap  y  &
			(1b)  (x \sqcup x)\sqcup  y = x \sqcup y \\
			(2a) x\sqcap y = y\sqcap  x  &
			(2b)  x \sqcup   y = y\sqcup   x  \\
			(3a) \bar{\neg} (x \sqcap  x) = \bar{\neg}  x  &
			(3b)  \lrcorner(x \sqcup   x )= \lrcorner x \\
			(4a)  x  \sqcap (x \sqcup y)=x \sqcap  x  &
			(4b)  x \sqcup  (x \sqcap y) = x \sqcup   x \\
			(5a) x \sqcap  (y \bar{\vee} z ) = (x\sqcap  y)\bar{\vee} (x \sqcap  z) &
			(5b)  x \sqcup  (y \bar{\wedge} z) = (x \sqcup  y) \bar{\wedge}  (x \sqcup  z) \\
			(6a)  x \sqcap (x\bar{\vee} y)= x \sqcap  x  &
			(6b)  x\sqcup  (x \bar{\wedge}  y) =x \sqcup   x \\
			(7a)  \bar{\neg}\bar{\neg}(x \sqcap  y)= x \sqcap  y &
			(7b)  \lrcorner\lrcorner(x \sqcup  y) = x\sqcup  y \\
			(8a)  x  \sqcap\bar{\neg}x= \bot &
			(8b)  x \sqcup \lrcorner x = \top  \\
			(9a)  \bar{\neg}\top = \bot  &
			(9b) \lrcorner\bot =\top \\
            (10a)  x \sqcap ( y \sqcap  z) = (x \sqcap  y) \sqcap  z  &
			(10b)  x \sqcup (y \sqcup  z) = (x \sqcup  y)\sqcup  z \\
            (11a) \bar{\neg}\bot = \top \sqcap   \top  &
			(11b) \lrcorner\top =\bot \sqcup  \bot \\
			(12)  (x \sqcap  x) \sqcup (x \sqcap x) = (x \sqcup x) \sqcap (x \sqcup x) &
		\end{array}$
		
		where $ x\bar{\vee} y := \bar{\neg}(\bar{\neg}x \sqcap\bar{\neg}y)$ and 
		$ x \bar{\wedge} y :=\lrcorner(\lrcorner x \sqcup \lrcorner y)$. 
If, in addition, for any $x\in D$, $x\sqcap x=x$ or $x\sqcup x=x$ holds, then \textbf{D} is called a {\it pure dBa}. Moreover, if for each $y\in D_{\sqcap}$ and $x\in D_{\sqcup}$ with $y\sqcup y=x\sqcap x$, there is a unique $z\in D$ with $z\sqcap z=y$ and $z\sqcup z=x$, $\textbf{D}$ is called {\it fully contextual}}
\end{definition}

Each dBa $\textbf{D}$ contains two underlying Boolean algebras.
\begin{proposition} 
	\label{pro1}
	{\rm \cite{vormbrock} \noindent 
		For a dBa $\textbf{ D}= (D,\sqcup,\sqcap,\bar{\neg},\lrcorner,\top,\bot)$, let $D_{\sqcap}=\{x\in D\mid x\sqcap x=x\}$ and $D_{\sqcup}=\{x\in D\mid x\sqcup x=x\}$. Then, 
		$\textbf{D}_{\sqcap}:=(D_{\sqcap},\sqcap,\bar{\vee},\bar{\neg},\bot,\bar{\neg}\bot)$  and $\textbf{D}_{\sqcup}:=(D_{\sqcup},\bar{\wedge},\sqcup,\lrcorner,\lrcorner\top,\top)$ are both Boolean algebras.  
	}
\end{proposition}

\section{Two-sorted Modal Logic for Formal Contexts}\label{sec2}
In this section, we introduce the logics \textbf{KB} and \textbf{KF}, along with their connections to rough and formal concepts. The results, along with detailed proofs, have been published in \cite{HowladerL23}. To ensure this paper is self-contained, we will include the key definitions and theorems from the previous work, though we will omit the proofs. 

\subsection{Two-sorted modal logic for rough concepts}
To represent rough concepts, we consider a particular two-sorted signature with only two unary modalities  $(\{s_{1}, s_{2}\}, \{\diamondg,\diamondm\})$ where the arity of $\diamondg$ and $\diamondm$ are $s_1s_2$ and $s_2s_1$, respectively. As usual, we denote dual modalities of $\diamondg$ and $\diamondm$ by $\boxg$ and $\boxm$, respectively.  We use $Fm(\textbf{KB}):=\{Fm(\textbf{KB})_{s_{1}}, Fm(\textbf{KB})_{s_{2}}\}$ to denote the indexed family of formulas built over this signature\footnote{To simplify the notation, we will directly write modal formulas as $\Diamond_i\varphi$ and $\Box_i\varphi$ instead of $\langle\Diamond_i\rangle\varphi$ and $[\Diamond_i]\varphi$ for $i=\mathfrak{o,p}$.}.

This modal language is interpreted in context-based models. For a context $\mathbb{K}=(G, M, I)$, a $\mathbb{K}$-based model is a quadruple $\mathfrak{M}:=(G, M, I, v)$, where $v$ is the valuation. Then, we define the satisfaction of modal formulas in the model by
\begin{enumerate}
    \item $\mathfrak{M},g\models_{s_1}\diamondm\varphi$ iff there is $m\in M$ such that $(g,m)\in I$ and $\mathfrak{M},m\models_{s_2}\varphi$,
    \item $\mathfrak{M},m\models_{s_2}\diamondg\varphi$ iff there is $g\in G$ such that $(g,m)\in I$ and $\mathfrak{M},g\models_{s_1}\varphi$.
\end{enumerate}
Hence, in terms of the general framework of many-sorted modal logic, the model is based on a special kind of two-sorted frames $\mathfrak{F}=(W_1,W_2, R_{\diamondg}, R_{\diamondm})$, where $W_1=G, W_2=M, R_{\diamondg}=I^{-1}$, and $R_{\diamondm}=I$. This kind of frame is called {\em bidirectional} because the two binary relations are converse to each other. We also say that $\mathfrak{F}$ is the corresponding frame of the context $(G,M,I)$ and any models based on it.
We use $\mathbb{CM}$ and $2\mathbb{SBF}$ to denote the classes of all context-based models  and  their corresponding  frames, respectively. 

The logic system over the particular signature is $\mathbf{KB}$ given in Figure \ref{KB}.

\begin{figure}[htp]\centering
	\framebox[135mm] {\parbox{125mm}{\begin{enumerate}
\item Axioms
\begin{enumerate}
    \item (PL): All propositional tautologies of sort $s_{i}$ for $i=1,2$.
    \item (K$^{1}_\boxg$): $\boxg(\varphi_{1}\rightarrow\psi_{2})\rightarrow(\boxg\varphi_{1}\rightarrow\boxg\psi_{2}) $ for $\phi_{1}, \psi_{2}\in Fm(\mathbf{KB})_{s_{1}}$
    \item (Dual$_\mathfrak{o}$): $\diamondg\varphi_{1}\leftrightarrow\neg\boxg\neg\varphi_{1}$
    \item $\varphi\rightarrow\boxm\diamondg \varphi$ for $\varphi\in Fm(\mathbf{KB})_{s_{1}}$
     \item (K$^{1}_\boxm$): $\boxm(\varphi_{1}\rightarrow\psi_{2})\rightarrow(\boxm\varphi_{1}\rightarrow\boxm\psi_{2}) $ for $\phi_{1}, \psi_{2}\in Fm(\mathbf{KB})_{s_{2}}$
    \item (Dual$_\mathfrak{p}$): $\diamondm\varphi_{1}\leftrightarrow\neg\boxm\neg\varphi_{1}$
    \item  $\varphi\rightarrow \boxg\diamondm\varphi$  for $\varphi\in Fm(\mathbf{KB})_{s_{2}}$
\end{enumerate}
\item  Inference rules:
\begin{enumerate}
    \item (MP)$_{s}$: for $s\in \{s_{1}, s_{2}\}$ and $\varphi,\psi\in Fm_{s}$ \[\infer{\psi}{\varphi, \varphi\rightarrow \psi}\]
    \item (UG$^{1}_\boxg)$: for $\varphi\in Fm(\mathbf{KB})_{s_{1}}$  \[\infer{\boxg\varphi}{\varphi}\]
     \item (UG$^{2}_\boxm)$: for $\varphi\in Fm(\mathbf{KB})_{s_{2}}$  \[\infer{\boxm\varphi}{\varphi}\]
\end{enumerate}
\end{enumerate}
}} \caption{The axiomatic system $\mathbf{KB}$ \label{KB}}
\end{figure}

We prove that $\mathbf{KB}$  is sound and complete with respect to $2\mathbb{SBF}$ (and so $\mathbb{CM}$, too).
As usual, the verification of soundness is straightforward. 
The completeness is proved using the canonical model of \textbf{KB}, which is an instance of that constructed in Definition \ref{canonicalmodel1}. That is,
\[\mathfrak{M}^{\textbf{KB}}:=(W_{s_{1}}^\textbf{KB}, W_{s_{2}}^\textbf{KB}, R_\diamondg^\textbf{KB}, R_{\diamondm}^\textbf{KB}, V^\textbf{KB})\] As above, we will omit the superscript \textbf{KB} for simplicity in the presentation of the proof. It is easy to see that the model satisfies the following properties for $x\in W_{s_{1}}$ and $y\in W_{s_{2}}$:
\begin{itemize}
 \item[(a)] $R_{\diamondg}yx$ iff $\varphi\in x$ implies that $\diamondg\varphi\in y$ for any $\varphi\in Fm(\mathbf{KB})_{s_1}$.
\item[(b)] $R_{\diamondm}xy$ iff $\varphi\in y$ implies that $\diamondm\varphi\in x$ for any $\varphi\in Fm(\mathbf{KB})_{s_2}$.
\end{itemize}
From these properties, we can derive the following theorem.
\begin{theorem}\label{pmbdlc}
{\rm $\mathbf{KB}$ is sound and strongly complete with respect to $2\mathbb{SBF}$, that is for any $s\in \{s_{1}, s_{2}\}$, $\varphi\in Fm(\mathbf{KB})_{s}$  and $\Phi\subseteq Fm(\mathbf{KB})_{s}$,  $\Phi\models^{2\mathbb{SBF}} _{s}\varphi$ iff  $\Phi\vdash_{s}^{\mathbf{KB}}\varphi$. }
\end{theorem}


\begin{example}
	\label{example1}
	{\rm In a typical application of FCA to association rule mining, a formal context $(G,M, I)$ can represent transaction data, where $G$ is the set of customers, $M$ is the set of product items, and for $g\in G$, and $m\in M$, $gIm$ means that the customer $g$ has bought the item $m$. Let $\mathfrak{M}=(G,M,I,v)$ be a context-based model, and let $\varphi\in Fm(\textbf{KB})_{s_{1}}$ and $\psi\in Fm(\textbf{KB})_{s_{2}}$ be formulas representing customers in the 30 to 50 age group and item set of electronic products, respectively. Then,  the formulas $\boxg \varphi$ and $\boxm \psi$ in the model $\mathfrak{M}$ may be interpreted as follows.
	\begin{itemize}
			\item For $g\in G$, $\mathfrak{M}, g\models_{s_{1}} \boxm\psi$ means that all items bought by $g$ are electronic products, or the customer $g$ only buys electronic products.
		\item For $m\in M$, $\mathfrak{M}, m\models_{s_{2}} \boxg\varphi$  means that all  customers buying $m$ are in the 30 to 50 age group.
\end{itemize} $\Box$ }
	\end{example}



\subsection{Two sorted modal logic for formal concepts}
\label{KF}
To represent formal concepts, we consider another two-sorted signature $(\{s_{1}, s_{2}\}, \{\boxng, \boxnm\}\})$, where $\Sigma_{s_{1}s_{2}}=\{\boxng\}$ and $\Sigma_{s_{2}s_{1}}=\{\boxnm\}$, and the logic  \textbf{KF} based on it. Syntactically, the signature is the same as that for $\mathbf{KB}$ except we use different symbols to denote the modalities. Hence, formation rules of formulas remain unchanged and we denote the indexed family of formulas by $Fm(\textbf{KF})=\{Fm(\textbf{KF})_{s_{1}}, Fm(\textbf{KF})_{s_{2}}\}$.
In addition, while both \textbf{KF} and $\mathbf{KB}$ are interpreted in context-based models, the main difference between them is on the way of their modalities being interpreted.
\begin{definition}\label{windosatis}
{\rm Let $\mathfrak{M}:=(G,M,I,v)$ be a context-based model. Then, 
\begin{itemize}
     \item[(a)] For $\varphi\in Fm(\textbf{KF})_{s_{1}}$ and $m\in M$, $\mathfrak{M},m \models_{s_{2}}\boxng\varphi$ iff for any $g\in G$, $\mathfrak{M},g\models_{s_{1}}\varphi$ implies $Igm$
    \item[(b)] For $\varphi\in Fm(\textbf{KF})_{s_{2}}$ and $g\in G$, $\mathfrak{M},g \models_{s_{1}}\boxnm\varphi$ iff for any $m\in M$, $\mathfrak{M},m\models_{s_{2}}\varphi$ implies $Igm$
\end{itemize}}
\end{definition}

The logic system $\mathbf{KF}:=\{\mathbf{KF}_{s_{1}}, \mathbf{KF}_{s_{2}} \}$ is shown in Figure~\ref{fig2}.
\begin{figure}[htp]\centering
	\framebox[135mm] {\parbox{125mm}{\begin{enumerate}
    \item Axioms:
\begin{itemize}
	\item[(PL)] Propositional tautologies of sort $s_i$ for $i=1,2$.
	\item [$(K_{\boxng}^{1})$] $\boxng(\varphi_{1}\wedge\neg \varphi_{2})\rightarrow (\boxng\neg\varphi_{1}\rightarrow \boxng\neg\varphi_{2})$ for $\varphi_{1},\varphi_{2}\in Fm(\mathbf{KF})_{s_1}$
    \item [$(B^{1})$]$ \varphi\rightarrow \boxnm\boxng \varphi$ for $\varphi\in Fm(\mathbf{KF})_{s_1}$
     \item [$(K_{\boxnm}^{2})$] $\boxnm(\psi_{1}\wedge \neg\psi_{2})\rightarrow (\boxnm\neg\psi_{1} \rightarrow \boxnm\neg\psi_{2})$ for $\psi_{1},\psi_{2}\in Fm(\mathbf{KF})_{s_2}$
  \item [$(B^{2})$]  $\psi\rightarrow \boxng\boxnm\psi$ for $\psi\in Fm(\mathbf{KF})_{s_2}$
  \end{itemize}
\item Inference rules:
\begin{itemize}
    \item $(MP)_{s}$: for $s\in \{s_{1}, s_{2}\}$ and $\varphi,\psi\in Fm_{s}$ 
    \[\infer{\psi}{\varphi, \varphi\rightarrow \psi}\]
    \item $(UG^{1}_{\boxng})$: for $\varphi\in Fm(\textbf{KF})_{s_{1}}$,
	 \[\infer{{\boxng}\varphi}{\neg\varphi}\] 
  \item $(UG^{2}_{\boxnm})$: for $\psi\in Fm(\textbf{KF})_{s_{2}}$, 
  \[\infer{\boxnm\psi}{\neg\psi}\]
\end{itemize}
\end{enumerate}
}} \caption{The axiomatic system \textbf{KF}}\label{fig2}
\end{figure}
\begin{example}
\label{exm1}
	{\rm Continuing with Example \ref{example1}, let $\varphi\in Fm(\textbf{KF})_{s_{1}}$ and $\psi \in Fm(\textbf{KF})_{s_{2}}$ remain unchanged. Then, the intuitive meanings of  the formulas $\boxng\varphi$ and $\boxnm\psi$ are as follows:
		\begin{itemize}
			\item For $g\in G$, $\mathfrak{M}, g\models_{s_{1}} \boxnm\psi$,  $g$ bought all electronic products.
			\item For $m\in M$, $\mathfrak{M}, m\models_{s_{2}} \boxng\varphi$, all customers in the 30 to 50 age group  bought $m$.  $\Box$
		\end{itemize}}		
	\end{example}
We define a translation $\rho:Fm(\textbf{KF})\rightarrow Fm(\mathbf{KB})$ where $\rho=\{\rho_{1}, \rho_{2}\}$ as follows:
\begin{enumerate}
\item $\rho_{i}(p):=p$ for all $p\in P_{s_{i}}$ ($i=1,2$).
\item $\rho_{i}(\varphi\wedge\psi):=\rho_{i}(\varphi)\wedge\rho_{i}(\psi)$ for $\varphi, \psi\in Fm(\textbf{KF})_{s_i}$ ($i=1,2$).
\item $\rho_{i}(\neg\varphi):=\neg\rho_{i}(\varphi)$ for $\varphi\in Fm(\textbf{KF})_{s_i}$ ($i=1,2$).
\item $\rho_{2}(\boxng\varphi):=\boxg\neg\rho_{1}(\varphi)$ for $\varphi\in  Fm(\textbf{KF})_{s_{1}}$.
\item $\rho_{1}(\boxnm\varphi):=\boxm\neg\rho_{2}(\varphi)$ for $\varphi\in Fm(\textbf{KF})_{s_{2}}$.
\end{enumerate}
It is also possible to define an inverse translation from $Fm(\mathbf{KB})$ to $Fm(\mathbf{KF})$. However, we only need the above-defined translation to prove the completeness of \textbf{KF} using that of \textbf{KB}. Based on the translation, we have the following correspondence theorem.
\begin{theorem}\label{translation}
          {\rm For any formula $\varphi\in Fm(\textbf{KF})_{s_i}(i=1,2)$  the following hold.
          \begin{itemize}
              \item[(a)] $\Phi\vdash^{\textbf{KF}}\varphi$ iff $\rho(\Phi)\vdash^\mathbf{KB}\rho(\varphi)$ for any $\Phi\subseteq Fm(\textbf{KF})_{s_i}$.
              \item[(b)]  Let $\mathfrak{M}:=(G,M,I,v)$ be a context-based model and let $\overline{\mathfrak{M}}:=(G,M,\overline{I}, v)$ be its complemented model. Then, for any  $g\in G, \mathfrak{M}, g\models_{s_1} \varphi$ iff $\overline{\mathfrak{M}}, g\models_{s_1} \rho_1(\varphi)$, and  for any  $m\in M, \mathfrak{M}, m\models_{s_2} \varphi$ iff $\overline{\mathfrak{M}}, m\models_{s_2} \rho_2(\varphi)$
              \item[(c)] $\varphi$ is valid in the class $\mathbb{CM}$ iff $\rho(\varphi)$ is valid in $\mathbb{CM}$.
          \end{itemize}
           }
\end{theorem}


 Using the  correspondence theorem, we can derive the following proposition and the completeness of {\bf KF}.
 \begin{proposition}\label{needproflattic}\badstart
{\rm \begin{itemize}
         \item[(a).] For $\varphi_{1}, \varphi_{2}\in Fm(\textbf{KF})_{s_{1}}$, If $\vdash^{\textbf{KF}} \varphi_{1}\rightarrow\varphi_{2}$ then $\vdash^{\textbf{KF}} \boxng\varphi_{2}\rightarrow\boxng\varphi_{1}$

        \item[(b).] For $\varphi_{1}, \varphi_{2}\in Fm(\textbf{KF})_{s_2}$, If $\vdash^{\textbf{KF}} \varphi_{1}\rightarrow\varphi_{2}$ then $\vdash^{\textbf{KF}} \boxnm\varphi_{2}\rightarrow\boxnm\varphi_{1}$
     \end{itemize}}
 \end{proposition}
\begin{theorem}
    {\rm \textbf{KF} is sound and strongly complete with respect to the class  $\mathbb{CM}$.}
\end{theorem}


\subsection{Logical representation of three concept lattices}\label{lattice}
Given the logic \textbf{KF} and \textbf{KB}, we can now define logical representations of formal and rough concepts. 
\begin{definition}
    {\rm Let $\mathfrak{C}$ be the bidirectional frame corresponding to the context $\mathbb{K}=(G,M,I)$. Then, we define $Fm_{PC}, Fm_{OC}\subseteq Fm(\textbf{KB})_{s_{1}}\times Fm(\textbf{KB})_{s_2}$ by
\begin{itemize}
\item[(a)] $Fm_{PC}:=\{(\varphi, \psi) \mid \models^\mathfrak{C}_{s_{1}} \varphi\leftrightarrow \boxm\psi, \models^\mathfrak{C}_{s_{2}}  \diamondg\varphi\leftrightarrow\psi\}$
 \item[(b)] $Fm_{OC}:=\{(\varphi, \psi) \mid    \models^\mathfrak{C}_{s_{1}} \varphi\leftrightarrow \diamondm\psi,  \models^\mathfrak{C}_{s_{2}} \boxg\varphi\leftrightarrow \psi\}$
\item[(c)] $Fm_{PC_{ext}}:=\pi_1(Fm_{PC})$ and $Fm_{PC_{int}}:=\pi_2(Fm_{PC})$
\item[(d)] $Fm_{OC_{ext}}:=\pi_1(Fm_{OC})$ and $Fm_{OC_{int}}:=\pi_2(Fm_{OC})$
    \end{itemize}
where $\pi_1$ and $\pi_2$ are projection operators\footnote{That is, for a subset $S\subseteq A\times B$, $\pi_1(S)=\{a\in A\mid \exists b\in B, (a.b)\in S\}$ and $\pi_2(S)=\{b\in B\mid \exists a\in A, (a.b)\in S\}$}.}
\end{definition}
By the definition, we have
$Fm_{PC_{ext}}=\{\varphi\in Fm(\textbf{KB})_{s_1}\mid~ \models^\mathfrak{C}_{s_{1} } \boxm\diamondg\varphi\leftrightarrow\varphi\}$, $Fm_{PC_{int}}=\{\varphi\in  Fm(\textbf{KB})_{s_2}\mid~ \models^\mathfrak{C}_{s_{2} } \diamondg\boxm\varphi\leftrightarrow\varphi\}$, $Fm_{OC_{ext}}=\{\varphi\in  Fm(\textbf{KB})_{s_{1}}\mid ~ \models^\mathfrak{C}_{s_{1} } \diamondm\boxg\varphi\leftrightarrow\varphi\}$, and $Fm_{OC_{int}}=\{\varphi\in  Fm(\textbf{KB})_{s_{2}}\mid ~ \models^\mathfrak{C}_{s_{2} }\boxg \diamondm\varphi\leftrightarrow\varphi\}$.

Note that these sets are implicitly parameterized by the underlying context and should be indexed with $\mathbb{K}$. However, for simplicity, we omit the index. Obviously, when $(\varphi, \psi)\in Fm_{PC}$, we have $([\!\![\varphi]\!\!], [\!\![\psi]\!\!])\in\mathbf{P}(\mathbb{K})$ for any $\mathbb{K}$-based models. Hence, $Fm_{PC}$ consists of pairs of formulas representing property oriented concepts. Analogously, $Fm_{OC}$ provides the representation of object oriented concepts. 

\begin{definition}
    {\rm Let $\mathfrak{C}:=\{(G, M,I^{-1},I)\}$ be the corresponding frame of a context. Then, we define $Fm_{FC}\subseteq Fm(\mathbf{KF})_{s_1}\times Fm(\mathbf{KF})_{s_2}$ by 
    \begin{itemize}
     \item[(a)] $Fm_{FC}:=\{(\varphi, \psi) \mid   \models^\mathfrak{C}_{s_1} \varphi\leftrightarrow \boxnm\psi,  \models^\mathfrak{C}_{s_2}  \boxng\varphi\leftrightarrow \psi\}$
    \item[(b)] $Fm_{FC_{ext}}:=\pi_1(Fm_{FC})$ and $Fm_{FC_{int}}:=\pi_2(Fm_{FC})$
    \end{itemize}}
\end{definition}
From the definition, we can derive $Fm_{FC_{ext}}=\{\varphi\in Fm(\mathbf{KF})_{s_1}\mid ~\models^\mathfrak{C}_{s_1} \boxnm\boxng\varphi\leftrightarrow\varphi\}$ and $Fm_{FC_{int}}=\{\varphi\in Fm(\mathbf{KF})_{s_{2}}\mid ~ \models^\mathfrak{C}_{s_{2} } \boxng\boxnm\varphi\leftrightarrow\varphi\}$. 
Hence, the set $Fm_{FC}$ provides a logical representation of formal concepts induced from the context $(G, M, I)$. The observation suggests the definition below.
\begin{definition}
\label{deflogcept}
          {\rm Let $\varphi\in Fm(\textbf{KB})_{s_1}$, $\psi\in  Fm(\textbf{KB})_{s_2}$,  $\eta\in Fm(\textbf{KF})_{s_1}$, and  $\gamma\in  Fm(\textbf{KF})_{s_2}$. Then, for a context  $\mathbb{K}$, we say that
          \begin{itemize}
              \item[(a)] $(\varphi, \psi)$ is a  {\it (logical) property oriented concept} of $\mathbb{K}$ if $(\varphi, \psi)\in Fm_{PC}$.
               \item[(b)] $(\varphi, \psi)$ is a {\it  (logical) object oriented concept} of $\mathbb{K}$ if $(\varphi, \psi)\in Fm_{OC}$.
                \item[(c)] $(\eta, \gamma)$ is a {\it (logical) formal concept} of $\mathbb{K}$ if $(\eta, \gamma)\in Fm_{FC}$.
          \end{itemize}}
      \end{definition}
We now explore the relationships between the three notions and their properties. In what follows, for a context $\mathbb{K}=(G,M,I)$, we use $\mathfrak{C}_{0}:=(G,M,I^{-1},I)$ and $\mathfrak{C}_{1}:=(G,M, \overline{I^{-1}},\overline{I})$ to denote frames corresponding to $\mathbb{K}$ and $\overline{\mathbb{K}}$, respectively.

Recall that we have defined the set $Fm_X$ for $X=PC,OC, FC$ with respect to a context $\mathbb{K}$, but for simplicity, we do not explicitly attach the context to the  notation. In what follows, we will use $Fm_X^c$ for $X=PC,OC, FC$ to denote the corresponding set defined with respect to $\overline{\mathbb{K}}$.

\begin{proposition}\label{mapconcept}
{\rm  Let $\mathbb{K}:=(G, M, I)$ be a context. Then, 
      \begin{itemize}
          \item[(a)]   $(\varphi, \psi)\in Fm_{PC}$  iff $(\neg\varphi, \neg\psi)\in Fm_{OC}$ for $\varphi\in Fm(\textbf{KB})_{s_1}$ and $\psi\in  Fm(\textbf{KB})_{s_2}$.
          \item[(b)] $(\varphi, \psi))\in Fm_{FC}$  iff $(\rho(\varphi), \neg \rho(\psi))\in Fm^c_{PC}$   for $\varphi\in Fm(\textbf{KF})_{s_1}$ and $\psi\in  Fm(\textbf{KF})_{s_2}$.
          \item[(c)]   $(\varphi,\psi)\in Fm_{FC}$ iff $(\neg\rho(\varphi), \rho(\psi))\in Fm^c_{OC}$  for $\varphi\in Fm(\textbf{KF})_{s_1}$ and $\psi\in  Fm(\textbf{KF})_{s_2}$.
      \end{itemize}}
      \end{proposition}
            

Now, we can  define a relation $\equiv_{1}$  on the set  $Fm_{PC}$ as follows:
For $(\varphi, \psi), (\varphi', \psi')\in Fm_{PC}$, $(\varphi, \psi) \equiv_{1}(\varphi', \psi')$ iff $\models^{\mathfrak{C}_{0}} \varphi\leftrightarrow\varphi'$. Analogously, we can define $\equiv_{2}$ and $\equiv_{3}$ on the set $Fm_{OC}$ and $Fm_{FC}$, respectively. Obviously, $\equiv_{1}, \equiv_{2}$ and $\equiv_{3}$ are all  equivalence relations.  We will denote the equivalence class of $\equiv_{i}$ by $[(\varphi, \psi)]_{\equiv_{i}} (i=1,2,3)$ and usually omit the subscript when the equivalence relation is clear.

Let $Fm_{PC}/\equiv_{1}, Fm_{OC}/\equiv_{2}$, and $ Fm_{FC}/\equiv_{3}$ be the sets  of equivalence classes.
Then we can define the following structures:
\begin{itemize}
\item $(Fm_{PC}/\equiv_{1}, \vee_{1}, \wedge_{1})$ where 
\begin{itemize}
\item  $[(\varphi, \psi)]\vee_{1}[(\varphi', \psi')]:= [(\boxm(\psi\wedge\psi'), (\psi\wedge\psi'))]$
\item  $[(\varphi, \psi)]\wedge_{1}[(\varphi', \psi')]:= [(\varphi\wedge\varphi', \diamondg(\varphi\wedge\varphi'))]$
\end{itemize}
\item  $(Fm_{OC}/\equiv_{2}, \vee_{2},  \wedge_{2})$  where 
\begin{itemize}
 \item $[(\varphi, \psi)]\vee_{2}[(\varphi', \psi')]:= [(\diamondm(\psi\wedge\psi'), (\psi\wedge\psi'))]$
\item  $[(\varphi, \psi)]\wedge_{2}[(\varphi', \psi')]:= [(\varphi\wedge\varphi', \boxg(\varphi\wedge\varphi') )]$
\end{itemize}
\item  $(Fm_{FC}/\equiv_{3}, \vee_{3}, \wedge_{3})$  where 
\begin{itemize}
\item $[(\varphi, \psi)]\vee_{3}[(\varphi', \psi']):=[(\boxnm(\psi\wedge\psi'), (\psi\wedge\psi'))]$  
\item  $[(\varphi, \psi)]\wedge_{3}[(\varphi', \psi')]:=[(\varphi\wedge\varphi', \boxng(\varphi\wedge\varphi'))]$
\end{itemize}
\end{itemize}
\begin{theorem}
 {\rm  For a context $\mathbb{K}$, $(Fm_{PC}/\equiv_{1}, \vee_{1}, \wedge_{1})$, $(Fm_{OC}/\equiv_{2}, \vee_{2}, \wedge_{2})$ and $(Fm_{FC}/\equiv_{3}, \vee_{3}, \wedge_{3})$,  are  lattices. }
\end{theorem}

Using Proposition \ref{mapconcept}, we have the following.
       \begin{theorem}
       \label{relationlogicalcon}
           {\rm Let us consider a  context $\mathbb{K}$ and its complement $\overline{\mathbb{K}}$. Then, for lattices defined above, we have
           \begin{itemize}
\item[(a)] $(Fm_{FC}/\equiv_{3}, \vee_{3}, \wedge_{3})$ and  $(Fm_{PC}^c/\equiv_{1}, \vee_{1}, \wedge_{1})$ are isomorphic.
\item[(b)] $(Fm_{PC}/\equiv_{1}, \vee_{1}, \wedge_{1})$ and $(Fm_{OC}/\equiv_{2}, \vee_{2}, \wedge_{2})$  are dually isomorphic.
\item[(c)] $(Fm_{FC}/\equiv_{3}, \vee_{3}, \wedge_{3})$ and $(Fm^c_{OC}/\equiv_{2}, \vee_{2}, \wedge_{2})$ are dually isomorphic.
\end{itemize}}
       \end{theorem}
\section{Two-sorted Boolean Modal logic}\label{TWBML}
We have seen that {\bf KB} and {\bf KF} can represent and reason about rough and formal concepts in a context, respectively. However, we can only do that in two separate logical systems. In this section, we develop a uniform framework that can deal with both formal and rough concepts at the same time. By mixing possibility and window modalities together, we can express combinations of both kinds of concepts. Specifically, we consider two-sorted signature $(\{s_{1}, s_{2}\}, \{\boxg, \boxm, \boxng, \boxnm \})$, where $\Sigma_{s_{1}s_{2}}=\{\boxg, \boxng\}$ and $\Sigma_{s_{2}s_{1}}=\{\boxm, \boxnm\}$.  The indexed family of formulas is denoted by $Fm(\textbf{BM}):=\{Fm(\textbf{BM})_{s_i}\mid i=1,2\}$.


The correspondence theorem in the last section (Theorem~\ref{translation}(b)) shows that box and window modalities in {\bf KB} and  {\bf KF} are  actually interpreted in two separate but mutually complemented context-based models. Hence, to interpret both kinds of modalities simultaneously in {\bf BM}, we need to consider the incidence relation and its complement in the same context-based model. However, to prove the completeness of our logic, we will start with a more general model in which box and window modalities are interpreted with respect to two binary relations that are pseudo-complement to each other. That is, a model with two relations $I, J$ such that $I\cup J=G\times M$ but $I$ and $J$ may be not disjoint. Then, by employing a copying technique introduced in \cite{Gargov1987} to eliminate the overlapping between $I$ and $J$, we can transform it into  a context-based model.

\begin{definition}
    {\rm A {\it generalized context-based model} (or simply generalized model) is a 5-tuple $\mathfrak{M}:=(G, M, I, J, v)$, where  both $(G, M, I)$ and $(G, M, J)$ are contexts,
    $I\cup J=G\times M$, and $v$ is a two-sorted valuation.}
\end{definition}
We define the satisfaction of \textbf{BM} formulas in $\mathfrak{M}$ inductively. For propositional connectives, and modal operators $\boxg, \boxm$, the definition remains the same as that for the satisfaction of \textbf{KB} formulas in the context-based model $\mathfrak{M}:=(G, M, I, v)$. For modal operators $\boxng, \boxnm$, we have
\begin{itemize}
    \item for $\varphi\in Fm(\textbf{BM})_{s_{1}}$ and $m\in M$,  $\mathfrak{M}, m\models_{s_{2}}\boxng\varphi$ iff for all $g\in G$, $\mathfrak{M}, g\models_{s_{1}}\varphi$ implies that $\neg Jgm$;
    \item for $\psi\in Fm(\textbf{BM})_{s_{2}}$ and $g\in G$, $\mathfrak{M}, g\models_{s_{1}}\boxnm\psi$ iff for all $m\in M$, $\mathfrak{M}, m\models_{s_{2}}\psi$ implies that $\neg Jgm$.
\end{itemize}

As in the case of Boolean modal logic \cite{Gargov1987}, we define the following operators: for $\varphi_{1}, \varphi_{2}\in Fm(\textbf{BM})_{s_{1}}$, $N_\mathfrak{o}(\varphi_{1}, \varphi_{2}):=\boxg \varphi_{1}\wedge \boxng \neg \varphi_{2}$ and $[\mathsf{U}_\mathfrak{o}]\varphi_{1}:= N_\mathfrak{o}(\varphi_{1}, \varphi_{1})$. Similarly, for $\psi_{1}, \psi_{2}\in Fm(\textbf{BM})_{s_{2}}$, $N_\mathfrak{p}(\psi_{1}, \psi_{2}):=\boxm \psi_{1}\wedge \boxnm \neg \psi_{2}$ and $[\mathsf{U}_\mathfrak{p}]\psi_{1}:= N_\mathfrak{p}(\psi_{1}, \psi_{1})$. We also define the following operators:  $\oboxg\varphi:=\boxng\neg\varphi$, $\oboxm\psi:=\boxnm\neg\psi$, $\oboxng\varphi:=\boxg\neg\varphi$ and $\oboxnm\psi:=\boxm\neg\psi$

The logic system $\mathbf{BM}:=\{\mathbf{BM}_{s_{1}}, \mathbf{BM}_{s_{2}} \}$ consist of all axioms and inference rules of \textbf{KF} and \textbf{KB} as well as the axioms given  in Figure~\ref{fig3}.
\begin{figure}[htp]\centering
	\framebox[155mm] {\parbox{145mm}{
\begin{itemize}
	\item[($N_\mathfrak{o}$)] $N_\mathfrak{o}(\varphi_{1}, \varphi_{2})\wedge N_\mathfrak{o}(\varphi_{1}\rightarrow\psi_{1},\varphi_{2} \rightarrow \psi_{2})\rightarrow N_\mathfrak{o}(\psi_{1}, \psi_{2})$  for $\varphi_{1},\varphi_{2},\psi_{1}, \psi_{2}\in Fm(\mathbf{BM})_{s_1}$
 \item[($N_\mathfrak{p}$)] $N_\mathfrak{p}(\varphi_{1}, \varphi_{2})\wedge N_\mathfrak{p}(\varphi_{1}\rightarrow \psi_{1}, \varphi_{2} \rightarrow \psi_{2})\rightarrow N_\mathfrak{p}(\psi_{1}, \psi_{2})$ for $\varphi_{1},\varphi_{2},\psi_{1}, \psi_{2}\in Fm(\mathbf{BM})_{s_2}$
	\item[($\mathsf{U}_\mathfrak{o}$)] $[\mathsf{U}_\mathfrak{o}]\varphi\rightarrow [\mathsf{U}_\mathfrak{o}][\mathsf{U}_\mathfrak{p}][\mathsf{U}_\mathfrak{o}]\varphi$, for $\varphi\in Fm(\mathbf{BM})_{s_1}$
 \item[($\mathsf{U}_\mathfrak{p}$)] $[\mathsf{U}_\mathfrak{p}]\varphi\rightarrow [\mathsf{U}_\mathfrak{p}][\mathsf{U}_\mathfrak{o}][\mathsf{U}_\mathfrak{p}]\varphi$, for $\varphi\in Fm(\mathbf{BM})_{s_2}$
  \end{itemize}
}} \caption{Axioms for the system \textbf{BM}}
\label{fig3}
\end{figure}

\begin{lemma}\label{osemantics}
    {\rm Let $\mathfrak{M}=(G, M, I, J, v)$ be a generalized model, $\varphi\in Fm(\textbf{BM})_{s_{1}}$ and $\psi\in Fm(\textbf{BM})_{s_{2}}$. The following hold. 
    \bla
        \item  For $m\in M$, $\mathfrak{M}, m\models_{s_{2}} [\mathsf{U}_\mathfrak{o}]\varphi$ iff for all $g\in G(Igm~\mbox{or}~Jgm\implies \mathfrak{M}, g\models_{s_{1}}\varphi)$.
        \item For $g\in G$, $\mathfrak{M}, g\models_{s_{1}} [\mathsf{U}_\mathfrak{p}]\psi$ iff for all $m\in M (Igm~\mbox{or}~Jgm\implies \mathfrak{M}, m\models_{s_{2}}\psi)$.
        \item  For $m\in M$, $\mathfrak{M}, m\models_{s_{2}}\oboxg\varphi$ iff for all $g\in G$, $(Jgm\implies \mathfrak{M}, g\models_{s_{1}}\varphi)$.
        \item  For $g\in G$, $\mathfrak{M}, g\models_{s_{1}} \oboxm\psi$ iff for all $m\in M$, $(Jgm\implies \mathfrak{M}, m\models_{s_{2}}\psi)$
        \item  For $m\in M$, $\mathfrak{M}, m\models_{s_{2}} \oboxng\varphi$ iff for all $g\in G$, $(\mathfrak{M}, g\models_{s_{1}}\varphi\implies \neg Igm)$.
          \item  For $g\in G$, $\mathfrak{M}, g\models_{s_{1}} \oboxnm\psi$ iff for all $m\in M$, $(\mathfrak{M}, m\models_{s_{2}}\psi\implies \neg Igm)$.
    \ela}
\end{lemma}
\begin{proof}
    Proof of (b) is similar to the proof of (a). Here we give proof for $(a)$. Let $m\in M$. Then,
    \begin{eqnarray*}
  \mathfrak{M}, m\models_{s_{2}} [\mathsf{U}_\mathfrak{o}]\varphi&\iff&\mathfrak{M}, m\models_{s_{2}} \boxg\varphi\wedge \boxng \neg \varphi, ~\mbox{by definition of}~ [\mathsf{U}_\mathfrak{o}].\\
    &\iff& \mathfrak{M}, m\models_{s_{2}} \boxg\varphi ~\mbox{and}~\mathfrak{M}, m\models_{s_{2}}\boxng \neg \varphi.\\
    &\iff& \mbox{for all}~ g\in G, (Igm \implies \mathfrak{M}, g\models_{s_{1}}\varphi) ~\mbox{and for all}~g\in G, (\mathfrak{M}, g\models_{s_{1}}\neg \varphi\implies\neg Jgm).\\
    &\iff& ~\mbox{for all}~ g\in G (Igm\implies\mathfrak{M}, g\models_{s_{1}}\varphi ~\mbox{and}~   Jgm\implies\mathfrak{M}, g\models_{s_{1}}\varphi) ~\mbox{by contraposition}.\\
    &\iff& ~\mbox{for all}~ g\in G (Igm ~or~ Jgm\implies \mathfrak{M}, g\models_{s_{1}}\varphi ).
     \end{eqnarray*}
    Proofs of $(c), (d), (e)$ and $(f)$ follow from  the definition.
\end{proof}

\begin{theorem}
\label{soundness}
    {\rm The logic \textbf{BM} is sound with respect to the class of all  generalized context-based models.}
\end{theorem}
\begin{proof}
      To prove this theorem it is sufficient to show that all the axiom in Figure \ref{fig3} are  valid.  Proofs for  $(N_\mathfrak{p})$ and $(\mathsf{U}_\mathfrak{p})$ are similar to those for $(N_\mathfrak{o})$ and $(\mathsf{U}_\mathfrak{o})$, respectively. Hence, we only present proofs for $(N_\mathfrak{o})$ and $(\mathsf{U}_\mathfrak{o})$.   
      
      For the former, let $\mathfrak{M}=(G,M,I,J,v)$ be a generalized model. Then, for any $m\in M$, $\mathfrak{M}, m\models_{s_{2}} N_\mathfrak{o}(\varphi_{1}, \varphi_{2})\wedge N_\mathfrak{o}(\varphi_{1}\rightarrow \psi_{1}, \varphi_{2} \rightarrow \psi_{2})$ implies that $\mathfrak{M}, m\models_{s_{2}}\boxg\varphi_{1}\wedge\boxng\neg \varphi_{2}$ and $\mathfrak{M}, m\models_{s_{2}}\boxg (\varphi_{1}\rightarrow \psi_{1})\wedge\boxng\neg (\varphi_{2}\rightarrow\psi_{2})$ by definition. This is equivalent to $\mathfrak{M}, m\models_{s_{2}}\boxg\varphi_{1}$, $\mathfrak{M}, m\models_{s_{2}}\oboxg\varphi_{2}$, $\mathfrak{M}, m\models_{s_{2}}\boxg (\varphi_{1}\rightarrow \psi_{1})$ and $\mathfrak{M}, m\models_{s_{2}}\oboxg(\varphi_{2}\rightarrow\psi_{2})$. By Lemma~\ref{osemantics}, we have $\mathfrak{M}, m\models_{s_{2}}\boxg\psi_{1}\wedge\oboxg\psi_{2}$, which is equivalent to $\mathfrak{M},m\models_{s_2} N_\mathfrak{o}(\psi_{1}, \psi_{2})$.

       For the latter, assume  $\mathfrak{M}, m_0\models_{s_{2}} [\mathsf{U}_\mathfrak{o}]\varphi$ for some $m_0\in M$. Then, by Lemma~\ref{osemantics}, for all $g\in G$ ($Igm_0$ or $Jgm_0 \implies \mathfrak{M},g\models_{s_{1}}\varphi$), which implies that for all $g\in G$, $\mathfrak{M},g\models_{s_{1}}\varphi$ as $I\cup J=G\times M$. Hence, for all $m\in M$, $\mathfrak{M}, m\models_{s_{2}} [\mathsf{U}_\mathfrak{o}]\varphi$, which in turn implies for all $g\in G$, $\mathfrak{M}, g\models_{s_1} [\mathsf{U}_\mathfrak{p}][\mathsf{U}_\mathfrak{o}]\varphi$ and for all $m\in M$ (and thus $m_0$), $\mathfrak{M}, m\models_{s_{2}} [\mathsf{U}_\mathfrak{o}][\mathsf{U}_\mathfrak{p}][\mathsf{U}_\mathfrak{o}]\varphi$.           
\end{proof}

For the canonical model construction, let $G$ and $M$ be the sets of all maximally consistent subsets of $Fm(\mathbf{BM})_{s_1}$ and $Fm(\mathbf{BM})_{s_2}$, respectively.  Then, for all $X\in G$ and $Y\in M$, we define the following sets of formulas.

      $\boxm X:=\{\varphi\mid\boxm\varphi\in X\}$, $\boxnm\neg X:=\{\varphi\mid\boxnm\neg\varphi\in X\}$, and $[\mathsf{U}_\mathfrak{p}]X:=\{\varphi\mid [\mathsf{U}_\mathfrak{p}]\varphi\in X\}$.

      $\boxg Y:=\{\varphi\mid\boxg\varphi\in Y\}$, $\boxng \neg Y:=\{\varphi\mid\boxng\neg\varphi\in Y\}$, and $[\mathsf{U}_\mathfrak{o}] Y:=\{\varphi\mid [\mathsf{U}_\mathfrak{o}] \varphi\in Y\}$.
      
The next two lemmas give basic properties of these sets of formulas.      
\begin{lemma}
\label{lemconectiveconn}
    {\rm For $X\in G$ and $Y\in M$,  $[\mathsf{U}_\mathfrak{p}] X=\boxm X\cap \boxnm\neg X$ and $[\mathsf{U}_\mathfrak{o}]Y=\boxg Y\cap \boxng\neg Y$.}
    \end{lemma}
\begin{proof}
    $\varphi\in [\mathsf{U}_\mathfrak{p}]X$ iff $[\mathsf{U}_\mathfrak{p}]\varphi\in X$ iff $\boxm \varphi\wedge\boxnm\neg\varphi\in X$ iff $\boxm\varphi\in X$ and $\boxnm \neg\varphi\in X$, as $X$ is a maximally consistent set. Hence, $\varphi\in [\mathsf{U}_\mathfrak{p}]X$ iff $\varphi\in \boxm  X$ and $\varphi\in \boxnm \neg X$, i.e., $\varphi\in \boxm X\cap \boxnm \neg X$. Therefore  $[\mathsf{U}_\mathfrak{p}] X=\boxm X\cap \boxnm\neg X$. The proof of $[\mathsf{U}_\mathfrak{o}]Y=\boxg Y\cap \boxng\neg Y$ is similar.
\end{proof}
      
\begin{lemma}\label{relsemantics}For all $X\in G$ and $Y\in M$,
          {\rm \bla
              \item $\boxm X\subseteq Y$ iff $\boxg Y\subseteq X$ ,
              \item $\boxnm \neg X\subseteq Y$ iff $\boxng\neg Y\subseteq X$,
              \item $[\mathsf{U}_\mathfrak{p}]X\subseteq Y$ iff $\boxm X\subseteq Y$ or $\boxnm \neg X\subseteq Y$
              \item $[\mathsf{U}_\mathfrak{o}] Y\subseteq X$ iff $\boxg Y\subseteq X$ or $\boxng\neg Y\subseteq X$
              \item $[\mathsf{U}_\mathfrak{p}]X\subseteq Y$  iff $[\mathsf{U}_\mathfrak{o}] Y\subseteq X$.
          \ela}
          \end{lemma}
      \begin{proof}\badstart
      \bla
          \item Assume that  $\boxm X\subseteq Y$ but $\boxg Y\not\subseteq X$. Then there exists $s_{1}$-sort formula $\psi$ such that $\psi\in \boxg Y$ and $\psi\notin X$. $\psi\in \boxg Y$  implies that $\boxg\psi\in Y$ by definition of the set $\boxg Y$ and  $\psi\notin X$ implies that $\neg\psi\in X$, as $X$ is maximally consistent.  In addition, the  maximal consistency of $X$ implies $\neg\psi\rightarrow \boxm\diamondg\neg\psi\in X$. Using modus ponens, we have $\boxm\diamondg\neg\psi\in X$ which implies that $\diamondg\neg\psi\in\boxm X\subseteq Y$ by definition of the set $\boxm X$ . However, $\diamondg \neg\psi\in Y$ implies that $\neg \boxg\psi\in Y$ which is in contradiction  with $\boxg\psi\in Y$. Analogously, we can show that $\boxg Y\subseteq X$ implies $\boxm X\subseteq Y$.
          \item Assume that $\boxng\neg Y\subseteq X$ but  $\boxnm \neg X\not\subseteq Y$. Then there exists $s_{2}$-sort formula $\psi$ such that $\psi\in \boxnm\neg X$ and $\psi\notin Y$. $\psi\in\boxnm\neg X$  implies that $\boxnm\neg \psi\in X$ by definition of the set $\boxnm\neg X$ and  $\psi\notin Y$ implies that $\neg \psi\in Y$, as $Y$ is maximally consistent.  In addition, as $\neg\psi\rightarrow \boxng\boxnm\neg \psi\in Y$, we have $\boxng\boxnm\neg \psi\in Y$, using modus ponens . Again using the definition of the set   $\boxng\neg Y$, $\boxng\boxnm\neg\psi\in Y$ implies that $\neg\boxnm\neg\psi\in \boxng\neg Y\subseteq  X$ which is a contradiction as  $\psi\in\boxnm\neg X$  implies that  $\boxnm\neg\psi\in X$. Hence, $\boxng\neg Y\subseteq X$ implies  $\boxnm \neg X\subseteq Y$. The proof of the converse is similar.
          \item  If $\boxm X\subseteq Y$ or $\boxnm\neg X\subseteq Y$, then $[\mathsf{U}_\mathfrak{p}] X=\boxm X\cap\boxnm\neg X\subseteq Y$ by Lemma~\ref{lemconectiveconn}. On the other hand, assume that $[\mathsf{U}_\mathfrak{p}] X\subseteq Y$ but $\boxm X\not\subseteq Y$ and $\boxnm\neg X\not\subseteq Y$. Then, there exists $\varphi\in \boxm X $ and $\varphi\notin Y$ and there exists $\psi\in \boxnm\neg X$ and $\psi\notin Y$. Hence, $\varphi\vee\psi\notin Y$ as $Y$ is maximally consistent. In addition, $\varphi\in \boxm X $ and $\psi\in \boxnm\neg X$ implies that $\boxm\varphi\in X$ and $\boxnm\neg\psi\in X$. Then,  by PL, MP, ${\rm UG}_{\boxm}$, and $K_{\boxm}$, we can infer that $\boxm (\varphi\vee\psi)\in X$. Analogously, by PL, MP, ${\rm UG}_{\boxnm}^2$, and $K_{\boxnm}^{2}$,  we have $\boxnm\neg (\varphi\vee\psi)\in X$. Therefore, $(\varphi\vee\psi)\in \boxm X\cap \boxnm\neg X=[\mathsf{U}_\mathfrak{p}] X\subseteq Y$, which leads to contradiction.
          \item The proof is similar to that for (c).
          \item It follows immediately from (a)-(d).
      \ela
\end{proof}
\begin{definition}\label{bmcanmodel}
         {\rm  The canonical model $\mathfrak{M}_{c}:=(G, M,I,J,v)$ is a 5-tuple where, 
         \bla
             \item $G$ and $M$ are the sets of all maximally consistent subsets of $Fm(\mathbf{BM})_{s_1}$ and $Fm(\mathbf{BM})_{s_2}$, respectively,
             \item $I\subseteq G\times M$ is defined by $IXY$ iff $\boxm X\subseteq Y$ (or equivalently $\boxg Y\subseteq X$),
            \item $J\subseteq G\times M$ is defined by $JXY$ iff $\boxnm \neg X\subseteq Y$ (or  $\boxng\neg Y\subseteq X$), and   
            \item $v$ is the two-sorted valuation defined by $v_{s_1}(p_{s_1})=\{X\in G\mid p_{s_1}\in X\}$ and $v_{s_2}(p_{s_2})=\{Y\in M\mid p_{s_2}\in Y\}$. 
         \ela}
     We also use $R_c$ to denote $I\cup J$. 
      \end{definition}
\begin{lemma}[Truth Lemma]\label{truthlma}
{\rm Let $\mathfrak{M}_{c}:=(G, M,I,J,v)$ be the canonical model. Then, for any $X\in G$, $Y\in M$, $\varphi\in Fm(\mathbf{BM})_{s_1}$, and  $\psi\in Fm(\mathbf{BM})_{s_2}$, we have $\varphi\in X$  iff $\mathfrak{M}_{c},X\models\varphi$ and $\psi\in Y$  iff $\mathfrak{M}_{c},Y\models\psi$}
\end{lemma}
\begin{proof}
    By induction on the complexity of formulas, the basis case follows from the definition of $v$ and the proof for propositional connectives is straightforward. The only interesting cases are modal formulas, but their proof is quite standard. Hence, as an example, we only show the case of $\varphi=\boxnm\psi$. 
    
    On one hand, $\boxnm\psi\in X$ iff $\boxnm\neg\neg\psi\in X$ iff $\neg\psi\in \boxnm\neg X$. Hence, by the definition of $J$, $\varphi\in X$ implies that for all $Y$ such that $JXY$, $\neg\psi\in Y$ which implies  $\mathfrak{M}_{c},Y\not\models\psi$ by inductive hypothesis.  This results in $\mathfrak{M}_{c},X\models\varphi$ according to the semantics of $\boxnm\psi$.  
    
    On the other hand, assume that $\varphi=\boxnm\psi\not\in X$. Then, we show that $\boxnm\neg X\cup\{\psi\}$ is consistent. If not, then there exists $\{\psi_1\cdots,\psi_k\}\in\boxnm\neg X$ such that $\vdash\neg(\bigwedge_{i=1}^k\psi_i\wedge\psi)$. As $\psi_i\in\boxnm\neg X$ means that $\boxnm\neg\psi_i\in X$ for all $1\leq i\leq k$, by (PL), (MP), $(UG_\boxnm^2)$, and $(K_\boxnm^2)$, we can infer that  $\boxnm\psi\in X$, which is in contradiction with the assumption. Hence, there exists a maximally consistent extension $Y\supseteq\boxnm\neg X\cup\{\psi\}$ that by definition satisfies $JXY$. In addition, $\mathfrak{M}_{c},Y\models\psi$ by inductive hypothesis. Therefore, $\mathfrak{M}_{c},X\not\models\varphi$ by the semantics of $\boxnm$.
\end{proof}
Note that $\mathfrak{M}_{c}$ is not necessarily a generalized model because it is possible that $R_c=I\cup J\not=G\times M$. However, we can construct a generalized model from $\mathfrak{M}_{c}$ using its generated sub-model from any element of $G\cup M$. Without loss of generality, let us consider an $X\in G$ and define  
\[M_{X}:=\{Y\in M\mid\exists X_{i},Y_{i}(0\leq i\leq n) (X_{0}=X, Y_{n}=Y,  R_cX_{i}Y_{i}(\forall 0\leq i\leq n), R_cX_iY_{i-1}(\forall 1\leq i\leq n))\},\]
\[G_{X}:=\{Z\in G\mid\exists Y\in M_X, R_cZY\}.\]
Then, we define $\mathfrak{M}_c^{X}:=(G_{X}, M_{X},  I_X, J_X, v_X)$ where $I_{X}= I\cap (G_{X}\times M_{X})$, $J_{X}=J\cap (G_{X}\times M_{X})$, and $v_X$ is the two-sorted valuation defined by $(v_X)_{s_1}(p)=v_{s_1}(p)\cap G_{X}$, $(v_X)_{s_2}(p)=v_{s_1}(p)\cap M_{X}$, for all propositional symbols $p$,  respectively. The next two lemmas show that $\mathfrak{M}_c^{X}$ is indeed a generalized model.
\begin{lemma}\label{transtive}\badstart
\bla
    \item For all $X\in G$ and $Y\in M$, $R_cXY$ iff $[\mathsf{U}_\mathfrak{p}]X\subseteq Y$  iff $[\mathsf{U}_\mathfrak{o}] Y\subseteq X$.
    \item For all $X_1, X_2\in G$ and $Y_1, Y_2\in M$, if $R_cX_1Y_1$, $R_cX_1Y_2$, and $R_cX_2Y_1$, then $R_cX_2Y_2$.
\ela
\end{lemma}
\begin{proof}\badstart
 \bla
    \item It follows immediately from Definition~\ref{bmcanmodel} and Lemma~\ref{relsemantics}.
    \item Let us consider any $\varphi\in[\mathsf{U}_\mathfrak{p}]X_2$. Then, $[\mathsf{U}_\mathfrak{p}]\varphi\in X_2$ which implies that $[\mathsf{U}_\mathfrak{p}][\mathsf{U}_\mathfrak{o}][\mathsf{U}_\mathfrak{p}]\varphi\in X_2$. Now, $R_cX_2Y_1$ implies that $[\mathsf{U}_\mathfrak{p}]X_2\subseteq Y_1$. Hence $[\mathsf{U}_\mathfrak{o}][\mathsf{U}_\mathfrak{p}]\varphi\in Y_1$. Also, $R_cX_1Y_1$ implies that $[\mathsf{U}_\mathfrak{o}] Y_1\subseteq X_1$. Thus, $[\mathsf{U}_\mathfrak{p}]\varphi\in X_1$. Finally, $R_cX_1Y_2$ implies $[\mathsf{U}_\mathfrak{p}]X_1\subseteq Y_2$ which leads to $\varphi\in Y_2$. Therefore, we have $[\mathsf{U}_\mathfrak{p}]X_2\subseteq Y_2$, which is equivalent to $R_cX_2Y_2$.
\ela   
\end{proof}
\begin{lemma}
          {\rm $\mathfrak{M}_c^{X}$ is a generalized context-based model.}
          \label{canonicalmodel}
      \end{lemma}
        \begin{proof}
         By construction, we have $I_{X}\cup J_X\subseteq G_{X}\times M_{X}$.  To prove the converse inclusion, we first show that $R_cXY$ for any $Y\in M_X$. By definition, $Y\in M_X$ implies the existence of $X_i, Y_i$  for $0\leq i\leq n$ such that $X_0=X$, $Y_n=Y$, for all $0\leq i\leq n$, $R_cX_iY_i$, and for all $1\leq i\leq n$, $R_cX_iY_{i-1}$. By induction on $i$, we can prove that $R_cXY_i$ for all $0\leq i\leq n$, so in particular, $R_cXY_n$, i.e., $R_cXY$. The induction base is simply  $R_cX_0Y_0$. Now, if $R_cXY_i$ holds, then we can derive $R_cXY_{i+1}$ from $R_cX_{i+1}Y_i$ and $R_cX_{i+1}Y_{i+1}$  by using Lemma~\ref{transtive}. Next, for any $(Z,Y)\in G_X\times M_X$, we have $R_cXY$ as $Y\in M_X$ and there exists $Y'\in M_X$ such that $R_cZY'$ as $Z\in G_X$. In addition, $R_cXY'$ also holds as $Y'\in M_X$. Hence, we have $R_cZY$ by Lemma~\ref{transtive}. Therefore, $G_X\times M_X\subseteq R_c=I\cup J$, which implies $ (G_X\times M_X)\subseteq I_X\cup J_X$. 
        \end{proof}
We can now prove the completeness of \textbf{BM} with respect to the class of generalized models.
        \begin{theorem}
        \label{completness1}
            {\rm The logic \textbf{BM} is complete with respect to the class of generalized models.}
        \end{theorem}
        \begin{proof}
           Let $\Phi$ be a \textbf{BM}-consistent subset of $Fm(\mathbf{BM})_{s_1}$. Then, there exists a maximally consistent extension $X\supseteq \Phi$. By Lemma \ref{canonicalmodel}, we can construct the generalized model  $\mathfrak{M}_c^{X}$, and by Truth Lemma,  $\mathfrak{M}_c^{X}, X\models\varphi$ for all $\varphi\in\Phi$. Analogously,  when $\Phi$ is a \textbf{BM}-consistent subset of $Fm(\mathbf{BM})_{s_2}$, we can prove its satisfiability by constructing the generalized model $\mathfrak{M}_c^{Y}$ where $Y\in M$ is a maximally consistent extension of $\Phi$. Therefore, completeness follows from Proposition~\ref{suffconstrongcomp}.
        \end{proof}
We can now transform a generalized model to a genuine context-based model. However, before showing that, we need to extend the classical definition of bounded morphisms~\cite{blackburn2002moda} to generalized models. 
\begin{definition}Let $\mathfrak{M}:=(G, M, I, J, v)$  and $\mathfrak{M}':=(G', M', I', J', v')$  be two generalized models. Then, a two-sorted mapping $f=(f_{s_1}, f_{s_2})$ such that $f_{s_1}:G\rightarrow G'$ and $f_{s_2}:M\rightarrow M'$ is a  {\em bounded morphism} from $\mathfrak{M}$ to $\mathfrak{M}'$ if it satisfies the following three conditions:
\begin{enumerate}
    \item $g$ (resp. $m$) and $f_{s_1}(g)$ (resp. $f_{s_2}(m)$) satisfy the same set of propositional variables of sort $s_1$ (resp. $s_2$),
    \item for $R=I,J$, if $(g,m)\in R$ then $(f_{s_1}(g), f_{s_2}(m))\in R'$,
    \item for $R=I,J$, if $(f_{s_1}(g), m')\in R'$ then there exists $m\in M$ such that $f_{s_2}(m)=m'$ and $(g,m)\in R$.
\end{enumerate}
If, in addition, $f$ is surjective, then we say that $\mathfrak{M}'$ is a {\em bounded morphic image} of $\mathfrak{M}$.
\end{definition}
As there is no essential difference between the definition and that in \cite{blackburn2002moda}, we can follow the same argument there to show that if $f=(f_{s_1}, f_{s_2})$ is a bounded morphism from $\mathfrak{M}$ to $\mathfrak{M}'$ then for each formula $\varphi\in Fm({\bf BM})_{s_1}$, $\psi\in Fm({\bf BM})_{s_2}$, $g\in G$ and $m\in M$, $\mathfrak{M},g\models\varphi$ iff $\mathfrak{M}',f_{s_1}(g)\models\varphi$  and $\mathfrak{M},m\models\psi$ iff $\mathfrak{M}',f_{s_2}(m)\models\psi$\footnote{While the proof in  \cite{blackburn2002moda} does not deal with cases of window modalities, we can consider them as box modalities with respect to the relation $J$ by a contrapositive reading of their semantics.}. In other words, the satisfaction of {\bf BM} formulas is invariant under bounded morphism. As a consequence, a generalized model is modally equivalent to its  bounded morphic images.
\begin{theorem}
      \label{completeness}
           {\rm Each  generalized model $\mathfrak{M}:=(G, M, I, J, v)$ is modally equivalent to some  generalized model $\underline{\mathfrak{M}}:=(\underline{G}, \underline{M}, \underline{I}, \underline{J},\underline{v})$ such that  $\underline{I}\cap\underline{J}=\emptyset$.}
       \end{theorem}

       \begin{proof}
           Let us take an isomorphic but disjoint copy  $\mathfrak{M}^{\prime}:=(G^{\prime}, M^{\prime}, I^{\prime}, J^{\prime},  v^{\prime})$ of the given model  $\mathfrak{M}$. For any $g\in G$ and $m\in M$, we use $g'$ and $m'$ to denote their corresponding isomorphic images in $G'$ and $M'$, respectively.
           Then, we construct a model $\underline{\mathfrak{M}}:=(\underline{G}, \underline{M}, \underline{I}, \underline{J}, \underline{v})$ as follows. First, $\underline{G}= G\cup G^{\prime}$, $\underline{M}= M\cup M^{\prime}$ and  $\underline{v}_{s_i}(p)=v_{s_i}(p)\cup v_{s_i}^{\prime}(p)$ for propositional variables $p$ and $i=1, 2$. Second, we define the relations  $\underline{I}$ and $\underline{J}$ by the following rules:
           \begin{itemize}
               \item If $Igm$ and $Jgm$, then set $\underline{I}gm'$, $\underline{I}g'm$, $\underline{J}gm$, and $\underline{J}g'm'$.
               \item If $Igm$ and $\neg Jgm$, then set $\underline{I}gm$, $\underline{I}gm'$, $\underline{I}g'm$, and $\underline{I}g'm'$.
               \item If $\neg Igm$ and $Jgm$, then set $\underline{J}gm$, $\underline{J}gm'$, $\underline{J}g'm$, and $\underline{J}g'm'$.
           \end{itemize}
As $I\cup J=G\times M$, there is no possibility of $\neg Igm$ and $\neg Jgm$. Hence, the three rules are exhaustive. By the construction, it is easy to see that $\underline{I}\cap\underline{J}=\emptyset$ and $\underline{I}\cup\underline{J}=\underline{G}\times \underline{M}$.

Since $\mathfrak{M}'$ is an isomorphic copy of $\mathfrak{M}$, they are modally equivalent. That is for any $s_1$ formula $\varphi$ and $s_2$ formula $\psi$, $\mathfrak{M}, g\models\varphi$ iff $\mathfrak{M}', g'\models\varphi$ and $\mathfrak{M}, m\models\psi$ iff $\mathfrak{M}', m'\models\psi$. 

Next, we define a two-sorted mapping $f$ as $f_{s_1}:\underline{G}\rightarrow G$ and $f_{s_2}:\underline{M}\rightarrow M$ such that $f_{s_1}(g)=f_{s_1}(g')=g$ and $f_{s_2}(m)=f_{s_2}(m')=m$. Then it is easy to verify that $f$ is a surjective bounded morphism from $\underline{\mathfrak{M}}$ to $\mathfrak{M}$. Hence, $\underline{\mathfrak{M}}$ and $\mathfrak{M}$ are modally equivalent.
       \end{proof}
It is easy to see that a generalized model  $\mathfrak{M}:=(G, M, I, J, v)$ satisfying $I\cap J=\emptyset$ is actually a context-based model because in this case, $J=\overline{I}$ and the modalities $\{\boxg,\boxm\}$ and $\{\boxng,\boxnm\}$ are interpreted in $\mathfrak{M}$ exactly as their interpretations with $(G, M, I, v)$ in \textbf{KB} and \textbf{KF}, respectively. Hence, we reach the following conclusion.
\begin{theorem}
    {\rm The logic \textbf{BM} is sound and complete with respect to the class of all context-based models.}
\end{theorem}
\begin{proof}
  It follows from   Theorem~\ref{soundness}, \ref{completness1} and  \ref{completeness}.
\end{proof}
\begin{re}
    {\rm Let us consider two fragments of the logic $\bf{BM}$, $\mathcal{L}_{\Box}$ and $\mathcal{L}_{\boxminus}$, generated by $\{\boxg, \boxm\}$ and $\{\boxng, \boxnm\}$, respectively.   According to Section~\ref{lattice},  $\mathcal{L}_{\Box}$ and  $\mathcal{L}_{\boxminus}$  can syntactically represent rough  and  formal concepts in contexts.   Hence,  the logic $\textbf{BM}$ provides a unified logical framework for the representation and reasoning about both kinds of concepts.}
\end{re}

\subsection{Logical representation of concepts in {\bf BM}}\label{BMLdBa}
In Section \ref{lattice}, we show that certain pairs of formulas in \textbf{KF} and \textbf{KB} can represent formal and rough concepts, respectively. As   \textbf{BM} contains all logical symbols of \textbf{KF} and \textbf{KB}, it provides a uniform framework to achieve the full expressive power of both.  In this subsection, we show that such expressive power facilitates the representation of the notions of concepts, semiconcepts and protoconcepts in FCA and RST in the single framework of \textbf{BM}. We also explicate the significance of the representation. Additionally, the logical representation of semiconcepts and protoconcepts will play an important role in our proof of the characterization of dBas in terms of Boolean algebras in next section.

To define logical representations for different notions of concepts, let $\mathbb{K}=(G, M, I)$ be a context and let $\mathfrak{C}_{\mathbb{K}}$  denote the class of all $\mathbb{K}$-based models.
\begin{definition}
    {\rm  Let $\varphi\in Fm(\textbf{BM})_{s_{1}}$ and $\psi\in Fm(\textbf{BM})_{s_{2}}$.  A pair of formulas $(\varphi, \psi)$ is called,

    \bla
    \item a {\it (logical) formal concept} of $\mathbb{K}$ if $\models^{\mathfrak{C}_{\mathbb{K}}}_{s_1} \varphi\leftrightarrow \boxnm\psi$ and   $\models^{\mathfrak{C}_{\mathbb{K}}}_{s_2}  \boxng\varphi\leftrightarrow \psi$,
    \item  a {\it (logical) property oriented concept}  of $\mathbb{K}$ if $\models^{\mathfrak{C}_{\mathbb{K}}}_{s_{1}} \varphi\leftrightarrow \oboxm\psi$ and $\models^{\mathfrak{C}_{\mathbb{K}}}_{s_{2}} \odiamondg\varphi\leftrightarrow\psi$,
     \item a {\it (logical) object oriented concept}  of  $\mathbb{K}$ if $\models^{\mathfrak{C}_{\mathbb{K}}}_{s_{1}} \varphi\leftrightarrow \odiamondm\psi$ and $ \models^{\mathfrak{C}_{\mathbb{K}}}_{s_{2}} \oboxg\varphi\leftrightarrow \psi$.
    \ela
   }
\end{definition}

The definition is simply a  reformulation of Definition \ref{deflogcept}. The main difference is that here we define them in a single language while  Definition \ref{deflogcept} is based on two different languages. When we studied rough and formal concepts using two different logics, we needed  a translation between their languages and two different models, respectively based on a context and its complement. Hence, for a formal concept in a given context, the corresponding property (or object) oriented concept according to Proposition~\ref{mapconcept} is defined with respect to the complemented context. Now, although we represent all these concepts in a single context, the corresponding property (or object) oriented concepts must be still defined with respect to the complement of the incidence relation $I$ for Proposition~\ref{mapconcept} to hold in the uniform framework. This is the reason why we define property and  object oriented concepts by using complemented modalities here. With such a  modified definition of property and object oriented concepts, we can have the following correspondence results just as in Proposition~\ref{mapconcept}.
\begin{proposition}
{\rm Let  $\varphi\in Fm(\textbf{BM})_{s_{1}}$ and $\psi\in Fm(\textbf{BM})_{s_{2}}$. Then, 
      \begin{itemize}
          \item[(a)]   $(\varphi, \psi)$  is a  property oriented concept iff $(\neg\varphi, \neg\psi)$ is a  object oriented concept.
          \item[(b)] $(\varphi, \psi)$  is a formal concept   iff $(\varphi, \neg \psi)$  is a  property oriented concept.
          \item[(c)]   $(\varphi,\psi)$  is a  formal concept iff $(\neg\varphi, \psi)$  is a   object oriented concept.
      \end{itemize}}
      \end{proposition}

For a logical concept $(\varphi, \psi)$, we call $\varphi$ and $\psi$ its {\it extent formula} and {\it intent formula}, respectively. We denote the sets of all logical formal concepts, logical property oriented concepts, and logical object oriented concepts by \texttt{FC}, ${\tt PC}$, and ${\tt OC}$, respectively.  As in Section \ref{lattice}, we can define equivalence relation $\equiv_{\tt FC}$, $\equiv_{\tt PC}$ and $\equiv_{\tt OC}$ on the sets ${\tt FC}$, ${\tt PC}$ and ${\tt OC}$ by identifying equivalent extent formulas of two concepts. It can be shown that the sets of equivalence classless  ${\tt FC}/\equiv_{\tt FC}$, ${\tt PC}/\equiv_{\tt PC}$ and ${\tt OC}/\equiv_{\tt OC}$ form lattices just as in Theorem \ref{relationlogicalcon}.

In addition to rough and formal concepts, the logic can also represent more general notions of concepts in FCA, such as semiconcepts and protoconcepts~\cite{wille}.
\begin{definition}
    {\rm   A pair of formulas $(\varphi, \psi)$ is called,

    \bla
    \item a {\it (logical) semiconcept} of  $\mathbb{K}$ if $\models^{\mathfrak{C}_{\mathbb{K}}}_{s_1} \varphi\leftrightarrow \boxnm\psi$ or  $\models^{\mathfrak{C}_{\mathbb{K}}}_{s_2}  \boxng\varphi\leftrightarrow \psi$,
    \item a {\it (logical) property oriented semiconcept}  of  $\mathbb{K}$ if $\models^{\mathfrak{C}_{\mathbb{K}}}_{s_{1}} \varphi\leftrightarrow \oboxm\psi$ or $\models^{\mathfrak{C}_{\mathbb{K}}}_{s_{2}}  \odiamondg\varphi\leftrightarrow\psi$,
     \item a {\it (logical) object oriented semiconcept}  of $\mathbb{K}$ if $\models^{\mathfrak{C}_{\mathbb{K}}}_{s_{1}} \varphi\leftrightarrow \odiamondm\psi$ or $ \models^{\mathfrak{C}_{\mathbb{K}}}_{s_{2}} \oboxg\varphi\leftrightarrow \psi$.
    \ela
   }
\end{definition}
We denote the sets of all logical semiconcepts, logical property oriented smiconcepts, and logical object oriented semiconcepts by ${\tt SC}$,  ${\tt PS}$, and ${\tt OS}$, respectively. Next, we define logical protoconcepts.
 \begin{definition}
    {\rm   A pair of formulas $(\varphi, \psi)$ is called,

    \bla
    \item a {\it logical protoconcept}  of   $\mathbb{K}$ if $\models^{\mathfrak{C}_{\mathbb{K}}}_{s_1} \boxng\varphi\leftrightarrow \boxng\boxnm\psi$,
    \item a  {\it logical property oriented protoconcept}  of  $\mathbb{K}$ if $\models^{\mathfrak{C}_{\mathbb{K}}}_{s_{1}} \odiamondg\varphi\leftrightarrow \odiamondg\oboxm\psi$, 
     \item a {\it logical object oriented protoconcept}  of $\mathbb{K}$ if 
     $\models^{\mathfrak{C}_{\mathbb{K}}}_{s_{2}} \odiamondm\oboxg\varphi\leftrightarrow \odiamondm\psi$.
    \ela
   }
\end{definition}

While the logic {\bf BM} (and also {\bf KB} and {\bf KF}) is aimed at the general representation and reasoning of knowledge and information contained in a context or a class of contexts, as an instance, this section demonstrates that the logic can especially represent all important notions in FCA. We can highlight the significance of the representation by looking at its difference with previous work on FCA-related logics~\cite{bimbo2001four,CONRADIE2021371,HCgame,CH,HOWLADER2023115}. Unlike those logics where a formal concept is usually represented by a single formula\footnote{See Section~\ref{relatedwork} for more detailed presentation and comparison.}, we allows the separate representation of its extent and intent with two formulas. Such a representation as pairs of formulas facilitate the logical derivation of more fine-grained information from  formal concepts, semiconcepts, and protoconcepts. Hence, with the representations, we can not only study the relationship between different notions of concepts but also derive useful information from their extents or intents independently. 

On one hand, as existing logics can only represent a particular notion of concepts, there is no way to express different kinds of concepts in a single formalism. Therefore, it is impossible to study the relationship between a formal concept
and a semiconcept in those systems, e.g., to see if there is a subsumption relation between their extents. However, given a logical formal concept  
$(\varphi_1,\psi_1)$ and semiconcept $(\varphi_2,\psi_2)$ with respect to a context, as well as background information about the context (also represented in our logical formulas), we can easily test if  $\varphi_1\rightarrow\varphi_2$ or $\varphi_2\rightarrow\varphi_1$ holds in {\bf BM}. 

On the other hand, we can derive further information from the extent or intent of a concept without resorting to its relationship with other concepts. For instance, regarding our data mining context, assume that $\varphi$  and $\psi$ denote a group of customers and a kind of product items, respectively. If it turns out that $(\varphi,\psi)$ forms a formal concept, then it means that the shopping behavior of this group of customers has strong connection with this kind of items. Now, for marketing purpose, a company may want to know more about the group of customers. For example, if $\varphi'$ denotes a community of individuals and $\varphi\rightarrow\varphi'$ is derivable from the background information of the context, then it is known that the group of customers belongs to a larger community. This will be valuable information for the design of marketing strategies because the company may target the advert for products $\psi$ at customers belonging to the community (but not in the group $\varphi$ yet).

Moreover, because our logic is flexible enough to express object and property formulas separately, it has the potential capability to represent notions not explored yet in FCA. That is, when a new notion of concept is invented, it is possibly represented by our logic if the relationship between its extent and intent can be captured by modalities used in  {\bf BM}. In such cases, results related to the new notion can also be derived using our axiomatic system. 
 
\section{Representation of double Boolean algebras}
\label{representation of dBa}
In section \ref{BMLdBa},  we define the logical protoconcepts and semiconcepts as a pair of BM formulas.  In this section, we construct dBa and pure dBa using  these pairs of formulas.

A semiconcept $(\varphi, \psi)$ is called a {\it left semiconcept} if $\models^{\mathfrak{C}_{\mathbb{K}}}_{s_2}  \boxng\varphi\leftrightarrow \psi$  and it is called a {\it right semiconcept} if $\models^{\mathfrak{C}_{\mathbb{K}}}_{s_1} \varphi\leftrightarrow \boxnm\psi$ . We denote the sets of all left and right semiconcepts by ${\tt SC}^{left}$ and  ${\tt SC}^{right}$, respectively. We can also define equivalence relations over ${\tt SC}^{left}$ and ${\tt SC}^{right}$. We say two left semiconcepts $(\varphi_{1}, \psi_{1})$ and $(\varphi_{2}, \psi_{2})$ are  equivalent, denoted by  $(\varphi_{1}, \psi_{1})\equiv^{left}_{sc}(\varphi_{2}, \psi_{2})$,  if  $\models^{\mathfrak{C}_{\mathbb{K}}}_{s_1}\varphi_{1}\leftrightarrow\varphi_{2}$. Similarly, two right semiconcepts   $(\varphi_{1}, \psi_{1})$ and $(\varphi_{2}, \psi_{2})$ are equivalent if  
$\models^{\mathfrak{C}_{\mathbb{K}}}_{s_2}\psi_{1}\leftrightarrow\psi_{2}$
and denoted by  $(\varphi_{1}, \psi_{1})\equiv^{right}_{sc}(\varphi_{2}, \psi_{2})$. For simplification, we omit the subscript from the equivalence relations and simply write them as $\equiv^{left}$ and $\equiv^{right}$. Let ${\tt SC}^{left}/\equiv^{left}$ and ${\tt SC}^{right}/\equiv^{right}$ be the corresponding sets of equivalence classes. Then, we define the  following structures :
\begin{itemize}
    \item $({\tt SC}^{left}/\equiv^{left}, \wedge, \neg, [(\bot_{s_{1}}, \boxng\bot_{s_{1}})] )$ where for any $[(\varphi, \psi)], [(\varphi^{\prime}, \psi^{\prime})]\in {\tt SC}^{left}/\equiv^{left}$, 
    \begin{itemize}
    \item $[(\varphi, \psi)]\wedge [(\varphi^{\prime}, \psi^{\prime})]:= [(\varphi\wedge\varphi^{\prime}, \boxng(\varphi\wedge\varphi^{\prime}))]$
    \item $\neg [(\varphi, \psi)]:= [(\neg \varphi, \boxng\neg\varphi)]$
\end{itemize}
\item $({\tt SC}^{right}/\equiv^{right}, \vee, \neg, [(\boxnm\bot_{s_{2}}, \bot_{s_{2}})] )$ where for any $[(\varphi, \psi)], [(\varphi^{\prime}, \psi^{\prime})]\in {\tt SC}^{right}/\equiv^{right}$, 
\begin{itemize}
   \item  $[(\varphi, \psi)]\vee [(\varphi^{\prime}, \psi^{\prime})]:= [( \boxnm(\psi\wedge\psi^{\prime}), \psi\wedge\psi^{\prime})]$
   \item  $\neg [(\varphi, \psi)]:= [(\boxnm\neg\psi, \neg\psi)]$
\end{itemize}
\end{itemize}
\begin{theorem}
    {\rm $({\tt SC}^{left}/\equiv^{left}, \wedge, \neg, [(\bot_{s_{1}}, \boxng\bot_{s_{1}})] )$ and  $({\tt SC}^{right}/\equiv^{right}, \vee, \neg, [(\boxnm\bot_{s_{2}}, \bot_{s_{2}})] )$ are both Boolean algebras.}
\end{theorem}
\begin{proof}
   We only prove this for the structure $({\tt SC}^{left}/\equiv^{left}, \wedge, \neg, [(\bot_{s_{1}}, \boxng\bot_{s_{1}})] )$. The proof for the other one is similar. First, for  $[(\varphi, \psi)], [(\varphi^{\prime}, \psi^{\prime})]\in {\tt SC}^{left}/\equiv^{left}$, we define $[(\varphi, \psi)]\vee [(\varphi^{\prime}, \psi^{\prime})]=\neg (\neg[(\varphi, \psi)]\wedge\neg [(\varphi^{\prime}, \psi^{\prime})])$. Then, the result follows immediately from the fact that all operations over the structure depend only on the first component of formula pairs and the equivalence classes of formulas form an Boolean algebra already.
\end{proof}
 Let $\mathfrak{D}\subseteq Fm(\textbf{BM})_{s_{1}}\times Fm(\textbf{BM})_{s_{2}}$ be a subset of formula pairs such that there are maps  $r:\mathfrak{D}\rightarrow {\tt SC}^{left}/\equiv^{left}$, $e:{\tt SC}^{left}/\equiv^{left}\rightarrow \mathfrak{D}$, $r^{\prime}:\mathfrak{D}\rightarrow {\tt SC}^{right}/\equiv^{right}$, and $e^{\prime}:{\tt SC}^{right}/\equiv^{right}\rightarrow \mathfrak{D}$ satisfying  
$r\circ e=id_{ {\tt SC}^{left}/\equiv^{left}}$ and $r^{\prime}\circ e^{\prime}=id_{ {\tt SC}^{right}/\equiv^{right}}$. Then, we define operations over $\mathfrak{D}$ as follows: for all $x, y\in \mathfrak{D}$
 \begin{center}
     $x\sqcap y:= e(r(x)\wedge r(y))$ $x\sqcup y:= e^{\prime}(r^{\prime}(x)\vee r^{\prime}(y))$\\
     $\bar{\neg}x:=e(\neg r(x))$ and $\lrcorner x:=e^{\prime}(\neg r^{\prime}(x))$\\
    $\bot_{\mathfrak{D}}:=e([(\bot_{s_{1}}, \boxng\bot_{s_{1}})])$ and $\top_{\mathfrak{D}}:= e^{\prime}([(\boxnm\bot_{s_{2}}, \bot_{s_{2}})])$
 \end{center}
Next, we show that $\mathfrak{D}$ forms a pure dBa with respect to these operations under some mild assumptions.
 \begin{theorem}
 \label{dbalogic}
     {\rm $(\mathfrak{D}, \sqcap, \sqcup, \bar{\neg}, \lrcorner, \top_{\mathfrak{D}}, \bot_{\mathfrak{D}})$ is a pure dBa iff the following conditions hold.
     \bla
     \item $e\circ r\circ e^{\prime}\circ r^{\prime}=e^{\prime}\circ r^{\prime}\circ e\circ r$
    \item $e(r(x)\wedge r(e^{\prime}(r^{\prime}(x)\vee r^{\prime}(y))))=e(r(x))$ and $e^{\prime}(r^{\prime}(x)\vee r^{\prime}(e(r(x)\wedge r(y))))=e^{\prime}(r^{\prime}(x))$ for all $x, y\in A$.
    \item  $r(e^{\prime}([(\boxnm\bot_{s_{2}}, \bot_{s_{2}})]))=\neg [(\bot_{s_{1}}, \boxng\bot_{s_{1}})]$ and $r^{\prime}(e( [(\bot_{s_{1}}, \boxng\bot_{s_{1}})]))=\neg[(\boxnm\bot_{s_{2}}, \bot_{s_{2}})]$.
    \item For $x\in \mathfrak{D}$, either $er(x)=x$ or $e^{\prime}r^{\prime}(x)=x$
     \ela}
 \end{theorem}
\begin{proof}
We first show that the structure satisfying these conditions also satisfies axioms (1a)-(11a) and (12) for dBa. Let $x, y, z\in \mathfrak{D}$. Then,
    \begin{description}
        \item (1a): 
        \begin{eqnarray*}
           (x\sqcap x)\sqcap y&=&e(r(x\sqcap x)\wedge r(y))~ \mbox{by definition of}~\sqcap\\
           &=&e(re(r(x)\wedge r(x))\wedge r(y))~\mbox{by definition of}~\sqcap\\
           &=& e((r(x)\wedge r(x))\wedge r(y))~\mbox{as}~re~\mbox{is an identity map on}~{\tt SC}^{left}/\equiv^{left}\\
           &=&e(r(x)\wedge r(y))\\
           &=&x\sqcap y~\mbox{by definition of}~\sqcap
        \end{eqnarray*}

        \item (2a): Commutativity follows from the definition of $\sqcap$ directly.
        \item (3a):
        \begin{eqnarray*}
            \bar{\neg}(x\sqcap x)&=&e(\neg r(x\sqcap x))~\mbox{by definition of}~\bar{\neg}\\&=&e(\neg re(r(x)\wedge r(x)))~\mbox{by definiton of}~\sqcap\\&=&e(\neg (r(x)\wedge r(x)))~\mbox{as}~re~\mbox{is an identity map on}~{\tt SC}^{left}/\equiv^{left}\\&=& e(\neg r(x))\\&=&\bar{\neg}x~\mbox{by definition of}~\bar{\neg}
        \end{eqnarray*}
        
        \item (4a): it follows from the first part of condition (b).
        \item (5a):
        \begin{eqnarray*}
            x\sqcap (y\bar{\vee} z)&=&x\sqcap \bar{\neg}(\bar{\neg}y\sqcap\bar{\neg}z)~\mbox{by definition of}~\bar{\vee}\\&=&e(r(x)\wedge r(\bar{\neg}(\bar{\neg}y\sqcap\bar{\neg} z)))~\mbox{by definition of }~\sqcap\\&=&e(r(x)\wedge r(e(\neg r (\bar{\neg}y\sqcap \bar{\neg}z))))~\mbox{by definition of }~\bar{\neg}\\&=&e(r(x)\wedge \neg r (\bar{\neg}y\sqcap\bar{\neg}z))~\mbox{as}~re~\mbox{is an identity map on}~{\tt SC}^{left}/\equiv^{left}\\&=&e(r(x)\wedge \neg r e (r(\bar{\neg}y)\wedge r(\bar{\neg}z)))~\mbox{by definition of}~\sqcap\\&=&e(r(x)\wedge \neg  (re(\neg r(y))\wedge re(\neg r(z))))~\mbox{by definition of}~\bar{\neg}~\mbox{and}~re ~\mbox{is an identity map}\\&=&e(r(x)\wedge \neg  (\neg r(y)\wedge \neg r(z)))~\mbox{as}~re ~\mbox{is an identity map}\\&=&e(r(x)\wedge   ( r(y)\vee  r(z)))~\mbox{by De Morgan's Laws}\\&=&e((r(x)\wedge r(y))\vee(r(x) \wedge r(z)))~\mbox{by distributivity}\\&=& (x\sqcap y)\bar{\vee} (x\sqcap z)
        \end{eqnarray*}
        
         The last equality depends on the claim that for $u,v\in {\tt SC}^{left}/\equiv^{left}$, $e(u\vee v)=e(u)\bar{\vee} e(v)$. The proof of the claim is as follows: 
         \begin{eqnarray*}
             e(u)\bar{\vee} e(v)&=&\bar{\neg}(\bar{\neg}e(u)\sqcap\bar{\neg}e(v))~\mbox{by definition of}~\bar{\vee}\\&=&e(\neg r(\bar{\neg}e(u)\sqcap\bar{\neg}e(v)))~\mbox{by definition of }~\bar{\neg}\\&=&e(\neg re(r(\bar{\neg}e(u))\wedge r(\bar{\neg}e(v))))~\mbox{by definition of}~\sqcap\\&=&e(\neg (r(\bar{\neg}e(u))\wedge r(\bar{\neg}e(v))))~re~\mbox{is an identity map on}~{\tt SC}^{left}/\equiv^{left}\\&=&e(\neg (re(\neg re(u))\wedge re(\neg re(v))))~\mbox{by definition of }\bar{\neg}\\&=&e(\neg(\neg u\wedge\neg v))~re~\mbox{is an identity map on}~{\tt SC}^{left}/\equiv^{left}\\&=&e(u\vee v)~ \mbox{by De Morgan's Laws}
         \end{eqnarray*}
          
        \item (6a):
        \begin{eqnarray*}
            x\sqcap (x\bar{\vee} y)&=&e(r(x)\wedge r( x\bar{\vee} y))~\mbox{by definition of}~\sqcap\\&=&e(r(x)\wedge   ( r(x)\vee  r(y)))\\&=&e(r(x))~\mbox{by absorption law}\\&=&e(r(x)\wedge r(x))~\mbox{by idempotemcy}\\&=& x\sqcap x~\mbox{by definition of}~\sqcap
        \end{eqnarray*}
         where the second equality follows from the fact that $r(x\bar{\vee} y)=r(x)\vee r(y)$ for any $x,y\in\mathfrak{D}$. We prove the fact as follows: 

         \begin{eqnarray*}
             r(x\bar{\vee} y)&=&r(\bar{\neg}(\bar{\neg} x\sqcap\bar{\neg}y))~\mbox{by definition of}~\bar{\vee}\\&=&re(\neg r(\bar{\neg}x\sqcap\bar{\neg}y))~\mbox{by definition of}~\bar{\neg}\\&=&\neg re(r(\bar{\neg}x)\wedge r(\bar{\neg}y))~\mbox{as}~re~\mbox{is an identity map and by definition of}~\sqcap\\&=&\neg (r(\bar{\neg}x)\wedge r(\bar{\neg}y))~\mbox{as}~re~\mbox{is an identity map}\\&=&\neg(re(\neg r(x))\wedge re(\neg r(y)))~\mbox{by defintion of}~\bar{\neg}\\&=&\neg(\neg r(x)\wedge\neg r(y))~\mbox{as}~re~\mbox{is an identity map }\\&=&r(x)\vee r(y)~ \mbox{by De Morgan,s Laws}
         \end{eqnarray*}
         
        \item  (7a):
        \begin{eqnarray*}
            \bar{\neg}\bar{\neg}(x\sqcap y)&=&e(\neg re(\neg r(x\sqcap y))~\mbox{by definition of}~\bar{\neg}\\&=&e(\neg\neg r(x\sqcap y))~\mbox{as}~re~\mbox{is an identity map}\\&=&e(r(x\sqcap y))~\mbox{by law of double negation in Boolean algebra}\\&=&e(re(r(x)\wedge r(y)))~\mbox{by definition of}~\sqcap\\&=&e(r(x)\wedge r(y))~\mbox{as}~re~\mbox{is an identity map}\\&=&x\sqcap y~\mbox{by definition of }~\sqcap.
        \end{eqnarray*}
        
        \item (8a): 
        \begin{eqnarray*}
            x\sqcap\bar{\neg}x&=&e(r(x)\wedge re(\neg r(x)))~\mbox{by definition of}~\sqcap\\&=&e(r(x)\wedge \neg r(x))~re~\mbox{is an identity map on}~{\tt SC}^{left}/\equiv^{left}\\&=&e([(\bot_{s_{1}}, \boxng\bot_{s_{1}})])~\mbox{as}~r(x)~\mbox{is an element of the Boolean algebr}~{\tt SC}^{left}/\equiv^{left}\\&=&\bot_{\mathfrak{D}}~\mbox{by definition of}~\bot_{\mathfrak{D}}.
        \end{eqnarray*}
        
        \item (9a): 
        \begin{eqnarray*}
            \bar{\neg}\top_{\mathfrak{D}}&=&e(\neg r(\top_{\mathfrak{D}}))~\mbox{by definition of }~\bar{\neg}\\&=&e(\neg r(e^{\prime}([(\boxnm\bot_{s_{2}}, \bot_{s_{2}})])))~\mbox{by definition of}~\top_{\mathfrak{D}}\\&=&e(\neg\neg [(\bot_{s_{1}}, \boxng\bot_{s_{1}})] )~\mbox{by condition (c)}\\&=& e( [(\bot_{s_{1}}, \boxng\bot_{s_{1}})] )~\mbox{by law of double negation of Boolean algebra}\\&=&\bot_{\mathfrak{D}}~\mbox{by definition of }~\bot_{\mathfrak{D}}.
        \end{eqnarray*}
        \item (10a): the associativity follows from $r\circ e=id_{{\tt SC}^{left}/\equiv^{left}}$.
        \item (11a):
        \begin{eqnarray*}
            \top_{\mathfrak{D}}\sqcap \top_{\mathfrak{D}}&=&e(r(\top_{\mathfrak{D}})\wedge r(\top_{\mathfrak{D}}))~\mbox{by definition of}~\sqcap\\&=&e(re^{\prime}([(\boxnm\bot_{s_{2}}, \bot_{s_{2}})]))~\mbox{by definition of}~\top_{\mathfrak{D}}\\&=& e(\neg [(\bot_{s_{1}}, \boxng\bot_{s_{1}})])~\mbox{by condition (c)}\\&=&e(\neg re( [(\bot_{s_{1}}, \boxng\bot_{s_{1}})]))~re~\mbox{is an identity map on}~{\tt SC}^{left}/\equiv^{left}\\&=&\bar{\neg} \bot_{\mathfrak{D}}~\mbox{by definition of}~\bar{\neg}~\mbox{and}~\bot_{\mathfrak{D}}
        \end{eqnarray*}

        \item (12): It  follows from condition (a).
    \end{description}
Analogously, we can prove that the structure satisfies axioms (1b)-(11b) for dBa. In addition, by condition (d), it follows that for any $x\in \mathfrak{D}$, either $x\sqcap x=x$ or $x\sqcup x=x$. Hence,  $(\mathfrak{D}, \sqcap, \sqcup, \bar{\neg}, \lrcorner, \top_{\mathfrak{D}}, \bot_{\mathfrak{D}})$ is a pure dBa.

Conversely, assume that $(\mathfrak{D}, \sqcap, \sqcup, \bar{\neg}, \lrcorner, \top_{\mathfrak{D}}, \bot_{\mathfrak{D}})$ is a pure dBa. Then, axiom (12) of dBa, i.e. all $x, y\in D$, $(x \sqcap  x) \sqcup (x \sqcap x) = (x \sqcup x) \sqcap (x \sqcup x)$, implies condition (a). Translating axioms (4a): $ x  \sqcap (x \sqcup y)=x \sqcap  x $ and (4b): $x \sqcup  (x \sqcap y) = x \sqcup   x$ in terms of $e, r, e^{\prime}, r^{\prime}$ we have condition (b). In addition, by axiom (9a): $\bar{\neg} \top_{\mathfrak{D}}=\bot_{\mathfrak{D}}$, we have $\bar{\neg} \top_{\mathfrak{D}}=e(\neg r(\top_{\mathfrak{D}}))=e(\neg r(e^{\prime}([(\boxnm\bot_{s_{2}}, \bot_{s_{2}})]))$ and $\bot_{\mathfrak{D}}=e( [(\bot_{s_{1}}, \boxminus\bot_{s_{1}})] )$. Hence, $e(\neg r(e^{\prime}([(\boxnm \bot_{s_{2}}, \bot_{s_{2}})]))=e( [(\bot_{s_{1}}, \boxminus\bot_{s_{1}})] )$ which implies that $\neg r(e^{\prime}([(\boxnm\bot_{s_{2}}, \bot_{s_{2}})])= [(\bot_{s_{1}}, \boxminus\bot_{s_{1}})] $, as $e$ is injective, from which the first part of condition (c) follows immediately. The second part of condition (c) is proved in the same way. Finally, because pure dBa satisfies that for all $x$, either $x\sqcap x=x$ or $x\sqcup x=x$, we have condition (d). 
\end{proof}

Next, we show that a construction satisfying conditions in Theorem~\ref{dbalogic} is always possible. However, before proving that, we need the following technical result.  
\begin{proposition}
\label{requiredproof}
    {\rm  Let $\varphi\in Fm(\textbf{BM})_{s_{1}}$ and $\psi\in Fm(\textbf{BM})_{s_{2}}$.
    \bla
        \item $\models^{\mathfrak{C}_{\mathbb{K}}}_{s_1}\boxnm \psi\leftrightarrow\boxnm \boxng \boxnm\psi$.
        \item  $ \models^{\mathfrak{C}_{\mathbb{K}}}_{s_1}\boxnm\boxng\varphi\rightarrow\boxnm(\boxng\varphi\wedge\psi)$.
        \item $ \models^{\mathfrak{C}_{\mathbb{K}}}_{s_1}\varphi\rightarrow\boxnm\boxng\varphi$.
    \ela}
\end{proposition}
\begin{proof}
Proof of the proposition is straightforward. We only give proof for $(b)$.
Let $g_{0}\in G$, and $\mathfrak{M}, g_{0}\models_{s_{1}}\boxnm\boxng\varphi$. By 
 the semantics of $\boxnm$, $(\mbox{for all}~ m\in M,~\mathfrak{M}, m\models_{s_{2}} \boxng\varphi\implies \neg Jg_{0}m)$.
Now let $m_{0}\in M$ and $\mathfrak{M}, m_{0}\models_{s_{2}} \boxng\varphi\wedge\psi$ which implies that $\mathfrak{M}, m_{0}\models_{s_{2}} \boxng\varphi$ and $\mathfrak{M}, m_{0}\models_{s_{2}} \psi$. From the above, we have $\neg Jgm$. Therefore $\mathfrak{M}, g_{0}\models_{s_{2}} \boxnm(\boxng\varphi\wedge\psi)$.  Since $g_{0}$ is an arbitrary element, $ \models^{\mathfrak{C}_{\mathbb{K}}}_{s_1}\boxnm\boxng\varphi\rightarrow\boxnm(\boxng\varphi\wedge\psi)$.
\end{proof}

Let us define an equivalence relation $\equiv_{\tt sc}$ over \texttt{SC} by $(\varphi_{1}, \psi_{1})\equiv_{\tt sc}(\varphi_{2}, \psi_{2})$ iff  $\models^{\mathfrak{C}_{\mathbb{K}}}_{s_1}\varphi_{1}\leftrightarrow\varphi_{2}$  and  $\models^{\mathfrak{C}_{\mathbb{K}}}_{s_2}\psi_{1}\leftrightarrow\psi_{2}$ and consider $\mathfrak{D}:=\texttt{SC}/\equiv_{\tt SC}$. Then, we can define the following maps:
\begin{center}
    {\rm $r: \mathfrak{D}\rightarrow  {\tt SC}^{left}/\equiv^{left}$, $r([(\varphi, \psi)])=[(\varphi, \boxng\varphi)]$ and $e:{\tt SC}^{left}/\equiv^{left}\rightarrow  \mathfrak{D}$, $e([(\varphi, \psi)])=[(\varphi, \psi)]$
    
   \noindent  $r^{\prime}: \mathfrak{D}\rightarrow  {\tt SC}^{right}/\equiv^{right}$, $r^{\prime}([(\varphi, \psi)])=[(\boxnm\psi, \psi)]$ and $e^{\prime}:  {\tt SC}^{right}/\equiv^{right}\rightarrow  \mathfrak{D}$, $e^{\prime}([(\varphi, \psi)])=[(\varphi, \psi)]$}
\end{center}

\begin{proposition}
    {\rm The maps defined above satisfy $r\circ e=id_{ {\tt SC}^{left}/\equiv^{left}}$ and $r^{\prime}\circ e^{\prime}=id_{ {\tt SC}^{right}/\equiv^{right}}$ and conditions $(a)-(d)$ of Theorem \ref{dbalogic}.}
\end{proposition}
\begin{proof}
    First, for $[(\varphi, \psi)]\in {\tt SC}^{left}/\equiv^{left} $,  $r\circ e ([(\varphi, \psi)])=r([(\varphi,\psi)])=[(\varphi, \boxng\varphi)]= [(\varphi, \psi)]$, as $\models^{\mathfrak{C}_{\mathbb{K}}}_{s_{2}} \psi\leftrightarrow\boxng\varphi$. Hence, $r\circ e=id_{ {\tt SC}^l/\equiv^l}$. Similarly we can show that $r^{\prime}\circ e^{\prime}=id_{ {\tt SC}^{right}/\equiv^{right}}$.

    Second, for $[(\varphi, \psi)]\in \mathfrak{D}$, $e\circ r\circ e^{\prime}\circ r^{\prime}([(\varphi, \psi)])=[(\boxnm\psi, \boxng\boxnm\psi)]$ and $e^{\prime}\circ r^{\prime}\circ e\circ r([(\varphi, \psi)])=[(\boxnm\boxng\varphi, \boxng\varphi)]$.  As $(\varphi, \psi)$ is a semiconcept, $\models^{\mathfrak{C}_{\mathbb{K}}}_{s_1} \varphi\leftrightarrow \boxnm\psi$ or  $\models^{\mathfrak{C}_{\mathbb{K}}}_{s_2}  \boxng\varphi\leftrightarrow \psi$. If $\models^{\mathfrak{C}_{\mathbb{K}}}_{s_1} \varphi\leftrightarrow \boxnm\psi$, then $\models^{\mathfrak{C}_{\mathbb{K}}}_{s_1}\boxng \varphi\leftrightarrow \boxng \boxnm\psi$ by Proposition \ref{needproflattic}, whence $\models^{\mathfrak{C}_{\mathbb{K}}}_{s_1}\boxnm\boxng \varphi\leftrightarrow\boxnm \boxng \boxnm\psi$. Thus $\models^{\mathfrak{C}_{\mathbb{K}}}_{s_1}\boxnm\boxng \varphi\leftrightarrow \boxnm\psi$ by Proposition \ref{requiredproof} (a).  Hence,  $e\circ r\circ e^{\prime}\circ r^{\prime}([(\varphi, \psi)])=e^{\prime}\circ r^{\prime}\circ e\circ r([(\varphi, \psi)])$. The proof for the case of $\models^{\mathfrak{C}_{\mathbb{K}}}_{s_2}  \boxng\varphi\leftrightarrow \psi$ is similar.

    Third, for $[(\varphi_{1}, \psi_{1})], [(\varphi_{2}, \psi_{2})]\in \mathfrak{D} $, $e(r([(\varphi_{1}, \psi_{1})])\wedge r(e^{\prime}(r^{\prime}([(\varphi_{1}, \psi_{1})])\vee r^{\prime}([(\varphi_{2}, \psi_{2})]))))=[(\varphi_{1}\wedge \boxnm(\psi_{1}\wedge\psi_{2}), \boxng(\varphi_{1}\wedge \boxnm(\psi_{1}\wedge\psi_{2})))]$. 
    It is clear that $\models^{\mathfrak{C}_{\mathbb{K}}}_{s_1}  \varphi_{1}\wedge \boxnm(\psi_{1}\wedge\psi_{2})\rightarrow \varphi_{1}$. As $(\varphi_{1}, \psi_{1})$ is a semiconcept, we have  $\models^{\mathfrak{C}_{\mathbb{K}}}_{s_1} \varphi_{1}\leftrightarrow \boxnm\psi_{1}$ or  $\models^{\mathfrak{C}_{\mathbb{K}}}_{s_2}  \boxng\varphi_{1}\leftrightarrow \psi_{1}$. Let us assume $\models^{\mathfrak{C}_{\mathbb{K}}}_{s_2}  \boxng\varphi_{1}\leftrightarrow \psi_{1}$ whence  $\models^{\mathfrak{C}_{\mathbb{K}}}_{s_2}  \boxng\varphi_{1}\wedge\psi_{2}\leftrightarrow \psi_{1}\wedge\psi_{2}$. Then, by Proposition \ref{needproflattic}, $\models^{\mathfrak{C}_{\mathbb{K}}}_{s_1}  \boxnm(\boxng\varphi_{1}\wedge\psi_{2})\leftrightarrow \boxnm(\psi_{1}\wedge\psi_{2})$ which implies that $\models^{\mathfrak{C}_{\mathbb{K}}}_{s_1} \boxnm\boxng\varphi_{1}\rightarrow  \boxnm(\psi_{1}\wedge\psi_{2})$. By Proposition~\ref{requiredproof} (b),  $ \models^{\mathfrak{C}_{\mathbb{K}}}_{s_1}\varphi_{1}\rightarrow\boxnm\boxng\varphi_{1}$. Hence, we have $ \models^{\mathfrak{C}_{\mathbb{K}}}_{s_1}\varphi_{1}\rightarrow\boxnm(\boxng\varphi_{1}\wedge\psi_{2})$ and $ \models^{\mathfrak{C}_{\mathbb{K}}}_{s_1}\varphi_{1}\rightarrow\varphi_{1}\wedge\boxnm(\boxng\varphi_{1}\wedge\psi_{2})$. 
   Therefore, $e(r([(\varphi_{1}, \psi_{1})])\wedge r(e^{\prime}(r^{\prime}([(\varphi_{1}, \psi_{1})])\vee r^{\prime}([(\varphi_{2}, \psi_{2})]))))=[(\varphi_{1}\wedge \boxnm(\psi_{1}\wedge\psi_{2}), \boxng(\varphi_{1}\wedge \boxnm(\psi_{1}\wedge\psi_{2})))]=[(\varphi_{1}, \boxng\varphi_{1})]=er([(\varphi_{1}, \psi_{1})])$. The proof for the case of $\models^{\mathfrak{C}_{\mathbb{K}}}_{s_1} \varphi_{1}\leftrightarrow \boxnm\psi_{1}$ is similar.
   Analogously, we can also show that $e^{\prime}(r^{\prime}([(\varphi_{1}, \psi_{1})])\vee r^{\prime}(e(r([(\varphi_{1}, \psi_{1})])\wedge r([(\varphi_{2}, \psi_{2})]))))=e^{\prime}(r^{\prime}([(\varphi_{2}, \psi_{2})]))$.

Forth,  $r(e^{\prime}([(\boxnm\bot_{s_{2}}, \bot_{s_{2}})]))=[(\boxnm\bot_{s_{2}}, \boxng\boxnm\bot_{s_{2}})]$.  To prove that $r(e^{\prime}([(\boxnm\bot_{s_{2}}, \bot_{s_{2}})]))=\neg [(\bot_{s_{1}}, \boxng\bot_{s_{1}})]$, it is sufficient to show that $\models^{\mathfrak{C}_{\mathbb{K}}}_{s_1}\boxnm\bot_{s_{2}}\leftrightarrow\neg\bot_{s_{1}}$. Now for all $\mathbb{K}$-based model $\mathfrak{M}$ and $g\in G$, $\mathfrak{M}, g \models\neg\bot_{s_{1}}$ and  $\mathfrak{M}, g\models\boxnm\bot_{s_{2}}$. Hence $\models^{\mathfrak{C}_{\mathbb{K}}}_{s_1}\boxnm\top_{s_{2}}\leftrightarrow\neg\bot_{s_{1}}$. Analogously, we can show that $r^{\prime}(e([(\bot_{s_{1}}, \boxng\bot_{s_{1}})]))=\neg [(\boxnm\bot_{s_{2}}, \bot_{s_{2}})]$.

Finally, for any $[(\varphi, \psi)]\in \mathfrak{D}$, $(\varphi, \psi)$ is either a left or a right semiconcept. Hence,  $\models^{\mathfrak{C}_{\mathbb{K}}}_{s_2}  \boxng\varphi\leftrightarrow \psi$ or $\models^{\mathfrak{C}_{\mathbb{K}}}_{s_1} \varphi\leftrightarrow \boxnm\psi$, which implies $r\circ e([(\varphi, \psi)])=[(\varphi, \boxng\varphi)]=[(\varphi, \psi)]$ or $r^{\prime}\circ e^{\prime}([(\varphi, \psi)])=[(\boxnm\psi, \psi)]=[(\varphi, \psi)]$, respectively.
\end{proof}

Note that a logical semiconcept is also a logical protoconcept. We denote the set of all logical protoconcepts by ${\tt PR}$. Thus $ {\tt SC}^{left}\cup {\tt SC}^{right}\subseteq {\tt PR}$.
As done with the set of logical semiconcepts, we can define an equivalence relation $\equiv_{\tt PR}$ over the set ${\tt PR}$. Let $\mathfrak{D}_1={\tt PR}/\equiv_{\tt PR}$ be the set of its equivalence classes. Then, with the definition of the following maps:
    \begin{center}
    {\rm $r_{1}: \mathfrak{D}_1\rightarrow  {\tt SC}^{left}/\equiv^{left}$, $r_{1}([(\varphi, \psi)])=[(\varphi, \boxng\varphi)]$ and $e_{1}:  {\tt SC}^{left}/\equiv^{left}\rightarrow\mathfrak{D}_1$, $e_{1}([(\varphi, \psi)])=[(\varphi, \psi)]$
    
   \noindent  $r_{1}^{\prime}: \mathfrak{D}_1\rightarrow  {\tt SC}^{right}/\equiv^{right}$, $r_{1}^{\prime}([(\varphi, \psi)])=[(\boxnm\psi, \psi)]$ and $e_{1}^{\prime}:  {\tt SC}^{right}/\equiv^{right} \rightarrow\mathfrak{D}_1$, $e_{1}^{\prime}([(\varphi, \psi)])=[(\varphi, \psi)]$}
\end{center}
we can show that  $r_{1}\circ e_{1}=id_{ {\tt SC}^{left}/\equiv^{left}}$, $r_{1}^{\prime}\circ e_{1}^{\prime}=id_{ {\tt SC}^{right}/\equiv^{right}}$, and $r_{1}, e_{1}, r^{\prime}_{1}$ and $e^{\prime}_{1}$ satisfy  (a)-(c) of Theorem \ref{dbalogic}. 

Let us define operations on the set $\mathfrak{D}_1$ as follows: for all $x, y\in  \mathfrak{D}_1$

 \begin{center}
     $x\sqcap y:= e_{1}(r_{1}(x)\wedge r_{1}(y))$ $x\sqcup y:= e_{1}^{\prime}(r_{1}^{\prime}(x)\vee r_{1}^{\prime}(y))$\\
     $\bar{\neg} x:=e_{1}(\neg r_{1}(x))$ and $\lrcorner x:=e_{1}^{\prime}(\neg r_{1}^{\prime}(x))$\\
    $\bot_{\mathfrak{D}}:=e_{1}([(\bot_{s_{1}}, \boxng\bot_{s_{1}})])$ and $\top_{\mathfrak{D}}:= e_{1}^{\prime}([(\boxnm\bot_{s_{2}}, \bot_{s_{2}})])$
 \end{center}
 Then we can show that $(\mathfrak{D}_1, \sqcap, \sqcup, \bar{\neg}, \lrcorner, \top, \bot)$ is a dBa. Moreover, we can prove the following general theorem.
\begin{theorem}
    {\rm Let $(B, \wedge, -, 0,1)$ and $(B',\vee^{\prime},-^{\prime}, 0',1')$ be two Boolean algebras and let $r:A \rightleftharpoons B: e$ and $r^{\prime}:A \rightleftharpoons B^{\prime}: e^{\prime}$ be  pairs of maps such that $r\circ e=id_{B}$ and $r^{\prime}\circ e^{\prime}=id_{B^{\prime}}$. $\textbf{A}:=(A, \sqcap, \sqcup, \bar{\neg},\lrcorner, e^{\prime}(1'), e(0))$ is a universal algebra where, $x\sqcap y:=e(r(x)\wedge r(y)), x\sqcup y:=e^{\prime}(r^{\prime}(x)\vee' r^{\prime}(y)), \neg x:=e(-r(x))$ and $\lrcorner x:=e^{\prime}(-^{\prime}r^{\prime}(x))$. Then $\textbf{A}$ is a dBa iff the following holds
    \bla
    \item $e\circ r\circ e^{\prime}\circ r^{\prime}=e^{\prime}\circ r^{\prime}\circ e\circ r$
    \item $e(r(x)\wedge r(e^{\prime}(r^{\prime}(x)\vee' r^{\prime}(y))))=e(r(x))$ and $e^{\prime}(r^{\prime}(x)\vee' r^{\prime}(e(r(x)\wedge r(y))))=e^{\prime}(r^{\prime}(x))$ for all $x, y\in A$.
    \item  $r(e^{\prime}(1^{\prime}))=1$ and $r^{\prime}(e(0))=0'$.
    \ela
    Moreover every dBa can be obtained from such a construction.}
\end{theorem}

\begin{proof}
    The first part of the proof is similar to that of Theorem \ref{dbalogic}. 
    
    To show that any dBa can be obtained from such a construction, let $\textbf{D}:=(D, \sqcap , \sqcup, \bar{\neg}, \lrcorner, \top, \bot)$ be a  dBa. Then, by Proposition \ref{pro1}, the two Boolean algebras in consideration are 
    \[(B,\wedge,-,0,1)=\textbf{D}_{\sqcap}:=(D_{\sqcap},\sqcap,\bar{\neg},\bot,\bar{\neg}\bot)\]
    \[(B',\vee',-',0',1')=\textbf{D}_{\sqcup}:=(D_{\sqcup},\sqcup,\lrcorner,\lrcorner\top,\top).\]
    We thus define the following maps:

    $r:D\rightarrow D_{\sqcap}$ by $r(x)=x\sqcap x$ for $x\in D$ and $e:D_{\sqcap}\rightarrow D$, by $e(x)=x$ for $x\in D_{\sqcap}$.
 
    $r^{\prime}:D\rightarrow D_{\sqcup}$ by $r^{\prime}(x)=x\sqcup x$ for $x\in D$ and $e^{\prime}:D_{\sqcup}\rightarrow D$, by $e^{\prime}(x)=x$ for $x\in D_{\sqcup}$.

 These maps are well-defined as $r(x)\sqcap r(x)=r(x)$ and $r'(x)\sqcup r'(x)=r'(x)$ follow from axioms (1a),(2a) and  (1b),(2b) of dBa, respectively. Also, by the definition, it is clear that $r\circ e=id_{D_{\sqcap}}$ and $r^{\prime}\circ e^{\prime}=id_{D_{\sqcup}}$. 
 
 We now show that they also satisfy conditions (a)-(c). 
 For condition (a), let $x\in D$. Then,  

 \begin{eqnarray*}
     e\circ r\circ e^{\prime}\circ r^{\prime}(x)&=&e\circ r\circ e^{\prime}(x\sqcup x)~\mbox{by definition  of}~r^{\prime}\\&=&e\circ r(x\sqcup x)~\mbox{by definition  of}~e^{\prime}\\&=&(x\sqcup x)\sqcap (x\sqcup x)~\mbox{by definition of }~ r~\mbox{and}~e\\&=&(x\sqcap x)\sqcup (x\sqcap x)~\mbox{by dBa axioms (12)}.
 \end{eqnarray*}
 Similar to the above we can show that $e^{\prime}\circ r^{\prime}\circ e\circ r(x)=(x\sqcap x)\sqcup (x\sqcap x)$. So $e\circ r\circ e^{\prime}\circ r^{\prime}(x)=e^{\prime}\circ r^{\prime}\circ e\circ r(x)$.

 For condition (b), let  $x, y\in D$. Then,
 
 \begin{eqnarray*}
     e(r(x)\sqcap r(e^{\prime}(r^{\prime}(x)\sqcup r^{\prime}(y))))&=&e(x\sqcap x\sqcap r(x\sqcup x\sqcup y\sqcup y))~\mbox{using definition of}~r, r^{\prime}, e^{\prime}\\&=&e(x\sqcap r(x\sqcup y))~\mbox{by dBa axioms (1a),(1b),(2b)}\\&=&x\sqcap(x\sqcup y)\sqcap (x\sqcup y)~\mbox{by definition of }~e~\mbox{and}~r\\&=&x\sqcap x\sqcap (x\sqcup y)~\mbox{by dBa axiom (4a)}\\&=&x\sqcap x\sqcap x~\mbox{by dBa axiom (4a),(10a)}\\&=&x\sqcap x=e(r(x))~\mbox{by dBa axiom (1a) and definition of}~e,r.
 \end{eqnarray*}
 
  Similarly, we can show that $e^{\prime}(r^{\prime}(x)\sqcup r^{\prime}(e(r(x)\sqcap r(y))))=e^{\prime}(r^{\prime}(x))$. For condition (c), 
  \begin{eqnarray*}
      r(e^{\prime}(1'))&=&r(e^{\prime}(\top))~\mbox{by the definition of}~1'\\&=&r(\top)~\mbox{by the definition of}~e'~\mbox{and}~\top\in D_{\sqcup}\\&=&\top\sqcap \top~\mbox{by the definition of}~r\\&=&\bar{\neg}\bot=1~\mbox{by dBa axiom (11a)}
  \end{eqnarray*}
Using dBa axiom (11b) and the fact that $\bot\in D_{\sqcap}$, we can derive the following equation in a similar way.
   \begin{eqnarray*}
 r'(e(0))&=&r'(e(\bot))=r'(\bot)\\
  &=&\bot\sqcup\bot=\lrcorner\top=0'. 
   \end{eqnarray*}
 
 Hence,  by the first part of the theorem, $(D, \sqcap_1, \sqcup_1, \bar{\neg}_1, \lrcorner_1, \top, \bot )$ is a dBa where $x\sqcap_1 y=e(r(x)\sqcap r(y))=x\sqcap y$, $x\sqcup_1 y=e^{\prime}(r^{\prime}(x)\sqcup r^{\prime}(y))=x\sqcup y$, $\neg_1x=e(\bar{\neg} r(x))=\bar{\neg}(x\sqcap x)=\bar{\neg}x$, and $\lrcorner_1 x=e'(\lrcorner r'(x))=\lrcorner(x\sqcup x)=\lrcorner x$. Therefore, \textbf{D} is indeed constructed by the two underlying Boolean algebras.
\end{proof}
Finally,  we note that what has been done with (logical) semiconcepts and protoconcepts here can carry over to object and property oriented ones.

\section{Possible extensions}\label{sec6}
In this section, we discuss two possible ways to extend two-sorted modal logics considered so far. These extensions can express more fine-grained quantitative information in context-based models, including numbers and proportions of worlds satisfying a formula.  This is motivated by the data mining application in our running example. In such application contexts, support and confidence are two important parameters to validate the induced rules\footnote{In data mining, the support of a rule $\varphi\rightarrow\psi$ is the number of instances satisfying both $\varphi$ and $\psi$, and its confidence is the proportion of instances satisfying  both $\varphi$ and  $\psi$ to those satisfying  $\varphi$.}, which indeed correspond to the two kinds of quantitative information to be modeled in the extended logics.  However, as these extensions are yet at the early stage of development, we will only present some preliminary facts about them from a semantic perspective and leave their complete axiomatization for the future work. 

\subsection{Two-sorted graded modal logic}\label{TSGML}
The first extension is two-sorted graded modal logic. Traditionally, graded modalities have been applied to counting the number of accessible possible worlds~\cite{BalestriniF99,Caro88,Cerrato90,Cerrato94,DeCaro,Fattorosi-Barnaba88,Fattorosi-BarnabaG95}. Hence, our extension is aimed at the representation and reasoning about  quantitative knowledge induced from formal contexts.  

The alphabet of the extended language consists of a two-sorted family $P:=\{P_{s_{1}}, P_{s_{2}}\}$ of propositional variables,  the propositional connectives $\vee$, $\wedge$ and $\neg$, and the modal operators $\gboxg{n}$, $\gboxng{n}$,  $\gboxm{n}$, $\gboxnm{n}$,  for $n\in \mathbb{N}$. The formula is the family $Fm(\textbf{BM})^{gd}:=\{Fm(\textbf{BM})^{gd}_{s_{1}}, Fm(\textbf{BM})^{gd}_{s_{2}}\}$ where, 
\[\varphi_{s_{1}}:=p_{s_{1}}\mid \varphi_{s_{1}}\wedge\varphi_{s_{1}}\mid \varphi_{s_{1}}\vee\varphi_{s_{1}}\mid \neg \varphi_{s_{1}}\mid \gboxm{n}\varphi_{s_2}\mid \gboxnm{n}\varphi_{s_2}\] 
\[\varphi_{s_{2}}:=p_{s_{2}}\mid \varphi_{s_{2}}\wedge\varphi_{s_{2}}\mid \varphi_{s_{2}}\vee\varphi_{s_{2}}\mid \neg \varphi_{s_{2}}\mid \gboxg{n}\varphi_{s_1}\mid \gboxng{n}\varphi_{s_1}.\]

For a context-based  model $\mathfrak{M}:=(G, M,  I, v)$, and given $g\in G$ and $m\in M$, recalling the definition of right and left neighborhoods $I_{g\sbullet}:=\{m\in M\mid Igm\}$ and $I_{\sbullet m}:=\{g\in G\mid Igm\}$, we further define  $I_{g\sbullet}(\psi):=\{m\in I_{g\sbullet}\mid m\models_{s_{2}} \psi\}$ and $I_{\sbullet m}(\varphi):=\{g\in I_{\sbullet m}\mid g\models_{s_{1}} \varphi\}$ for any formulas $\varphi\in Fm(\textbf{BM})_{s_{1}}^{gd}$ and $\psi\in Fm(\textbf{BM})_{s_{2}}^{gd}$. Also, we define the same notations for the complemented relation $\overline{I}$.
The  satisfaction of a formula is then defined as follows 

\begin{definition}
\label{satisfictiongraded}
    {\rm Let $\mathfrak{M}:=(G, M, I, v)$  be a model, then for $m\in M$ and $g\in G$, 
    \bla
        \item $\mathfrak{M},m\models_{s_{2}} \gboxg{n}\varphi$ iff $|I_{\sbullet m}(\neg\varphi)|\leq n$,
        \item $\mathfrak{M},m\models_{s_{2}} \gboxng{n}\varphi$ iff  $|\overline{I}_{\sbullet m}(\varphi)|\leq n$,
        \item $\mathfrak{M},g\models_{s_{1}} \gboxm{n}\psi$ iff $|I_{g\sbullet}(\neg\psi)|\leq n$, 
        \item  $\mathfrak{M},g\models_{s_{1}} \gboxnm{n}\psi$ iff $|\overline{I}_{g\sbullet}(\psi)|\leq n$.
    \ela}
\end{definition}
Intuitively, a property $m$ satisfies the formula $\gboxg{n}\varphi$ if the number of objects with the property but not satisfying $\varphi$ is at most $n$ and it satisfies the formula $\gboxng{n}\varphi$ if the number of objects without the property but satisfying $\varphi$ is at most $n$. The following example shows a concrete interpretation of the graded modalities in the application context of our running example.

\begin{example}
\label{graderexample}
    {\rm Recalling the context of transaction data $(G,M,I)$ in Example \ref{example1},  where $G$ and $M$ are sets of customers and product items, respectively and for $g\in G$ and $m\in M$, $gIm$ means that the customer $g$ has bought the item $m$, let $\mathfrak{M}=(G,M, I, v)$ be a model based on the context and let $\varphi\in Fm(\textbf{BM})^{gd}_{s_{1}}$ and $\psi\in Fm(\textbf{BM})^{gd}_{s_{2}}$ remain the same as in Example \ref{example1}. Then,  the interpretation of some graded modal formulas in $\mathfrak{M}$ are as follows.
   \bla
			\item For $g\in G$, $\mathfrak{M}, g\models_{s_{1}} \gboxm{n}\psi$ means that the number of non-electronic products bought by $g$  is no more than $n$. When $n$ is a small number (e.g. $n=1$), this may mean that $g$ rarely bought non-electronic products. However, the interpretation is sometimes misleading. For example, if the customer $g$ has bought only one item in the transaction data, then we can also say that he bought non-electronic products only. The two readings of the formula  $\gboxm{1}\psi$ seem in conflict with each other. To  resolve this, we will need to use the weighted modalities to be introduced in next subsection.
		\item For $m\in M$, $\mathfrak{M}, m\models_{s_{2}} \gboxg{n}\varphi$ means that the number of customers buying $m$ and outside the 30-50 age group is no more than $n$.
       \item For $g\in G$, $\mathfrak{M}, g\models_{s_{1}} \gboxnm{n}\psi$ means that  the number of electronic products not bought by $g$  is no more  than $n$. 
        \item For $m\in M$, $\mathfrak{M}, m\models_{s_{2}} \gboxng{n}\varphi$  means that  the number of customers in the 30-50 age group not buying $m$ is no more than $n$. When $n$ is small, this may mean that $m$ is a popular product among customers in this age group. Of course, the interpretation makes sense only when the data set contains enough number of customers in the group. 
   \ela }
\end{example}

We can also define several derived modalities $\ogboxg{n}\varphi:=\gboxng{n}\neg\varphi$, $\ogboxng{n}\varphi:=\gboxg{n}\neg\varphi$, $\ogboxm{n}\psi:=\gboxnm{n}\neg\psi$, $\ogboxnm{n}\psi:=\gboxm{n}\neg\psi$, and their semantics follows from the definition as follows immediately.
\begin{proposition}
    {\rm For a model $\mathfrak{M}:=(G, M, I, v)$, $g\in G$ and $m\in M$, we have the following.
    \bla
    \item   $\mathfrak{M},m\models_{s_{2}} \ogboxg{n}\varphi$ iff $|\overline{I}_{\sbullet m}(\neg\varphi)|\leq n$ and $\mathfrak{M},m\models \ogboxng{n}\varphi$ iff  $|I_{\sbullet m}(\varphi)|\leq n$.
     \item   $\mathfrak{M},g\models_{s_{1}} \ogboxm{n}\psi$ iff $|\overline{I}_{g\sbullet}(\neg\psi)|\leq n$ and $\mathfrak{M},g\models \ogboxnm{n}\psi$ iff $|I_{g\sbullet}(\psi)|\leq n$.
    \ela}
\end{proposition}

\begin{example}
\label{exm4}
	{\rm Continuing with Example \ref{graderexample}, let $\varphi\in Fm(\textbf{BM})^{gd}_{s_{1}}$ and $\psi\in Fm(\textbf{BM})^{gd}_{s_{2}}$ remain unchanged. Then, the intuitive meanings of  the formulas $\ogboxg{n}\varphi$, $\ogboxng{n}\varphi$ $\ogboxm{n}\psi$, $\ogboxnm{n}\psi$ are as follows:
		\bla
			\item For $g\in G$, $\mathfrak{M}, g\models_{s_{1}} \ogboxm{n}\psi$ means that the number of non-electronic products not bought by $g$ is at most $n$.
			\item For $m\in M$, $\mathfrak{M}, m\models_{s_{2}} \ogboxg{n}\varphi$, the number of  customers  not in the 30 to 50 age group and not  buy $m$ is less than $n$.

          \item For $g\in G$, $\mathfrak{M}, g\models_{s_{1}} \ogboxnm{n}\psi$,  $g$ the number of electronic products bought by $g$ is less than $n$.
          
			\item For $m\in M$, $\mathfrak{M}, m\models_{s_{2}} \ogboxng{n}\psi$, the number of  customers   in the 30 to 50 age group and   buy $m$ is less than $n$.
		\ela}		
	\end{example}
Now we define duals of graded modal operators. $\gdiamondg{n}\varphi:=\neg\gboxg{n}\neg \varphi$, $\nabla_\mathfrak{o}^{n}\varphi:=\neg\gboxng{n}\neg \varphi$, $\gdiamondm{n}\psi:=\neg\gboxm{n}\neg \psi$ and $\nabla_\mathfrak{p}^{n}\psi:=\neg\gboxnm{n}\neg \psi$
Then, the proposition below follows immediately from the definition.
\begin{proposition}
    {\rm  For a model $\mathfrak{M}:=(G, M,I, v)$, $g\in G$ and $m\in M$, we have the following results.
    
    \bla
    \item   $\mathfrak{M},m\models_{s_{2}} \gdiamondg{n}\varphi$ iff $|I_{\sbullet m}(\varphi)|> n$ and $\mathfrak{M},m\models_{s_{2}} \nabla_\mathfrak{o}^{n}\varphi$ iff  $|\overline{I}_{\sbullet m}(\neg\varphi)|> n$.

     \item   $\mathfrak{M},g\models_{s_{1}} \gdiamondm{n}\psi$ iff $|I_{g\sbullet}(\psi)|> n$ and $\mathfrak{M},g\models_{s_{1}} \nabla_\mathfrak{p}^{n}\psi$ iff $|\overline{I}_{g\sbullet}(\neg\psi)|> n$.
    \ela}
\end{proposition}
From the proposition, we can see that the diamond-style modalities have a more natural reading than their duals. For example, $\mathfrak{M},m\models_{s_{2}} \gdiamondg{n}\varphi$ means that the number of objects having property $m$ and satisfying $\varphi$ is greater than $n$. This is precisely to say that the support of the rule `` if an object has property $m$, then it satisfies $\varphi$'' is above the threshold $n$. Indeed, in most existing work on graded modal logic~\cite{BalestriniF99,Caro88,Cerrato90,Cerrato94,DeCaro,Fattorosi-Barnaba88,Fattorosi-BarnabaG95}, $\Diamond^n$ is usually taken as primitive operators, whereas $\Box^n$ is defined by duality. This is why the semantics of $\Box^n\varphi$ is defined indirectly with the number of worlds not satisfying $\varphi$. Despite this, the box-style graded modal formulas can still represent some useful quantitative  information of formal contexts as shown in Example~\ref{graderexample}. Hence, to keep uniform with the presentation of our main logic {\bf BM}, we choose the box-style  modalities as primitive operators of its graded extension and simply point out that their dual diamond-style modalities can better express supports of rules induced from the data mining context.

We can also define dual modalities of $\ogboxg{n}$, $\ogboxng{n}$,  $\ogboxm{n}$, $\ogboxnm{n}$ in the same way. Furthermore, with graded modalities, we can define modal operators that represent the exact number of objects/properties satisfying a formula $\varphi$ as follows: $\gdiamondg{n}!\varphi:=\gdiamondg{n-1}\varphi\wedge\neg\gdiamondg{n}\varphi$, $\gdiamondm{n}!\psi:=\gdiamondm{n-1}\psi\wedge\neg\gdiamondm{n}\psi$, $\overline{\Diamond}_\mathfrak{o}^{n}!\varphi:=\overline{\Diamond}_\mathfrak{o}^{n-1}\varphi\wedge\neg\overline{\Diamond}_\mathfrak{o}^{n}\varphi$ and $\overline{\Diamond}_\mathfrak{p}^{n}!\psi:=\overline{\Diamond}_\mathfrak{p}^{n-1}\psi\wedge\neg\overline{\Diamond}_\mathfrak{p}^{n}\psi$

Next, we show that the language has the expressive power to specify properties of the binary relation and its complement in a context. 
\begin{theorem}
    {\rm  Let $\mathfrak{M}:=(G, M,I, v)$ be a model. Then the following results hold.
    \bla
    \item $I$ and $\overline{I}$ are partial functions iff $\mathfrak{M}\models_{s_{1}}\neg\gdiamondm{1}\top_{s_2}$ and $\mathfrak{M}\models_{s_{1}}\neg\ogdiamondm{1}\top_{s_2}$,  respectively.
     \item $I$ and $\overline{I}$ are functions iff $\mathfrak{M}\models_{s_{1}}\gdiamondm{1}!\top_{s_2}$ and $\mathfrak{M}\models_{s_{1}}\ogdiamondm{1}!\top_{s_2}$, respectively.
     \item $I$ is an injective function  iff $\mathfrak{M}\models_{s_{1}}\gdiamondm{1}!\top_{s_2}$ and $\mathfrak{M}\models_{s_{2}} \neg\gdiamondg{1}\top_{s_1}$.
\item $\overline{I}$ is an injective function  iff $\mathfrak{M}\models_{s_{1}}\ogdiamondm{1}!\top_{s_2}$ and  $\mathfrak{M}\models_{s_{2}} \neg\ogdiamondg{1}\top_{s_1}$.
 \item $I$ is a surjective function iff $\mathfrak{M}\models_{s_{1}}\gdiamondm{1}!\top_{s_2}$ and $\mathfrak{M}\models_{s_2} \gdiamondg{0}\top_{s_1}$.
\item $\overline{I}$ is a surjective function iff $\mathfrak{M}\models_{s_{1}}\ogdiamondm{1}!\top_{s_2}$ and $\mathfrak{M}\models_{s_2} \ogdiamondg{0}\top_{s_1}$.
\item $I$ is a bijective function iff $\mathfrak{M}\models_{s_{2}}\gdiamondg{1}!\top_{s_1}$ and $\mathfrak{M}\models_{s_{1}} \gdiamondm{1}!\top_{s_2}$.
 \item $\overline{I}$ is a bijective function iff $\mathfrak{M}\models_{s_{2}}\ogdiamondg{1}!\top_{s_1}$ and $\mathfrak{M}\models_{s_{1}} \ogdiamondm{1}!\top_{s_2}$.
    \ela}
\end{theorem}
\begin{proof}
    Let $\mathfrak{M}:=(G, M,  I, v)$ be a model. Then, 
    \bla
        \item $I$ is a partial function iff for all $g\in G$, $I_{g\sbullet}=\emptyset$ or $I_{g\sbullet}$ is a singleton, that is, $|I_{g\sbullet}(\top_{s_2})|\leq 1$  which is equivalent to $\mathfrak{M}, g\models\neg\gdiamondm{1}\top_{s_2}$.  Analogously,  we can prove the case for $\overline{I}$.
         \item   $I$ is  function iff all $g\in G$,  $I_{g\sbullet}$ is a singleton which is equivalent to $|I_{g\sbullet}(\top_{s_2})|=1$. That is, $\mathfrak{M}\models \gdiamondm{1}!\top_{s_2}$. The proof for the case of $\overline{I}$ is similar.
        \item  $I$ is an injective function iff for all $g\in G$, $I_{g\sbullet}$ is a singleton and for all $m\in M$, $|I_{\sbullet m}|\leq 1$. This is indeed equivalent to  $\mathfrak{M}\models_{s_{1}}\gdiamondm{1}!\top_{s_2}$ and $\mathfrak{M}\models_{s_{2}} \neg\gdiamondg{1}\top_{s_1}$.
        \item The proof is similar to that for (c).
        \item $I$ is a surjective function iff it is a function and for all $m\in M$, $I_{\sbullet m}\not= \emptyset$. The former is equivalent to $\mathfrak{M}\models \diamondm_{1}!\top_{s_2}$ by (b) and the latter is equivalent to $|I_{\sbullet m}(\top_{s_1})|>0$, which is exactly the truth condition of $\mathfrak{M}\models_{s_2} \gdiamondg{0}\top_{s_1}$.
        \item The proof is similar to that for (e).
        \item $I$ is a bijective function iff for any $g\in G$ and $m\in M$, both $I_{g\sbullet}$ and $I_{\sbullet m}$ are singletons. This is equivalent to $|I_{g\sbullet}(\top_{s_2})|=1$ and $|I_{\sbullet m}(\top_{s_1})|=1$, which are exactly truth conditions of  $\mathfrak{M}\models_{s_{1}} \gdiamondm{1}!\top_{s_2}$ and $\mathfrak{M}\models_{s_{2}}\gdiamondg{1}!\top_{s_1}$, respectively.
     \item The proof is similar to that for (g).
     \ela
\end{proof}
 
\subsection{Two sorted weighted modal logic}
\label{TSWML}
In this section, we extend the two-sorted Boolean modal logic to the weighted framework~\cite{LEGASTELOIS2017341}. For a two-sorted family $P:=\{P_{s_{1}}, P_{s_{2}}\}$ of propositional variables, we consider a two-sorted family of formulas $Fm(\textbf{BM})^{wt}:=\{Fm(\textbf{BM})^{wt}_{s_{1}}, Fm(\textbf{BM})^{wt}_{s_{2}}\}$ as follows: 
\[\varphi_{s_{1}}:=p_{s_{1}}\mid \varphi_{s_{1}}\wedge\varphi_{s_{1}}\mid \varphi_{s_{1}}\vee\varphi_{s_{1}}\mid \neg \varphi_{s_{1}}\mid \gboxm{c}\psi_{s_2}\mid \gboxnm{c}\psi_{s_2}\mid \ogboxm{c}\psi_{s_2}\mid\ogboxnm{c}\psi_{s_2}\] 
\[\psi_{s_{2}}:=p_{s_{2}}\mid \psi_{s_{2}}\wedge\psi_{s_{2}}\mid \psi_{s_{2}}\vee\psi_{s_{2}}\mid \neg \psi_{s_{2}}\mid \gboxg{c}\varphi_{s_1}\mid \gboxng{c}\varphi_{s_1}\mid \ogboxg{c}\varphi_{s_1}\mid\ogboxng{c}\varphi_{s_1},\]
where $c\in [0,1]$ is any rational number.

For any formulas $\varphi\in \textbf{BM})^{wt}_{s_{1}}$ and $\psi\in \textbf{BM})^{wt}_{s_{2}}$, we define the satisfaction of weighted modal formulas as follows:
\begin{definition}
\label{satisfictionwighted}
    {\rm Let $\mathfrak{M}:=(G, M, I, v)$  be a finite model. Then, for $m\in M$ and $g\in G$,
\begin{center}
\begin{minipage}[t]{65mm}
 \begin{inparaitem}
\item[(a)] $\mathfrak{M},m\models_{s_{2}} \gboxg{c}\varphi$ iff $\frac{|I_{\sbullet m}(\varphi)|}{|I_{\sbullet m}|}\geq c$,\\
    \item[(b)] $\mathfrak{M},m\models_{s_{2}} \gboxng{c}\varphi$ iff $\frac{|I_{\sbullet m}(\varphi)|}{|[\!\![\varphi]\!\!]|}\geq c$,\\
    \item[(c)] $\mathfrak{M},m\models_{s_{2}} \ogboxg{c}\varphi$ iff $\frac{|\overline{I}_{\sbullet m}(\varphi)|}{|\overline{I}_{\sbullet m}|}\geq c$,\\
    \item[(d)] $\mathfrak{M},m\models_{s_{2}} \ogboxng{c}\varphi$ iff $\frac{|\overline{I}_{\sbullet m}(\varphi)|}{|[\!\![\varphi]\!\!]|}\geq c$,
\end{inparaitem}   
\end{minipage} 
\begin{minipage}[t]{65mm}
 \begin{inparaitem}
    \item[(e)] $\mathfrak{M},g\models_{s_{1}} \gboxm{c}\psi$ iff $\frac{|I_{g\sbullet}(\psi)|}{|I_{g\sbullet}|}\geq c$,\\
    \item[(f)]  $\mathfrak{M},g\models_{s_{1}} \gboxnm{c}\psi$ iff $\frac{|I_{g\sbullet}(\psi)|}{|[\!\![\psi]\!\!]|}\geq c$,\\
    \item[(g)] $\mathfrak{M},g\models_{s_{1}} \ogboxm{c}\psi$ iff $\frac{|\overline{I}_{g\sbullet}(\psi)|}{|\overline{I}_{g\sbullet}|}\geq c$, \\ 
    \item[(h)]  $\mathfrak{M},g\models_{s_{1}} \ogboxnm{c}\psi$ iff $\frac{|\overline{I}_{g\sbullet}(\psi)|}{|[\!\![\psi]\!\!]|}\geq c$.
\end{inparaitem}   
\end{minipage}
\end{center}}
\end{definition}
Note that graded and weighted modalities have close relationship. For instance, suppose $|I_{\sbullet m}|=k\geq n$, we have $\mathfrak{M},m\models_{s_{2}} \gboxg{n}\varphi$ in graded logic iff $|I_{\sbullet m}(\neg\varphi)|\leq n$ iff $\frac{|I_{\sbullet m}(\varphi)|}{|I_{\sbullet m}|}\geq 1-n/k$ iff $\mathfrak{M},m\models_{s_{2}} \gboxg{(1-n/k)}\varphi$ in weighted logic. However, these two kinds of modalities are still not equivalent because, to interpret weighted modalities, we have to assume that the model is finite (or at least each object is connected to finite number of attributes and vice versa), but there is no such restriction on the interpretation of graded modalities. In addition, the above-mentioned mutual derivation of graded and weighted modal formulas is possible only when the cardinality of left or right neighborhoods (i.e. $k$ above) is available. However, we do not always have (or need) such information when interpreting graded modalities.

Unlike the semantics of graded modalities, the box-style weighted modalities are interpreted with worlds satisfying the formula directly. This also complies with semantics of the existing weighted logic~\cite{LEGASTELOIS2017341}. As a consequence, we can express important quantitative information such as confidence or accuracy of rules induced from the context with primitive modalities of the weighted extension. For example, the satisfaction of  $\gboxg{c}\varphi$ in $m$ intuitively means that at least $c\cdot 100\%$ of objects with property $m$ satisfy $\varphi$. In other words, this means that the confidence of the rule ``if an object has property $m$, then it satisfies $\varphi$'' reaches the threshold $c$. The following example further shows that weighted modal formulas can represent a variety of quantitative information in a formal context.

\begin{example}
{\rm Continuing with our running example, let $\mathfrak{M}=(G,M, I, v)$ be a finite model based on the formal context and let $\varphi\in Fm(\textbf{BM})^{wt}_{s_{1}}$ and   $\psi\in Fm(\textbf{BM})^{wt}_{s_{2}}$ remain the same as in Example \ref{example1}. Then,  the weighted modal formulas may be interpreted in $\mathfrak{M}$ as follows.
\bla
\item For $g\in G$, $\mathfrak{M}, g\models_{s_{1}} \gboxm{c}\psi$ means that  at least $c\cdot 100\%$ of items bought by $g$ is electronic products.
\item For $m\in M$, $\mathfrak{M}, m\models_{s_{2}} \gboxg{c}\varphi$  means that at least $c\cdot 100\%$ of customers  buying the item $m$ is in the 30-50 age group.
\item For $g\in G$, $\mathfrak{M}, g\models_{s_{1}} \gboxnm{c}\psi$,  $g$ bought at least  $c\cdot 100\%$ of  electronic product items.
\item For $m\in M$, $\mathfrak{M}, m\models_{s_{2}} \gboxng{c}\varphi$, at least  $c\cdot 100\%$ of customers in  the 30-50 age group  bought the product $m$. 
\item For $g\in G$, $\mathfrak{M}, g\models_{s_{1}} \ogboxm{c}\psi$ means that at least $c\cdot 100\%$ of items {\em not} bought by $g$ is electronic products.
\item For $m\in M$, $\mathfrak{M}, m\models_{s_{2}} \ogboxg{c}\varphi$  means that at least $c\cdot 100\%$ of customers  {\em not} buying the item $m$ is in the 30-50 age group.
\item For $g\in G$, $\mathfrak{M}, g\models_{s_{1}}\ogboxnm{c}\psi$,  $g$ bought at most $(1-c)\cdot 100\%$  of  electronic product items.
\item For $m\in M$, $\mathfrak{M}, m\models_{s_{2}} \ogboxng{c}\varphi$, at least $c\cdot 100\%$ of customers in the 30-50 age group does not buy the item $m$.
		\ela
When the formal context contains a large amount of transaction data so that the samples are representative enough, we can given the weights a probabilistic interpretation. For example, (a) may be interpreted as
follows: the conditional probability of an item being an electronic product given that $g$ bought it is at least  $c$. }
\end{example}

We show that the two-sorted weighted modal logic generalize {\bf BM} by defining an embedding translation $\tau=\{\tau_1,\tau_2\}:Fm(\textbf{BM})\rightarrow Fm(\textbf{BM})^{wt}$ as follows.
\begin{itemize}
    \item $\tau_1(p_{s_1})=p_{s_1}$ and $\tau_2(p_{s_2})=p_{s_2}$.
    \item $\tau$ is homomorphic with respect to propositional connectives.
    \item $\tau_1({\sf X}_\mathfrak{p}\psi)={\sf X}_\mathfrak{p}^1\tau_2(\psi)$ for ${\sf X}=\Box,\boxminus,\overline{\Box},
   \overline{\boxminus}$.
   \item $\tau_2({\sf X}_\mathfrak{o}\varphi)={\sf X}_\mathfrak{o}^1\tau_1(\varphi)$ for ${\sf X}=\Box,\boxminus,\overline{\Box},
   \overline{\boxminus}$.   
\end{itemize}
\begin{proposition}
    {\rm Let $\mathfrak{M}=(G,M,I,v)$ be a finite model and let $\varphi\in  Fm(\textbf{BM})_{s_{1}}$ and $\psi\in  Fm(\textbf{BM})_{s_{2}}$ be two formulas. Then, for any $g\in G$ and $m\in M$
    \bla
    \item $\mathfrak{M},g\models\varphi$ iff  $\mathfrak{M},g\models\tau_1(\varphi)$.
    \item $\mathfrak{M},m\models\psi$ iff  $\mathfrak{M},m\models\tau_2(\psi)$.
  \item $\models{\sf X}^{0}\tau_1(\varphi)$ for ${\sf X}\in \{\boxg,\boxng\}$ and $\models{\sf Y}^{0}\tau_2(\psi)$ for ${\sf Y}\in\{\boxm, \boxnm\}$
    \ela}
\end{proposition}
\begin{proof}
   It follows from the definitions.
\end{proof}

\begin{proposition}
    {\rm Let $\varphi\in  Fm(\textbf{BM})_{s_{1}}^{wt}$ and $\psi\in  Fm(\textbf{BM})_{s_{2}}^{wt} $. Then $\models_{s_{2}} \gboxg{1}\neg \varphi\leftrightarrow\ogboxng{1}\varphi$, $\models_{s_{1}}\gboxm{1}\neg\psi\leftrightarrow\ogboxnm{1}\psi$, $\models_{s_{2}} \ogboxg{1}\neg \varphi\leftrightarrow\gboxng{1}\varphi$, and $\models_{s_{1}}\ogboxm{1}\neg\psi\leftrightarrow\gboxnm{1}\psi$.}
\end{proposition}
\begin{proof}
    Let $\mathfrak{M}:=(G,M,I, v)$ be a finite model and $m\in M$. Then, $\mathfrak{M}, m\models_{s_{2}} \gboxg{1}\neg\varphi$ iff $I_{\sbullet m}=I_{\sbullet m}(\neg\varphi)$ iff $I_{\sbullet m}\subseteq [\!\![\neg \varphi]\!\!]$ iff $[\!\![\varphi]\!\!]\subseteq \overline{I}_{\sbullet m}$ iff.  $\mathfrak{M}, m\models_{s_{2}} \ogboxng{1}\varphi$.
   The proof of other cases is similar.
\end{proof}
The preceding two propositions also show that, as in the case of two-sorted Boolean modal logic, several weighted modal operators are definable with each other in the case of extreme weights. However, in the next proposition, we show that such mutual definability does not hold generally. 
\begin{proposition}{\rm
$\gboxng{c}\varphi\leftrightarrow\ogboxg{d}\neg\varphi$, $\gboxnm{c}\psi\leftrightarrow\ogboxm{d}\neg\psi$, $\ogboxng{c}\varphi\leftrightarrow\gboxg{d}\neg\varphi$ and  $\ogboxnm{c}\psi\leftrightarrow\gboxm{d}\neg\psi$ all fails unless $c=d=1$ or $c=d=0$.  } 
\end{proposition} 
\begin{proof}
 {\rm We show that $\gboxng{c}\varphi\leftrightarrow\ogboxg{d}\neg\varphi$ has a counter-model for  any $c, d\in [0,1]$ except $c=d=1$ or $c=d=0$. To do that, let us consider the following contingency table, where each $k_i (1\leq i\leq 4)$ denotes the number of objects for the corresponding entry. For example, $k_1$ is the cardinality of $I_{\sbullet m}(\varphi)$.
  \begin{table}[h]
		\centering
		\caption {A contingency table}\label{CU}
		\begin{tabular}{|c|c|c|}
				\hline
				&$[\!\![\varphi]\!\!]$ & $[\!\![\neg \varphi]\!\!]$\\ \hline
				
				$I_{\sbullet m}$ & $k_1$ & $k_2$\\
				\hline
				$\overline{I}_{\sbullet m}$ & $k_3$ & $k_4$\\
				\hline
			\end{tabular}
\end{table}
Then, for given $c,d\in[0,1]$ that is not the exceptional case, we can construct a model $(G,M,I,v)$ and find an $m\in M$ to falsify the equivalence.  By definition, 
$\mathfrak{M},m\models \gboxng{c}\varphi$ and $\mathfrak{M},m\models\ogboxg{d}\neg\varphi$  mean  $\frac{|I_{\sbullet m}(\varphi)|}{|[\!\![\varphi]\!\!]|}\geq c$ and $\frac{|\overline{I}_{\sbullet m}(\neg\varphi)|}{|\overline{I}_{\sbullet m}|}\geq d$, respectively. That is, $\frac{k_1}{k_1+k_3}\geq c$ and $\frac{k_4}{k_3+k_4}\geq d$. Thus, it suffices to show that for any non-exceptional $c,d$, we can always choose appropriate $k_i (i=1,3,4)$ to make one of the inequalities true and the other false. We prove it by simple case analysis.
\begin{itemize}
    \item $c=0$: then $\frac{k_1}{k_1+k_3}\geq c$ is trivially true. Hence, as $d>0$, we can set $k_4=1$ and choose $k_3$ large enough such that $\frac{1}{k_3+1}< d$. The proof for the case of $d=0$ is similar.
    \item $c=1$: setting $k_3=1$  falsifies $\frac{k_1}{k_1+k_3}\geq c$  and as $d<1$, we can choose $k_4$ large enough to make $\frac{k_4}{k_4+1}\geq d$. Analogously, we can prove the case of $d=1$.
    \item $c,d\in(0,1)$: by setting $k_3=1$ and $k_4=0$ and choosing $k_1$ large enough to make $\frac{k_1}{k_1+1}\geq c$, we can construct the counter-model accordingly. 
\end{itemize}}
\end{proof}
Next, we define the dual modal operators $\gdiamondg{c}\varphi:= \neg\gboxg{c}\neg \varphi$, $\gdiamondm{c}\psi:= \neg\gboxm{c}\neg \psi$, $\nabla_\mathfrak{o}^{c}\varphi:=\neg\gboxng{c}\neg\varphi$ and $\nabla_\mathfrak{p}^{c}\psi:=\neg\gboxnm{c}\neg\psi$. 

\begin{proposition}
{\rm Let $\mathfrak{M}:=(G, M, I, v)$  be a finite model and let $g\in G$ and $m\in M$. Then for $\varphi\in Fm(\textbf{BM})^{wt}_{s_{1}}$ and $\psi\in Fm(\textbf{BM})^{wt}_{s_{2}}$ the following holds.
    \bla
         \item $\mathfrak{M},m\models_{s_{2}} \gdiamondg{c}\varphi$ iff $\frac{|I_{\sbullet m}(\varphi)|}{|I_{\sbullet m}|}> 1-c$ and   $
         \mathfrak{M},m\models_{s_{2}} \nabla_\mathfrak{o}^{c}\varphi$ iff $\frac{|I_{\sbullet m}(\neg\varphi)|}{|[\!\![\varphi]\!\!]|} < c$.
        \item $\mathfrak{M},g\models_{s_{1}} \gdiamondm{c}\psi$ iff $\frac{|I_{g\sbullet}(\psi)|}{|I_{g\sbullet}|}> 1-c$ and $\mathfrak{M},g\models_{s_{1}} \nabla_\mathfrak{p}^{c}\psi$ iff $\frac{|I_{g\sbullet}(\neg\psi)|}{|[\!\![\psi]\!\!]|}< c$.
         \ela}
   
\end{proposition}
\begin{proof}
    {\rm The proof follows from Definition \ref{satisfictionwighted}.}
\end{proof}

 A similar result can be proved for the definable modalities $\ogdiamondg{c}\varphi:= \neg\ogboxg{c}\neg \varphi$, $\ogdiamondm{c}\psi:= \neg\ogboxm{c}\neg \varphi$, $\overline{\nabla}_\mathfrak{o}^{c}\varphi:=\neg\ogboxng{c}\neg\varphi$ and $\overline{\nabla}_\mathfrak{p}^{c}\psi:=\neg\ogboxnm{c}\neg\psi$. In addition, we have the following straightforward results for weighted modalities.
\begin{proposition}\label{waightedneedproflattic}
{\rm
For $\varphi\in Fm(\textbf{BM})^{wt}_{s_{1}}$ and  $\psi\in Fm(\textbf{BM})^{wt}_{s_{2}}$, and $c, d\in [0,1]$ such that $c\geq d$,
\bla
         \item $\models_{s_{2}} {\sf X}^{c}\varphi\rightarrow{\sf X}^{d}\varphi$ for ${\sf X}=\boxg,\boxng, \oboxg, \oboxng$. 
          \item  $\models_{s_{1}}{\sf Y}^{c}  \psi\rightarrow{\sf Y}^{d}\psi$ for ${\sf Y}=\boxm,\boxnm, \oboxm, \oboxnm$.  
     \ela}
 \end{proposition}
Also, the following classical rules are still valid when the weight is equal to 1.

\begin{proposition}
    {\rm For $\varphi_{1}, \varphi_{2}\in Fm(\textbf{BM})^{wt}_{s_{1}}$ and  $\psi_{1}, \psi_{2}\in Fm(\textbf{BM})^{wt}_{s_{2}}$, the following hold.
    \bla
    \item If $\models_{s_{1}} \varphi_{1}\rightarrow \varphi_{2}$ then $\models_{s_{2}} {\sf X}^1\varphi_{2}\rightarrow  {\sf X}^1\varphi_{1}$ for ${\sf X}=\boxng,\oboxng$.
    \item If $\models_{s_{2}} \psi_{1}\rightarrow \psi_{2}$ then $\models_{s_{1}} {\sf Y}^1\psi_{2}\rightarrow {\sf Y}^1\psi_{1}$  for ${\sf Y}=\boxnm,\oboxnm$.
    \ela}
\end{proposition}
\begin{proof}
    {\rm All proofs are similar. Hence, we only prove the case of $\boxng$.  Let $\mathfrak{M}:=(G, M, I, v)$ be a model and $m\in M$. Then $\models_{s_{1}} \varphi_{1}\rightarrow \varphi_{2}$ implies $[\!\![\varphi_{1}]\!\!]\subseteq [\!\![\varphi_{2}]\!\!]$. By definition,  $\mathfrak{M}, m\models_{s_{2}} \gboxng{1}\varphi_{2}$ implies $\frac{|I_{\sbullet m}\cap[\!\![\varphi_{2}]\!\!]|}{|[\!\![\varphi_{2}]\!\!]|}=1$, which in turn implies $[\!\![\varphi_{1}]\!\!]\subseteq [\!\![\varphi_{2}]\!\!]\subseteq  I_{\sbullet m}$.  Hence, we have also $\mathfrak{M},m\models_{s_{2}} \gboxng{1}\varphi_{1}$.}  
\end{proof}
However, except for boundary cases, the axioms of {\bf BM} cannot be easily generalized to the weighted modal logic. 
\begin{proposition}
    {\rm $\not{\models}_{s_{1}} p\rightarrow \gdiamondm{c}\gboxg{d} p$ and $\not{\models}_{s_{1}} q\rightarrow \gboxm{c}\gdiamondg{e} q$ for any $c,d\in(0,1]$ and $e\in [0, 1)$.}
\end{proposition}
\begin{proof}
    To show $\not{\models}_{s_{1}} p\rightarrow \gdiamondm{c}\gboxg{d} p$, we need to construct a model that refutes the formula. Let $\mathfrak{M}:=(G,M, I, v)$ be a model such that 
     \bla
     \item 
     $|G|=n$ where $n\geq \max(\frac{1}{1-e},1+ \frac{1}{d})$,
       \item  for some $g_{0}\in G$, $v_{1}(p)=v_1(q):= \{g_{0}\}$
       \item  $I_{g_{0}\cdot}=\{m_0\}$ and $I_{\sbullet m_0}=G$ for some $m_0\in M$.
     \ela
     As $\frac{|I_{\sbullet m_0}(p)|}{|I_{\sbullet m_0}|}=\frac{1}{n}$ and $d>\frac{1}{n}$, we have $\mathfrak{M}, m\models_{s_{2}} \neg \gboxg{d}p$ for all $m\in I_{g_{0}\cdot}$, which implies that  $I_{g_{0}\cdot}(\gboxg{d}p)=\emptyset$. Hence, $\frac{|I_{g_{0}\cdot}(\gboxg{d}p)|}{|I_{g_{0}\cdot}|}=0$. Therefore,  for any $c\in (0, 1]$, $\mathfrak{M}, g_{0}\not\models_{s_1} \gdiamondg{c}\gboxg{d}p.$
     
    In addition,  $\frac{|I_{\sbullet m_0}(q)|}{|I_{\sbullet m_0}|}=\frac{1}{n}\leq 1-e$ which implies that $\mathfrak{M}, m\models_{s_{2}} \neg\gdiamondg{e}q$ for any $m\in I_{g_{0}\cdot}$, i.e., $\frac{|I_{g_{0}\cdot}(\gdiamondg{e}q)|}{|I_{g_{0}\cdot}|}=0$, whence $\mathfrak{M}, g_{0}\not{\models}\gboxm{c}\gdiamondg{e}q$ for any $c\in (0, 1]$
\end{proof}
Note that the proposition does not hold for some boundary values of $c,d,e$. For example, if $d=0$, then $\gboxg{0}p$ is trivially satisfied in all $m\in M$. Hence, $\gdiamondm{c}\gboxg{d} p$ and so $p\rightarrow \gdiamondm{c}\gboxg{d} p$ are valid for any $c>0$

\begin{proposition}
    {\rm   $\not{\models}_{s_{2}} \gboxng{c}(p\wedge\neg q)\rightarrow (\gboxng{d}\neg p\rightarrow \gboxng{e}\neg q)$ for $c,d\in[0,1)$ and $e\in(0,1]$ and
     $ \not{\models}_{s_{1}} p\rightarrow \gboxnm{c}\gboxng{d}p$ for $c, d\in(0,1]$}
\end{proposition}
\begin{proof}
    First, to construct a  model $\mathfrak{M}=(G, M, I, v)$ and a property $m_0\in M$ refuting  $\gboxng{c}(p\wedge\neg q)\rightarrow (\gboxng{d}\neg p\rightarrow \gboxng{e}\neg q)$, let us consider the following three-way contingency table specifying the cardinality of respective truth sets, e.g., $k_1=|I_{\sbullet m_0}(p\wedge q)|$, $|I_{\sbullet m_0}(p\wedge\neg q)|=1$, and so on.
    \[\begin{array}{|c|c|c|c|c|c|c|}\cline{1-3}\cline{5-7}
       I_{\sbullet m_{0}} & q &\neg q& & \overline{I}_{\sbullet m_{0}}&q&\neg q\\ \cline{1-3}\cline{5-7}
        p & k_1 &1& &p &k_2&0\\
        \neg p&k_3&0& &\neg p& 0&k_4\\ \cline{1-3}\cline{5-7}
    \end{array}\]
    
    Then, $\gboxng{c}(p\wedge\neg q)$, $\gboxng{d}\neg p$, and $\neg\gboxng{e}\neg q$ corresponds to 
    $\frac{|I_{\sbullet m_{0}}(p\wedge\neg q)|}{|[\!\![p_\wedge\neg q]\!\!]|}=\frac{1}{1+0}=1\geq c$, $\frac{|I_{\sbullet m_{0}}(\neg p)|}{|[\!\![\neg p]\!\!]|}=\frac{k_3}{k_3+k_4}\geq d$ and  $\frac{|I_{\sbullet m_0}(\neg q)|}{|[\!\![\neg q]\!\!]|}=\frac{1}{1+k_4}< e$, respectively. Hence, it suffices to choose $k_4>\frac{1}{e}-1$ and $k_3\geq \frac{k_4\cdot d}{1-d}$. Obviously, for any $c,d\in[0,1)$ and $e\in(0,1]$, there are always such values of $k_4$ and $k_3$ satisfying the condition. Therefore, the counter-model exists.

Second, to show $ \not{\models}_{s_{1}} \varphi\rightarrow \gboxnm{c}\gboxng{d}\varphi$, we consider a model $\mathfrak{M}=(G, M, I, v)$ where 
    \bla
    \item $|G|=n+1$ such that $\frac{1}{n}< d$.
    \item For $g_{0}, g_{1}\in G$ and $m_{0}\in M$, $I_{g_{0}\cdot}=\{m_{0}\}$ and $I_{\sbullet m_{0}}=\{g_{0}, g_{1}\}$.
    \item $[\!\![\varphi]\!\!]=G\setminus\{g_{1}\}$.
    \ela
   Then, $\frac{|I_{\sbullet m_{0}}(\varphi)|}{|[\!\![\varphi]\!\!]|}=\frac{1}{n}< d$ which implies that $\mathfrak{M}, m_{0}\models_{s_{2}} \neg\gboxng{d}\varphi$. So $\frac{|I_{g_{0}\cdot}( \gboxng{d}\varphi)|}{|[\!\![\gboxng{d}\varphi]\!\!]|}=0$ whence  $\mathfrak{M}, g_{0}\not{\models}_{s_{1}} \gboxnm{c}\gboxng{d}\varphi$.
\end{proof}
These negative results show that the straightforward generalization of several main axioms in {\bf BM} are not valid in the weighted setting. Hence, the complete axiomatization of two-sorted weighted modal logic will be a challenging issue to be addressed in our future work.   

\subsection{More general models}
While both graded and weighted extensions of {\bf BM} are interpreted in context-based models, the latter can also express information in more general models.  A straightforward generalization of formal context is to incorporate a measure into it. Specifically, a {\em context with measure} is a quadruple $\mathbb{K}=(G,M,I,\mu)$, where $(G,M,I)$ is a context and $\mu=\{\mu_1,\mu_2\}$ is a two-sorted monotonic measure such that $\mu_1:\mathcal{P}(G)\rightarrow\Re^+\cup\{0\}$ and $\mu_2:\mathcal{P}(M)\rightarrow\Re^+\cup\{0\}$ satisfy $\mu(X)\leq\mu(Y)$ if $X\subseteq Y$ and $\mu(\emptyset)=0$. Then, a model based on $\mathbb{K}$ is a 5-tuple $\mathfrak{M}:=(G, M, I, \mu, v)$, where $v$ is the valuation function defined above. The weighted modal logic formulas are interpreted in such models by replacing the cardinality measure with $\mu$. 

For such a general framework, there exist various interesting subclasses of contexts with different types of measures\cite{hal05}. The most well-known measure arises in probability theory. When the measure $\mu$ is defined by a probability distribution, we call $(G,M,I,\mu)$ a {\em probabilistic context}. Then,  weighted modal formulas  interpreted in probabilistic context-based models 
express lower bounds on conditional probabilities. According to the semantics in Definition~\ref{satisfictionwighted}, our weighted modal logic is interpreted in a special kind of probabilistic context-based models where the probability measure is derived from the uniform distribution.

On the other hand, we can also consider measures derived from possibility distributions.  Given  an universe $W$, a {\em possibility distribution} over $W$ is a mapping $\pi:W\rightarrow [0,1]$, The possibility and necessity measures derived from $\pi$ are $\Pi, N:\mathcal{P}(W)\rightarrow[0,1]$, respectively  defined by $\Pi(X)=\sup_{x\in X}\pi(x)$ and $N(X)=1-\Pi(\overline{X})$. As above, when $\mu$ is a possibility measure, we call $(G,M,I,\mu)$ a {\em possibilistic context} and for models based on it, weighted modal formulas are interpreted as expressing lower bounds on conditional possibilities. However, unlike in the case of probability measure, where the definition of conditional probability is unique, there is no consensus on the definition of conditional possibility. In addition to the definition of conditional possibility by $\Pi(Y|X)=\frac{\Pi(X\cap Y)}{\Pi(X)}$ used here, there is an alternative definition that requires $\Pi(Y|X)$ to be the least specific (i.e. largest) solution satisfying $\min(\Pi(Y|X),\Pi(X))=\Pi(X\cap Y)$.

Although there exist some work on logical systems for reasoning about conditional probability and possibility~\cite{DautovicDO21,DautovicDO23,hal05,LiauL96,OgnjanovicRM16}, which could be starting points for axiomatizing weighted modal logic, adapting these systems to our logic is still a challenging issue because of a subtle difference between them. While our weighted modal operators are unary, existing logics all employ dyadic representation of conditional probability and possibility. Hence, in those systems, when conditional probability or possibility formulas like $P^{\geq c}(\varphi\mid\psi)$ are interpreted with respect to two events (i.e. subsets of possible worlds), both events have explicit representations (respectively, $\varphi$ and $\psi$) in the logical language. By contrast, in the interpretation of weighted modal formulas like  $\gboxg{c}\varphi$ or $\gboxng{c}\varphi$, only the event corresponding to $\varphi$ is represented in the logical language. The other event corresponds to  neighborhoods of objects or properties (i.e. $I_{g\sbullet}$ or $I_{\sbullet m}$) and does not occur in the formula explicitly. As a consequence, probability or possibility axioms that can be easily represented in existing logics may be not representable in weighted modal logic. For example, the basic axiom $P^{\geq c}(\psi\mid\top)\wedge P^{\geq d}(\varphi\mid\psi)\rightarrow P^{\geq cd}(\varphi\wedge\psi\mid\top)$  in \cite{DautovicDO23} is not representable here.

Another way to generalize a formal context is to replace its incidence relation with measures instead of adding a measure to it. Let us define a {\em measure-based context} as $\mathbb{K}=(G,M,\mu)$, where $\mu=(\mu^1,\mu^2)$ are measure functions $\mu^1: G\rightarrow (\mathcal{P}(M)\rightarrow\Re^+\cup\{0\})$ and $\mu^2: M\rightarrow (\mathcal{P}(G)\rightarrow\Re^+\cup\{0\})$ such that $\mu^1(g)$ and $\mu^2(m)$ both satisfy the above-mentioned monotonicity and boundary conditions for any $g\in G$ and $m\in M$. For simplicity, we will write $\mu_g$ and $\mu_m$ instead of $\mu^1(g)$ and $\mu^2(m)$,  respectively. This kind of context can encode information about the strength or frequency of connections between objects and properties.  In case that $\mu$ is a possibility measure, we can recover a possibility distribution $\pi$ for each $\mu_g$ and $\mu_m$ and define two extra measures by $\vartriangle\!\!(X)=\inf_{x\in X}\pi(x)$ and $\triangledown(X)=1-\inf_{x\not\in X}\pi(x)$, which are called guaranteed possibility and potential certainty, respectively\cite{DuboisSP07}. Then, in accordance with \cite{DuboisSP07}, the modalities $\Box^c, \Diamond^c, \boxminus^c,$ and $\nabla^c$ are interpreted in a way that exactly correspond to the measures $N,\Pi, \vartriangle$, and $\triangledown$, respectively. For example, if $\mathfrak{M}=(G,M,\mu,v)$ is a model based on the possibility measure context $(G,M,\mu)$, then the satisfaction of the formula $\gboxng{c}\varphi$ in $m\in M$ is defined by $\mathfrak{M},m\models_{s_{2}} \gboxng{c}\varphi$ iff $\vartriangle_m\!\!([\!\![\varphi]\!\!])\geq c$, where $\vartriangle_m$ is the guaranteed possibility derived from the possibility  measure $\mu_m$. There has been a little work on the axiomatic systems for the $(\Box,\Diamond)$-fragment of single-sorted weighted modal logic interpreted in this kind of model~\cite{qml1,qml2,qml3}. There is no essential difficulty to generalize those systems to their many-sorted version. Though, to our best knowledge, there is no existing axiomatic system for the $(\boxminus,\nabla)$-fragment of single-sorted weighted modal logic yet, its axiomatization seems quite straightforward by using a translation similar to that between {\bf KB}  and {\bf KF}. Hence, a less trivial issue on the axiomatization of logic interpreted in possibility measure-based models should be about the weighted extension of {\bf BM}, which will be left for the future work. 

\section{Comparison with existing work}\label{relatedwork}
Numerous researchers have explored non-distributive logic as a framework for formal context. This section highlights related work and compares it to our findings.

Bimbo and Dunn explore the propositional non-distributive logic \textbf{LN} in \cite{bimbo2001four}. Its formulas are defined by the grammar $\varphi:=p(p\in P)\mid\top\mid\bot\mid\varphi\wedge\varphi\mid\varphi\vee\varphi$, where $P$ is the set of propositional symbols. By using a two-sorted relational semantics,  formulas of \textbf{LN} are interpreted as formal concepts in a context.

Subsequently, \textbf{LN} is extended to the propositional lattice logic (\textbf{PLL}) and its modal extension in \cite{HCgame, CH}. In \cite{HCgame}, game-theoretic semantics for both \textbf{PLL} and  \textbf{MLL}, a modal extension of \textbf{PLL} is explored, building on the two-sorted relational frames. For \textbf{MLL}, additional relations are added to the context, allowing propositional formulas to be treated as concepts and modal formulas to be interpreted with these relations alongside $I$.

In \cite{CH}, a two-sorted model for \textbf{PLL} is defined as $\mathfrak{M} := ((G, M, I), v)$, where $(G, M, I)$ is a context and $v$ is a valuation assigning each $p$ to the extent of a concept. Satisfaction $(\models)$ and co-satisfaction $(\succ)$ are used to interpret formulas. Specifically, for an atomic formula $p$ and object $g$ in the model $\mathfrak{M}$, we have $\mathfrak{M}, g\models p$ if $g \in v(p)$, and for an property $m$, $\mathfrak{M}, m\succ p$ if $m \in v(p)^{+}$. The definitions extend to compound formulas inductively. 

Also, two-sorted residuated normal modal logic \(\mathbf{ML_{2}}\) is introduced in \cite{CH}. The language of \(\mathbf{ML_{2}}\) is the same as that of \(\mathbf{KB}\), but the proof systems differ. \(\mathbf{ML_{2}}\) has a sequent calculus for a specific class $\mathfrak{K}$ of contexts, ensuring that neighborhoods of $g$ and $m$ are non-empty for all \(g \in G\) and \(m \in M\), whereas we investigate a Hilbert-style axiomatic system for the class of all contexts. By adding the axioms \(\boxg\bot_{s_{2}} \rightarrow \bot_{s_{1}}\) and \(\boxm\bot_{s_{1}} \rightarrow \bot_{s_{2}}\) to \(\mathbf{KB}\), which correspond to certain axioms of \(\mathbf{ML_{2}}\), we can extend \(\mathbf{KB}\) to be sound and complete for the class $\mathfrak{K}$.  Additionally, we can established a translation from \(\mathbf{ML_{2}}\) to \(\mathbf{KB}\), so that a sequent is provable in \(\mathbf{ML_{2}}\) if and only if its corresponding formula is a theorem in the extended \(\mathbf{KB}\) system.  

It is shown that {\bf PLL} can be reduced to \(\mathbf{ML_{2}}\) with two translations~\cite{CH}. As a {\bf PLL} formula represents a formal concept, the translated modal formulas in \(\mathbf{ML_{2}}\) correspond to its extent and intent. This shows that a formal concept can be represented by  a pair of formulas in \(\mathbf{ML_{2}}\), just like its logical representation in {\bf KB} (via {\bf KF} and our translation $\rho$ in Section~\ref{KF}).  Since \(\mathbf{BM}\) contains \(\mathbf{KB}\) as a fragment, it is easy to see that {\bf PLL} is also reducible to \(\mathbf{BM}\). However, \(\mathbf{BM}\) is more expressive than \(\mathbf{ML_{2}}\) due to its inclusion of box and window modalities at the same time. Hence, unlike \(\mathbf{ML_{2}}\) that only deals with formal concepts, most kinds of concepts introduced in  FCA and RST  can be represented using the \textbf{BM} language. In this sense, logical representations of these concepts shown in  Section~\ref{BMLdBa} properly generalize the translations provided in \cite{CH}.

In \cite{Conradie2017167}, Conradie et al. examined non-distributive modal logic as the epistemic logic of concepts within a context. The formulas for the epistemic logic \textbf{L} are specified by the grammar: $\varphi := p(p \in P) \mid \top \mid \bot \mid \varphi \wedge \varphi \mid \varphi \vee \varphi \mid \square_{i} \varphi$ where \( P \) is the set of propositions and \( i \in Ag \), the set of agents. They define an \textit{enriched formal context} as \( \mathbb{F} = ((G, M, I), \{R_{i} \mid i \in Ag\}) \), with \( R_{i} \subseteq G \times M \), satisfying certain additional properties. A model for \textbf{L} is represented as \( \mathfrak{M} = (\mathbb{F}, v) \), where \( v \) maps propositional variables to concepts in \( (G, M, I) \). Interpretation involves two sub-relations: satisfaction \( (\models) \) and co-satisfaction \( (\succ) \). The definitions are:

\begin{center}
    \( \mathfrak{M}, g \models \square_{i} \varphi \) if for all \( m \in M \), if \( \mathfrak{M}, m \succ \varphi \), then \( g R_{i} m \).

    \( \mathfrak{M}, m \succ \square_{i} \varphi \) if for all \( g \in G \), if \( \mathfrak{M}, g \models \square_{i} \varphi \), then \( g I m \).
\end{center}

The logic is further extended to a basic modal logic, {\L}, in \cite{CONRADIE2021371} with two modalities $\Diamond$ and $\square$.

  The two-sorted logic known as \textbf{PDBL} was proposed by Howlader et al. in \cite{HOWLADER2023115}. Its formulas are defined by $\varphi:=\top\mid\bot\mid p\mid P\mid \varphi\sqcup\varphi\mid\varphi\sqcap\varphi\mid\neg\varphi\mid\lrcorner\varphi$, where $p\in\textbf{OV}$, the set of object variable and $P\in \textbf{PV}$, the set of property variable. On the aspect of algebraic semantics, \textbf{PDBL} corresponds to pure double Boolean algebra. For relational semantics, its formulas are treated as semiconcepts, following a method similar to that used in \cite{Conradie2017167,CH}. Furthermore, 
  \textbf{PDBL} is also extended to a modal system, \textbf{MPDBL}. For its relational semantics, contexts are further extended to Kripke contexts, which include two additional relations  $R$ and $S$ over $G$ and $M$, respectively. Modal formulas are then interpreted  with respect to  $R$ and $S$. 
   
The logics proposed in this paper differ from the above-mentioned ones with the genuine two-sorted syntax. The compound formulas for \textbf{KB}, \textbf{KF}, and \textbf{BM} are constructed separately for each sort. These logics are based on classical propositional logic, which is distributive. By contrast, most logics discussed in this section are non-distributive and use a single-sorted language, except for {\bf PDBL} proposed in \cite{HOWLADER2023115} and \(\mathbf{ML_{2}}\) used in \cite{CH}. While we have seen that \(\mathbf{ML_{2}}\) is essentially equivalent to our {\bf KB} system,   \textbf{PDBL} are only two-sorted in its atomic formulas, but there is no distinction of sorts for compound formulas. 
  
Although our logics and existing ones are all interpreted in a context with two universes, there is a substantial difference between their ways of interpretation. 
For \textbf{KB}, \textbf{KF}, and \textbf{BM}, the interpretation of formulas is uniformly based on the satisfaction relation in both universes. As a consequence, each formula for a specific sort corresponds to a subset of the designated universe and logical connectives are interpreted as set-theoretic operations like union, intersection, and negation. On the other hand, in existing logics, the interpretation of formulas is based on satisfaction relation with one universe and co-satisfaction relation with the other.  The interpretation assigns  formulas to elements in a specific structure, such as a concept lattice~\cite{bimbo2001four,HCgame,CH,Conradie2017167,CONRADIE2021371} or algebra of semiconcepts \cite{HOWLADER2023115}, where logical connectives reflect algebraic operations. For instance, for $\varphi$ and $\psi$ being interpreted as concepts (semiconcepts) $(A, B)$ and $(C, D)$,  $\varphi \wedge \psi$ corresponds to meet of concepts, while  $\varphi \vee \psi$ corresponds to their join.


Moreover, our modal formulas are interpreted simply with respect to the incidence relation $I$ in a context. Hence, our modalities exactly correspond to $(\cdot)^+$ and $(\cdot)^-$ operators in FCA and approximation operators in RST. On the other hand, most modal extensions of logics discussed in this section, except \(\mathbf{ML_{2}}\) that is essentially equivalent to our {\bf KB}, need to use extra relations over the context to interpret their modalities.


 In addition to the logical aspect, we also establish a representation theorem for an arbitrary dBa in terms of Boolean algebras. This result is entirely novel in relation to the work conducted in \cite{howlader2020}. In that paper, the authors proved Stone-type topological representation for two specific subclasses of dBas, i.e., fully contextual and pure dBas. Given our algebraic representation here along with the Stone representation of Boolean algebras, we will be able to prove a Stone-type representation theorem for arbitrary dBas.

Based on the comparison above, we can see that the contribution of the paper and its significance are twofold. On one hand, we provide a uniform framework {\bf BM} that can represent and reason with general information in formal contexts. As we have argued in Section~\ref{BMLdBa}, our representational formalism is much more expressive and flexible than basic FCA-related logics like {\bf PLL}. On the other hand, we address an open issue in \cite{howlader2020} by proving a novel characterization theorem of dBa in terms of Boolean algebras, which has profound implication on the algebraic study of FCA.


\section{Conclusion and future direction}
\label{conclusion}
In this paper, we present three two-sorted modal logics, {\bf KB}, {\bf KF}, and {\bf BM}, for reasoning with formal contexts. While logics {\bf KB} and {\bf KF} can represent and reason about rough and formal concepts, respectively, the logic {\bf BM} provides a uniform framework to deal with various notions of concepts in formal contexts at the same time. We prove soundness and completeness of Hilbert-style axiomatic systems for all the three logics. With the algebraic counterpart of logical semiconcepts and protoconcepts, we also prove the characterization of a (pure) dBa in terms of its underlying Boolean algebras. In addition, we use a running example to show the potential applicability of our logics to knowledge representation and reasoning in data mining. Finally, to represent more fine-grained quantitative information in some application contexts, we consider two possible extensions of our logics from a  semantic perspective.

As a promising research program, logical representation and reasoning with formal contexts is still in its relatively early stage of development. Hence, it is inevitable that many important issues remain unexplored yet. Below we list several major open problems and directions of further exploration for this program.

\noindent {\bf Axiomatization of graded and weighted modal logics}: While we have presented the semantics for graded and weighted extensions of {\bf BM}, their complete axiomatization is still lacking. On one hand, although there exists some complete axiomatic system for single-sorted graded modal logic~\cite{DeCaro}, the window modalities in {\bf BM} incur further difficulty to the proof of its completeness. Therefore, developing a complete axiomatization of two-sorted graded modal logic with respect to context-based semantics will need a novel synergy of both techniques for proving the completeness of {\bf BM} and graded modal logic. On the other hand, for weighted modal logic, we have shown that the straightforward generalizations of main axioms in {\bf BM} are not valid. Hence, the development of a sound and complete axiomatic system for two-sorted weighted modal logic is yet another challenging issue to be addressed in the future work.

\noindent {\bf Dynamics of context}: As a formal context consists of objects, properties, and a relation  between them, the relationship between objects and properties can change over time. Hence, to model and analyze the dynamics of contexts is also desirable. The logics presented so far are static as they can model only a snapshot of dynamically changing contexts at some time. By contrast, to  reason  about dynamics of a context, we need to take its temporal aspects into account. A possible way is to integrate temporal or dynamic logics\cite{tenselogic,dynamiclogic} with systems proposed in this paper. This will be also a direction for our future research.

\noindent {\bf Two-sorted frames beyond formal contexts}: In recent years, many different types of two-sorted frames have received much attention in the study of logic,  such as  {\it reduced and separating frames} \cite{CWFSPAPMTAWN}, {\it enriched formal contexts} \cite{Conradie2017167} and {\it LE-frames} \cite{Greco}. The Sequent Calculi for enriched formal contexts and  LE-frames have been given in \cite{Conradie2017167} and \cite{Greco}, respectively.  Following the work reported in this paper,  the development of Hilbert-style axiomatization for such frames is another direction deserving further exploration.

\noindent {\bf Many-valued contexts}: In many applications, the binary relation in formal contexts may be not two-valued. For example, by slightly changing the scenario introduced in our running example, suppose that the transaction data also contains the times of a customer purchasing an item.  We can reasonable assume that the more frequent a customer bought an item, the more he likes the product. For each item $m$, let $N_m$ be the maximum number of times that a single customer bought it. Then, if a customer $g$ has bought the item $n$ times, we can take the degree that $g$ likes $m$ as $\frac{n}{N_m}$. As a consequence, we have a many-valued context $(G,M,\tilde{I})$ where $\tilde{I}(g,m)\in[0,1]$ is the degree that $g$ likes $m$. Furthermore, we may want to express fuzzy information about the contexts, such as ``{\em Young \/} customers generally like the product item $m$'' or ``The customer $g$ likes {\em expensive \/} items''. Indeed, to deal with vague information in formal contexts, various fuzzy FCA methods have been proposed~\cite{Antoni2018,Belohlavek04,BelohlavekV05,Belohlavek11,BritoBEBC18}.

For the knowledge representation and reasoning of such information, we will need many-valued versions of different modal logics proposed in this paper. The basic idea of many-valued modal logic is simple. We simply have to replace the underlying two-valued logic with a many-valued one. However, when generalizing modal logics to the many-valued case, there are several subtle issues deserving special attention. First, in the many-valued case, box and diamond modalities are generally not mutually definable. Hence, when we consider the signature of a logic, the $\Box$-fragment, the $\Diamond$-fragment, and its full version are all different. Furthermore, the domain of truth values does matter. Therefore, for a logic introduced in the paper like {\bf KB}, there may be a plethora of  many-valued variants parameterized by their signatures and truth value domains. We expect that the investigation of these logics can yield fruitful results.

Second, because contraposition law usually does not hold for many-valued logic, we no longer have a correspondence between $\Box$ and $\boxminus$ modalities as in the case of {\bf KB} and {\bf KF}. Its bad implication is that we can not obtain a complete axiomatization of {\bf mv-KF} from a straightforward translation of that for {\bf mv-KB}. On the good side, this implies that we do not have to consider the interaction axiom between two mutually complemented modalities. Hence, finding axioms characterizing $\boxminus$ is another open problem for two-sorted many-valued modal logic. 

Finally, because possible worlds in many-valued models do not correspond to crisp subsets of formulas, we can no longer take maximally consistent subsets of formulas as possible worlds in the canonical model construction. To address this issue, we will follow the technique in the pioneering work ~\cite{fml11}, where a possible world in the canonical model is a non-modal homomorphism $h$ such that $h(\phi)=1$ for any theorem $\phi$ of the given logic. Then, the fuzzy relation corresponding to a modality is determined by its semantics accordingly. Implementing the proof technique for different two-sorted  many-valued modal logics will be an interesting technical challenge.
\begin{acks}
This work is partially supported by the National Science and Technology Council of Taiwan under Grants: NSTC 113-2221-E-001-018-MY3 and NSTC 113-2221-E-001-021-MY3. We are very grateful to the anonymous referees for their insightful comments and valuable suggestions.
\end{acks}


\begin{thebibliography}{10}
	
	\bibitem{Antoni2018}
	{\sc Antoni, L., Kraj{\v{c}}i, S., and Kr{\'i}dlo, O.}
	\newblock On fuzzy generalizations of concept lattices.
	\newblock In {\em Interactions Between Computational Intelligence and
		Mathematics}, L.~K{\'o}czy and J.~Medina, Eds. Springer, 2018, pp.~79--103.
	
	\bibitem{BALBIANI2012260}
	{\sc Balbiani, P.}
	\newblock Deciding the word problem in pure double {B}oolean algebras.
	\newblock {\em Journal of Applied Logic 10}, 3 (2012), 260 -- 273.
	
	\bibitem{Belohlavek04}
	{\sc Belohl{\'{a}}vek, R.}
	\newblock Concept lattices and order in fuzzy logic.
	\newblock {\em Annals of Pure and Applied Logic 128}, 1-3 (2004), 277--298.
	
	\bibitem{Belohlavek11}
	{\sc Belohl{\'{a}}vek, R.}
	\newblock What is a fuzzy concept lattice? {II}.
	\newblock In {\em Proceedings of 13th International Conference on Rough Sets,
		Fuzzy Sets, Data Mining and Granular Computing {RSFDGrC 2011}\/} (2011),
	S.~Kuznetsov, D.~Slezak, D.~Hepting, and B.~Mirkin, Eds., vol.~6743 of {\em
		Lecture Notes in Computer Science}, Springer, pp.~19--26.
	
	\bibitem{BelohlavekV05}
	{\sc Belohl{\'{a}}vek, R., and Vychodil, V.}
	\newblock What is a fuzzy concept lattice?
	\newblock In {\em Proceedings of the International Workshop on Concept Lattices
		and their Applications ({CLA})\/} (2005), R.~Belohl{\'{a}}vek and
	V.~Sn{\'{a}}sel, Eds., vol.~162 of {\em {CEUR} Workshop Proceedings},
	CEUR-WS.org.
	
	\bibitem{bimbo2001four}
	{\sc Bimb{\'o}, K., and Dunn, J.~M.}
	\newblock Four-valued logic.
	\newblock {\em Notre Dame Journal of Formal Logic 42}, 3 (2001), 171--192.
	
	\bibitem{birkhoff1940lattice}
	{\sc Birkhoff, G.}
	\newblock {\em Lattice theory}, vol.~25.
	\newblock American Mathematical Soc., 1940.
	
	\bibitem{blackburn2002moda}
	{\sc Blackburn, P., De~Rijke, M., and Venema, Y.}
	\newblock {\em Modal Logic}.
	\newblock Cambridge University Press, Cambridge, 2001.
	
	\bibitem{fml11}
	{\sc Bou, F., Esteva, F., Godo, L., and Rodr\'{\i}guez, R.}
	\newblock On the minimum many-valued modal logic over a finite residuated
	lattice.
	\newblock {\em Journal of Logic and Computation 21}, 5 (2011), 739--790.
	
	\bibitem{BritoBEBC18}
	{\sc Brito, A., Barros, L., Esmi, E., Bertato, F., and Coniglio, M.}
	\newblock Fuzzy formal concept analysis.
	\newblock In {\em Proceedings of the 37th Conference of the North American
		Fuzzy Information Processing Society, {NAFIPS}\/} (2018), G.~A. Barreto and
	R.~Coelho, Eds., vol.~831 of {\em Communications in Computer and Information
		Science}, Springer, pp.~192--205.
	
	\bibitem{Caro88}
	{\sc Caro, F.~D.}
	\newblock Graded modalities, {II} (canonical models).
	\newblock {\em Studia Logica 47}, 1 (1988), 1--10.
	
	\bibitem{Cerrato90}
	{\sc Cerrato, C.}
	\newblock General canonical models for graded normal logics (graded modalities
	{IV)}.
	\newblock {\em Studia Logica 49}, 2 (1990), 241--252.
	
	\bibitem{Cerrato94}
	{\sc Cerrato, C.}
	\newblock Decidability by filtrations for graded normal logics (graded
	modalities {V)}.
	\newblock {\em Studia Logica 53}, 1 (1994), 61--74.
	
	\bibitem{CONRADIE2021371}
	{\sc Conradie, W., Frittella, S., Manoorkar, K., Nazari, S., Palmigiano, A.,
		Tzimoulis, A., and Wijnberg, N.~M.}
	\newblock Rough concepts.
	\newblock {\em Information Sciences 561\/} (2021), 371--413.
	
	\bibitem{CWFSPAPMTAWN}
	{\sc Conradie, W., Frittella, S., Palmigiano, A., Piazzai, M., Tzimoulis, A.,
		and Wijnberg, N.~M.}
	\newblock Categories: How {I} {L}earned to {S}top {W}orrying and {L}ove {T}wo
	{S}orts.
	\newblock In {\em Logic, Language, Information, and Computation\/} (2016),
	J.~V{\"a}{\"a}n{\"a}nen and et~al., Eds., Springer, Berlin, pp.~145--164.
	
	\bibitem{Conradie2017167}
	{\sc Conradie, W., Frittella, S., Palmigiano, A., Piazzai, M., Tzimoulis, A.,
		and Wijnberg, N.~M.}
	\newblock Toward an epistemic-logical theory of categorization.
	\newblock {\em Electronic Proceedings in Theoretical Computer Science, EPTCS
		251\/} (2017), 167--186.
	
	\bibitem{CONRADIE2019923}
	{\sc Conradie, W., and Palmigiano, A.}
	\newblock Algorithmic correspondence and canonicity for non-distributive
	logics.
	\newblock {\em Annals of Pure and Applied Logic 170}, 9 (2019), 923--974.
	
	\bibitem{nlsm}
	{\sc Conradie, W., Palmigiano, A., Robinson, C., and Wijnberg, N.}
	\newblock Non-distributive logics: from semantics to meaning.
	\newblock In {\em Contemporary logic and computing}, vol.~1 of {\em Landsc.
		Log}. Coll. Publ., London, 2020, pp.~38--86.
	
	\bibitem{DautovicDO21}
	{\sc Dautovic, S., Doder, D., and Ognjanovic, Z.}
	\newblock An epistemic probabilistic logic with conditional probabilities.
	\newblock In {\em Proceedings of the 17th European Conference on Logics in
		Artificial Intelligence ({JELIA})\/} (2021), W.~Faber, G.~Friedrich,
	M.~Gebser, and M.~Morak, Eds., vol.~12678 of {\em Lecture Notes in Computer
		Science}, Springer, pp.~279--293.
	
	\bibitem{DautovicDO23}
	{\sc Dautovic, S., Doder, D., and Ognjanovic, Z.}
	\newblock Reasoning about knowledge and conditional probability.
	\newblock {\em International Journal of Approximate Reasoning 163\/} (2023),
	109037.
	
	\bibitem{DuboisSP07}
	{\sc Dubois, D., de~Saint{-}Cyr, F.~D., and Prade, H.}
	\newblock A possibility-theoretic view of formal concept analysis.
	\newblock {\em Fundamenta Informaticae 75}, 1-4 (2007), 195--213.
	
	\bibitem{duntsch2002modal}
	{\sc D{\"u}ntsch, I., and Gediga, G.}
	\newblock Modal-style operators in qualitative data analysis.
	\newblock In {\em Proceedings of the 2002 {IEEE} International Conference on
		Data Mining\/} (2002), K.~Vipin et~al., Eds., {IEEE} Computer Society,
	pp.~155--162.
	
	\bibitem{Fattorosi-Barnaba88}
	{\sc Fattorosi{-}Barnaba, M., and Cerrato, C.}
	\newblock Graded modalities. {III} (the completeness and compactness of
	{S4$^0$)}.
	\newblock {\em Studia Logica 47}, 2 (1988), 99--110.
	
	\bibitem{DeCaro}
	{\sc Fattorosi-Barnaba, M., and De~Caro, F.}
	\newblock Graded modalities. {I}.
	\newblock {\em Studia Logica 44\/} (1985), 197--221.
	
	\bibitem{Fattorosi-BarnabaG95}
	{\sc Fattorosi{-}Barnaba, M., and Grassotti, S.}
	\newblock An infinitary graded modal logic (graded modalities {VI)}.
	\newblock {\em Math. Log. Q. 41\/} (1995), 547--563.
	
	\bibitem{BalestriniF99}
	{\sc Fattorosi{-}Barnaba, M., and U.P.Balestrini}.
	\newblock The modality of finite (graded modalties {VII)}.
	\newblock {\em Math. Log. Q. 45\/} (1999), 471--480.
	
	\bibitem{tenselogic}
	{\sc Gabbay, D.~M., Hogger, C.~J., and Robinson, J.~A.}, Eds.
	\newblock {\em Handbook of logic in artificial intelligence and logic
		programming (Vol. 4): epistemic and temporal reasoning}.
	\newblock Oxford University Press, Inc., 1995.
	
	\bibitem{Gargov1987}
	{\sc Gargov, G., Passy, S., and Tinchev, T.}
	\newblock Modal environment for boolean speculations.
	\newblock In {\em Mathematical Logic and Its Applications\/} (Boston, MA,
	1987), D.~G. Skordev, Ed., Springer US, pp.~253--263.
	
	\bibitem{Gehrke}
	{\sc Gehrke, M.}
	\newblock Generalized {K}ripke frames.
	\newblock {\em Studia Logica 84}, 2 (2006), 241--275.
	
	\bibitem{Greco}
	{\sc Greco, G., Jipsen, P., Liang, F., Palmigiano, A., and Tzimoulis, A.}
	\newblock Algebraic proof theory for {$\rm LE$}-logics.
	\newblock {\em ACM Trans. Comput. Log. 25}, 1 (2024), Art. 6, 37.
	
	\bibitem{hal05}
	{\sc Halpern, J.}
	\newblock {\em Reasoning about uncertainty}.
	\newblock {MIT} Press, 2005.
	
	\bibitem{dynamiclogic}
	{\sc Harel, D., Tiuryn, J., and Kozen, D.}
	\newblock {\em Dynamic Logic}.
	\newblock MIT Press, 2000.
	
	\bibitem{HCgame}
	{\sc Hartonas, C.}
	\newblock {Game-theoretic semantics for non-distributive logics}.
	\newblock {\em Logic Journal of the IGPL 27}, 5 (2019), 718--742.
	
	\bibitem{CH}
	{\sc Hartonas, C.}
	\newblock Lattice logic as a fragment of (2-sorted) residuated modal logic.
	\newblock {\em Journal of Applied Non-Classical Logics 29}, 2 (2019), 152--170.
	
	\bibitem{howlader2018algebras}
	{\sc Howlader, P., and Banerjee, M.}
	\newblock Algebras from semiconcepts in rough set theory.
	\newblock In {\em International Joint Conference on Rough Sets\/} (2018), H.~S.
	Nguyen et~al., Eds., Springer, Cham, pp.~440--454.
	
	\bibitem{howlader2020}
	{\sc Howlader, P., and Banerjee, M.}
	\newblock Object oriented protoconcepts and logics for double and pure double
	{B}oolean algebras.
	\newblock In {\em International Joint Conference on Rough Sets\/} (2020),
	R.~Bello et~al., Eds., Springer, Cham, pp.~308--323.
	
	\bibitem{HOWLADER2023115}
	{\sc Howlader, P., and Banerjee, M.}
	\newblock A non-distributive logic for semiconcepts and its modal extension
	with semantics based on {K}ripke contexts.
	\newblock {\em International Journal of Approximate Reasoning 153\/} (2023),
	115--143.
	
	\bibitem{MR4566932}
	{\sc Howlader, P., and Banerjee, M.}
	\newblock Topological representation of double {B}oolean algebras.
	\newblock {\em Algebra Universalis 84\/} (2023), Paper No. 15, 32.
	
	\bibitem{HowladerL23}
	{\sc Howlader, P., and Liau, C.}
	\newblock Two-sorted modal logic for formal and rough concepts.
	\newblock In {\em - International Joint Conference on Rough Sets ({IJCRS})\/}
	(2023), A.~Campagner, O.~U. Lenz, S.~Xia, D.~Slezak, J.~Was, and J.~Yao,
	Eds., vol.~14481 of {\em LNCS}, Springer, pp.~154--169.
	
	\bibitem{LEGASTELOIS2017341}
	{\sc Legastelois, B., Lesot, M.-J., and d'Allonnes, A.~R.}
	\newblock Typology of axioms for a weighted modal logic.
	\newblock {\em International Journal of Approximate Reasoning 90\/} (2017),
	341--358.
	
	\bibitem{mspml}
	{\sc Leu\c{s}tean, I., Moang\u{a}, N., and \c{S}erb\u{a}nu\c{t}\u{a}, T.~F.}
	\newblock A many-sorted polyadic modal logic.
	\newblock {\em Fundamenta Informaticae 173}, 2-3 (2020), 191--215.
	
	\bibitem{qml1}
	{\sc Liau, C., and Lin, I.}
	\newblock Quantitative modal logic and possibilistic reasoning.
	\newblock In {\em Proceedings of the 10th ECAI\/} (1992), B.~Neumann, Ed., John
	Wiley \& Sons. Ltd, pp.~43--47.
	
	\bibitem{qml2}
	{\sc Liau, C., and Lin, I.}
	\newblock Proof methods for reasoning about possibility and necessity.
	\newblock {\em International Journal of Approximate Reasoning 9}, 4 (1993),
	327--364.
	
	\bibitem{qml3}
	{\sc Liau, C., and Lin, I.}
	\newblock Reasoning about higher order uncertainty in possibilistic logic.
	\newblock In {\em Proceedings of 7th International Symposium on Methodologies
		for Intelligent Systems\/} (1993), J.~Komorowski and Z.~Ra\'{s}, Eds., LNAI
	689, Springer-Verlag, pp.~316--325.
	
	\bibitem{LiauL96}
	{\sc Liau, C., and Lin, I.}
	\newblock Possibilistic reasoning - {A} mini-survey and uniform semantics.
	\newblock {\em Artificial Intelligence 88}, 1-2 (1996), 163--193.
	
	\bibitem{OgnjanovicRM16}
	{\sc Ognjanovic, Z., Raskovic, M., and Markovic, Z.}
	\newblock {\em Probability Logics - Probability-Based Formalization of
		Uncertain Reasoning}.
	\newblock Springer, 2016.
	
	\bibitem{pawlak1982rough}
	{\sc Pawlak, Z.}
	\newblock Rough sets.
	\newblock {\em International Journal of Computer and Information Sciences 11},
	5 (1982), 341--356.
	
	\bibitem{vormbrock}
	{\sc Vormbrock, B.}
	\newblock A solution of the word problem for free double {B}oolean algebras.
	\newblock In {\em Formal Concept Analysis\/} (2007), S.~O. Kuznetsov and
	et~al., Eds., Springer, Berlin, pp.~240--270.
	
	\bibitem{vormbrock2005semiconcept}
	{\sc Vormbrock, B., and Wille, R.}
	\newblock Semiconcept and protoconcept algebras: the basic theorems.
	\newblock In {\em Formal Concept Analysis: Foundations and Applications\/}
	(2005), B.~Ganter et~al., Eds., Springer, Berlin, pp.~34--48.
	
	\bibitem{FCA}
	{\sc Wille, R.}
	\newblock Restructuring lattice theory: an approach based on hierarchies of
	concepts.
	\newblock In {\em Ordered sets ({B}anff, {A}lta., 1981)}, vol.~83. Reidel,
	Dordrecht-Boston, Mass., 1982, pp.~445--470.
	
	\bibitem{wille}
	{\sc Wille, R.}
	\newblock Boolean concept logic.
	\newblock In {\em Conceptual Structures: Logical, Linguistic, and Computational
		Issues\/} (2000), B.~Ganter et~al., Eds., Springer, Berlin, pp.~317--331.
	
	\bibitem{yao2004comparative}
	{\sc Yao, Y.~Y.}
	\newblock A comparative study of formal concept analysis and rough set theory
	in data analysis.
	\newblock In {\em International Conference on Rough Sets and Current Trends in
		Computing\/} (2004), S.~Tsumoto et~al., Eds., Springer, Berlin, pp.~59--68.
	
	\bibitem{yao2004concept}
	{\sc Yao, Y.~Y.}
	\newblock Concept lattices in rough set theory.
	\newblock In {\em IEEE Annual Meeting of the Fuzzy Information Processing
		Society-NAFIPS\/} (2004), vol.~2, IEEE, pp.~796--801.
	
\end{thebibliography}
\end{document}